\documentclass[11pt, letterpaper]{amsart}

\usepackage[T1]{fontenc}
\usepackage[utf8]{inputenc}
\usepackage{lmodern}
\usepackage{stmaryrd}

\newcommand{\Dod}{\mathrm{Dod}}

\newcommand{\sphere}[1]{{S}^{#1}}

\newcommand{\RP}[1]{\mathbb{RP}^{#1}}
\newcommand{\CP}[1]{\mathbb{CP}^{#1}}

\newcommand{\ZZ}{\mathbb{Z}}

\newcommand{\Twist}{\tau}

\usepackage[american]{babel}
\usepackage[babel, final]{microtype}
\usepackage{amssymb}
\usepackage[all]{xy}

\usepackage{tabularx}

\usepackage{mathtools}

\usepackage{ifpdf}
\usepackage{comment}
\usepackage{multirow}

\usepackage[pdftex, dvipsnames]{xcolor}
\usepackage[pdftex, final]{graphicx}
\usepackage{caption}
\usepackage{pinlabel}
\usepackage[pdftex,%
    includehead,%
    includefoot,%
    nomarginpar,%
    lmargin=1in,%
    rmargin=1in,%
    tmargin=1in,%
    bmargin=1in,%
]{geometry}
\usepackage[pdftex,%
    final,%
    colorlinks=true,%
    linkcolor=NavyBlue,%
    citecolor=NavyBlue,%
    filecolor=NavyBlue,%
    menucolor=NavyBlue,%
    urlcolor={blue!45!black},%
    bookmarks=true,%
    bookmarksdepth=3,%
    bookmarksnumbered=true,%
    bookmarksopen=true,%
    bookmarksopenlevel=2,%
]{hyperref}
\hypersetup{
    pdftitle={A group-theoretic framework for low-dimensional topology},
    pdfauthor={Blackwell, Kirby, Klug, Longo, Ruppik},
    pdfsubject={A group-theoretic framework for low-dimensional topology},
    pdfkeywords={Splitting homomorphisms,
    group trisections,
    tangles in handlebodies,
    links in 3-manifolds,
    knotted surfaces in 4-manifolds,
    surface groups,
    free groups,
    Stallings folding.}
}
\usepackage{cleveref}
\usepackage{float}
\usepackage{xfrac}

\definecolor{myGreen}{rgb}{0.0, 0.5, 0.0}
\definecolor{myGraphRed}{RGB}{255, 85, 85}
\definecolor{myGraphGreen}{RGB}{90, 160, 44}

\definecolor{maroon(html/css)}{rgb}{0.5, 0.0, 0.0}
\definecolor{aqua}{rgb}{0.0, 1.0, 1.0}

\usepackage[hang,flushmargin]{footmisc}

\makeatletter
\def\@tocline#1#2#3#4#5#6#7{\relax
  \ifnum #1>\c@tocdepth 
  \else
    \par \addpenalty\@secpenalty\addvspace{#2}%
    \begingroup \hyphenpenalty\@M
    \@ifempty{#4}{%
      \@tempdima\csname r@tocindent\number#1\endcsname\relax
    }{%
      \@tempdima#4\relax
    }%
    \parindent\z@ \leftskip#3\relax \advance\leftskip\@tempdima\relax
    \rightskip\@pnumwidth plus4em \parfillskip-\@pnumwidth
    #5\leavevmode\hskip-\@tempdima
      \ifcase #1
       \or\or \hskip 1em \or \hskip 2em \else \hskip 3em \fi%
      #6\nobreak\relax
    \dotfill\hbox to\@pnumwidth{\@tocpagenum{#7}}\par
    \nobreak
    \endgroup
  \fi}
\makeatother

\usepackage{csquotes}
\usepackage[
	backend=biber,
	style=alphabetic,
	sorting=nyt,
	giveninits=true,
	maxnames=10,
	url=false,
	doi=false,
	isbn=false,
	backref=true
]{biblatex}
\DefineBibliographyStrings{english}{
  backrefpage={$\uparrow$},
  backrefpages={$\uparrow$}
}

\DeclareFieldFormat[article,incollection,inproceedings]{title}{#1}
\DeclareFieldFormat[unpublished]{title}{\textit{#1}}

\usepackage{tikz}
\usetikzlibrary{tikzmark}
\usepackage{tikz-cd}
\usetikzlibrary{cd}
\usetikzlibrary{backgrounds}

\makeatletter
\g@addto@macro\@floatboxreset\centering
\makeatother

\allowdisplaybreaks[4]

\addto\extrasamerican{%
}

\newtheorem{theorem}{Theorem}[section]
\newtheorem{conjecture}{Conjecture}[section]
\newtheorem{corollary}{Corollary}[section]

\newtheorem{lemma}{Lemma}[section]
\newtheorem*{ttheorem}{Topological input theorem}
\theoremstyle{definition}

\newtheorem{definition}{Definition}[section]
\newtheorem{example}{Example}[section]
\newtheorem{remark}{Remark}[section]
\newtheorem{question}{Question}[section]
\newtheorem{claim}{Claim}
\newtheorem*{*claim}{Claim}

\makeatletter
\let\c@conjecture=\c@theorem
\let\c@corollary=\c@theorem
\let\c@proposition=\c@theorem
\let\c@lemma=\c@theorem
\let\c@definition=\c@theorem
\let\c@example=\c@theorem
\let\c@remark=\c@theorem
\let\c@equation=\c@theorem
\let\c@question=\c@theorem
\makeatother

\def\makeautorefname#1#2{\expandafter\def\csname#1autorefname\endcsname{#2}}
\makeautorefname{theorem}{Theorem}%
\makeautorefname{conjecture}{Conjecture}%
\makeautorefname{corollary}{Corollary}%
\makeautorefname{proposition}{Proposition}%
\makeautorefname{lemma}{Lemma}%
\makeautorefname{definition}{Definition}%
\makeautorefname{example}{Example}%
\makeautorefname{remark}{Remark}%
\makeautorefname{question}{Question}%
\makeautorefname{claim}{Claim}%

\numberwithin{equation}{section}

\newcommand*{\Alphabet}{ABCDEFGHIJKLMNOPQRSTUVWXYZ1234567890}
\newcommand*{\alphabet}{abcdefghijklmnopqrstuvwxyz1234567890}
\newlength\fcaph
\newlength\fdesc
\newlength\factualfontsize
\newcounter{commentcounter}

\newcommand{\Alg}{\texttt{\textup{Alg}}}
\newcommand{\Man}{\texttt{\textup{Man}}}
\newcommand{\Diag}{\texttt{\textup{Diag}}}

\title[A group-theoretic framework for low-dimensional topology]{ \small
A group-theoretic framework for low-dimensional topology\\
\footnotesize Or, how not to study low-dimensional topology?$^{\ast}$
}

\thanks{$^{\ast}$We were inspired by Stallings' paper ``How not to prove the Poincar\'{e} conjecture'' \cite{stallings1966nottoprove} in more ways than one.}

\author{Sarah Blackwell}
\address{University of Virginia, Charlottesville, United States}
\email{\href{mailto:blackwell@virginia.edu}{blackwell@virginia.edu}}
\urladdr{\url{https://seblackwell.com}}

\author{Robion Kirby}
\address{University of California, Berkeley, United States}
\email{\href{mailto:kirby@math.berkeley.edu}{kirby@math.berkeley.edu}}
\urladdr{\url{https://math.berkeley.edu/~kirby}}

\author{Michael Klug}
\address{University of Chicago, Chicago, United States}
\email{\href{mailto:michael.r.klug@gmail.com }{michael.r.klug@gmail.com }}
\urladdr{\url{https://mathematics.uchicago.edu/people/profile/michael-klug}}

\author{Vincent Longo}
\address{University of Connecticut, Storrs, United States}
\email{\href{mailto:vincent.2.longo@uconn.edu}{vincent.2.longo@uconn.edu}}
\urladdr{\url{https://math.uconn.edu/person/vincent-longo}}

\author{Benjamin Ruppik}
\address{Heinrich-Heine-Universität, Düsseldorf, Germany}
\email{\href{mailto:ruppik@hhu.de}{benjamin.ruppik@hhu.de}}
\urladdr{\url{https://bruppik.de}}

\def\subjclassname{\textup{2020} Mathematics Subject Classification}
\expandafter\let\csname subjclassname@1991\endcsname=\subjclassname
\expandafter\let\csname subjclassname@2000\endcsname=\subjclassname
\subjclass{
    57K40, 
    57M05, 
    20F05 
    \hfill
    Date: \today
}

\begin{document}

\begin{abstract}
A correspondence, by way of Heegaard splittings, between closed oriented 3-manifolds and pairs of surjections from a surface group to a free group has been studied by Stallings, Jaco, and Hempel. This correspondence, by way of trisections, was recently extended by Abrams, Gay, and Kirby to the case of smooth, closed, connected, oriented 4-manifolds. 
We unify these perspectives and generalize this correspondence to the case of links in closed oriented 3-manifolds and links of knotted surfaces in smooth, closed, connected, oriented 4-manifolds. 
The algebraic manifestations of these four subfields of low-dimensional topology (3-manifolds, 4-manifolds, knot theory, and knotted surface theory) are all strikingly similar, and this correspondence perhaps elucidates some unique character of low-dimensional topology.
\end{abstract}

\maketitle

\section{Introduction}
\label{sec:intro}

All manifolds and submanifolds discussed in this paper are smooth and, with the exception of surfaces in $4$-manifolds, oriented.  In this paper, we use decompositions of manifolds in dimensions three and four, possibly together with links, to give a group-theoretic framework for studying these spaces.  We begin by reviewing the simplest case of closed 3-dimensional manifolds, where this work has already been carried out by Stallings and Jaco \cite{stallings1966nottoprove,jaco1970stable}. 

A \textit{Heegaard decomposition}, or \textit{Heegaard splitting}, of a closed 3-manifold $M^3$ is a pair of handlebodies $H_1$ and $H_2$ embedded inside of $M$ with boundaries a common genus $g$ surface $\Sigma_g$ such that $M = H_1 \cup_{\Sigma_g} H_2$.  Every such 3-manifold admits a Heegaard decomposition (for example by triangulating $M$ and taking a regular neighborhood of the 1-skeleton).  By choosing a basepoint on $\Sigma_g$, we then obtain the following pushout diagram between fundamental groups, where the maps are induced by inclusion.
\begin{center}
\begin{tikzcd}[row sep=scriptsize, column sep=scriptsize]
        \pi_1(\Sigma_g, \ast)
        \arrow[rr, twoheadrightarrow] \arrow[dd, twoheadrightarrow]
        & &
        \pi_1(H_1,\ast) \arrow[dd, twoheadrightarrow]
        \\ \\
        \pi_1(H_2,\ast)
        \arrow[rr, twoheadrightarrow]
        & &
        \pi_{1}(M,\ast)
\end{tikzcd}
\end{center}
Note that $\pi_1(H_1, \ast)$ and $\pi_1(H_2, \ast)$ are both free groups of rank $g$, and the maps are surjections. Jaco proved that given a surjective homomorphism $\phi \colon \pi_1(\Sigma_g, \ast) \twoheadrightarrow F_g$, there is a unique handlebody $H(\phi)$ with $\partial H(\phi) = \Sigma_g$ such that the map induced on $\pi_1$ by inclusion of $\Sigma_g$ as the boundary agrees with $\phi$ (see Jaco \cite{jaco1969splitting}, \autoref{lem:uniqueness}, and also Leininger and Reid \cite[Lem.~2.2]{leininger2002corank} for a simpler proof in this case).  From this, it follows that a pair of surjective homomorphisms $(\phi_1,\phi_2)$ with $\phi_i \colon \pi_1(\Sigma_g, \ast) \twoheadrightarrow F_g$ determines a 3-manifold $H(\phi_1) \cup_{\Sigma_g} H(\phi_2)$, and that every closed 3-manifold arises in this way. Jaco referred to these pairs of maps as \textit{splitting homomorphisms}.

One concrete application of this is the following group-theoretic recasting of the 3-dimensional Poincar\'{e} conjecture.  Writing $\pi_1(\Sigma_{g}, \ast) = \langle a_1, b_1, \ldots, a_g, b_g \mid [a_1,b_1] \cdots [a_g,b_g] = 1 \rangle$, there is a surjective homomorphism 
    \begin{align*}
        \pi_1(\Sigma_g, \ast) &\twoheadrightarrow \langle x_1, \ldots x_g \rangle \times \langle y_1, \ldots y_g \rangle \\
        a_i                 &\mapsto (x_i,1) \\
        b_j                 &\mapsto (1,y_j).
    \end{align*}
The Poincar\'{e} conjecture is equivalent to the statement that this is the unique surjective homomorphism of these groups modulo pre-composing with automorphisms and post-composing with products of automorphisms (see Hempel's \textit{3-Manifolds} \cite{hempelBook}).  Thus by Perelman's work \cite{perelman2003finite} this result follows, and we are left in the state where the only known proof of this perhaps innocent-looking group-theoretic result involves a careful analysis of Ricci flow.  

In addition to the observation that every closed 3-manifold admits a Heegaard decomposition, there is a corresponding uniqueness theorem called the \textit{Reidemeister-Singer theorem}, which states that any two Heegaard decompositions of a fixed 3-manifold differ by a sequence of simple inverse geometric operations called stabilization and destabilization \cite{reidemeister1933topologie,singer1933heegaard}.  Jaco proposed a way of incorporating the Reidmeister-Singer theorem into the construction of 3-manifolds from appropriate pairs $(\phi_1, \phi_2)$ to obtain a bijective correspondence \cite{jaco1970stable}.  

More recently, a 4-dimensional analogue of Heegaard splittings, called \textit{trisections}, together with a corresponding uniqueness theorem has been introduced by Gay and Kirby \cite{gay2016trisecting}. A trisection of a closed 4-manifold $X^4$ is a decomposition $X=X_1 \cup X_2 \cup X_3$ into 4-dimensional 1-handlebodies $X_i$, which pairwise intersect in genus $g$ handlebodies $H_g$, and with triple intersection a genus $g$ surface $\Sigma_g$. Every smooth, closed, connected, oriented $4$-manifold admits a trisection, which is unique up to a stabilization operation \cite{gay2016trisecting}. (See \autoref{sec:closed_4-manifolds} for a further review of trisections.) 
The inclusion maps between the various components of a trisection of a $4$-manifold induce maps between their fundamental groups, which produces the following commutative diagram, where every face is a pushout and every homomorphism is surjective. The basepoint is chosen to lie on $\Sigma_g$.
\begin{center}
  \begin{tikzcd}[row sep=scriptsize, column sep=scriptsize]
        & 
        \pi_1(H_g,\ast)
        \arrow[rr, twoheadrightarrow] \arrow[dd, twoheadrightarrow] & &
        \pi_{1}(X_1, \ast)
        \arrow[dd, twoheadrightarrow] \\
        \pi_1(\Sigma_g, \ast)
        \arrow[ur, twoheadrightarrow] \arrow[rr, crossing over, twoheadrightarrow] \arrow[dd, twoheadrightarrow]
        & &
        \pi_1(H_g,\ast)
        \arrow[ur, twoheadrightarrow]
        \\
        & 
        \pi_{1}(X_3,\ast)
        \arrow[rr, twoheadrightarrow] & &
        \pi_{1}(X^4,\ast) \\
        \pi_1(H_g,\ast)
        \arrow[rr, twoheadrightarrow] \arrow[ur, twoheadrightarrow]
        & &
        \pi_{1}(X_2,\ast)
        \arrow[from=uu, crossing over, twoheadrightarrow] \arrow[ur, twoheadrightarrow] 
    \end{tikzcd}
\end{center}

In \cite{abrams2018group}, Abrams, Gay, and Kirby noticed that the analogue of being able to recover a 3-manifold from a pair of surjective homomorphisms $(\phi_1,\phi_2)$ holds in dimension four via trisections.  Namely, given three surjective homomorphisms $(\phi_1,\phi_2,\phi_3)$ with $\phi_i \colon \pi_1(\Sigma_g, \ast) \twoheadrightarrow F_g$ such that the pairwise pushout of any pair $\phi_i$ and $\phi_j$ is a free group $F_k$, then since $\#^k(S^1 \times S^2)$ is the unique closed, orientable 3-manifold with fundamental group $F_k$ (by Perelman's work \cite{perelman2003finite}), we obtain a closed 4-manifold by realizing three handlebodies $H(\phi_i)$, gluing them along their common boundary $\Sigma_g$, and filling in their pairwise unions, which are diffeomorphic to $\#^k(S^1 \times S^2)$, with three 4-dimensional 1-handlebodies (uniquely by Laudenbach and Po{\'e}naru \cite{laudenbach1972note}). They called these triple of maps (which then determine the entire cube pictured above) a \textit{group trisection}, where the object being trisected is the group resulting from pushing out the three maps into a cube (in this case, $ \pi_{1}(X^4,\ast)$).

Additionally in \cite{abrams2018group}, Abrams, Gay, and Kirby use the uniqueness theorem for tisections to obtain results analogous to those previously mentioned in dimension three.  Namely, they obtain a group-theoretic statement that is equivalent to the smooth 4-dimensional Poincar\'{e} conjecture and, by modding out the set of such triples $(\phi_1,\phi_2, \phi_3)$, they obtain a bijection between a group-theoretically defined set and the set of all smooth, closed, connected, oriented 4-manifolds. 

Not only can every 3-manifold be split into a union of two handlebodies, but additionally, given a link $L \subset M$ we have a Heegaard splitting $M = H_1 \cup_{\Sigma_g} H_2$ such that the tangles $T_1 = L \cap H_1$ and $T_2 = L \cap H_2$ are trivial (that is, consist of arcs that can all be simultaneously isotoped in $H_i$ into $\Sigma_g$).  This is called a \textit{bridge splitting} of $L \subset M$. Note that the complement of $L$ in each handlebody is again a handlebody and hence has free fundamental group.  In the case of $M = S^3$ with the Heegaard splitting into balls, this is the classical setting of \textit{bridge position} of links (see \cite{schubert1954numerische}).

One dimension up, a similar story emerges. 
A \textit{knotted surface} is a closed (potentially non-orientable or disconnected) surface smoothly embedded in a $4$-manifold. Meier and Zupan showed that given a knotted surface in a trisected 4-manifold, it can always be isotoped to be in \textit{bridge position}, meaning that it intersects the trisected 4-manifold in such a way that the surface inherits its own trisection, called a \textit{bridge trisection} \cite{meier2017bridgeS4,meier2018bridge4manifolds}. This is unique up to a stabilization operation \cite{meier2017bridgeS4,hughes2018isotopies}. (See \autoref{sec:links_4-manifolds} for a further review of bridge trisections.)
Given the existence and uniqueness of such a decomposition in this setting, it is natural to wonder whether knotted surfaces in 4-manifolds can also be given such a group-theoretic framework. Achieving this goal was the initial motivation for this work.

The main results of this paper are bijective correspondences from group-theoretic sets to the set of 3-manifolds together with a link and 4-manifolds together with a knotted surface, and are summarized in \autoref{fig:chart}.  Just as the cases of 3-manifolds and 4-manifolds are facilitated by Heegaard splittings and trisections, respectively, our result for links in 3-manifolds and surfaces in 4-manifolds use bridge splittings and bridge trisections, respectively.

\Large 
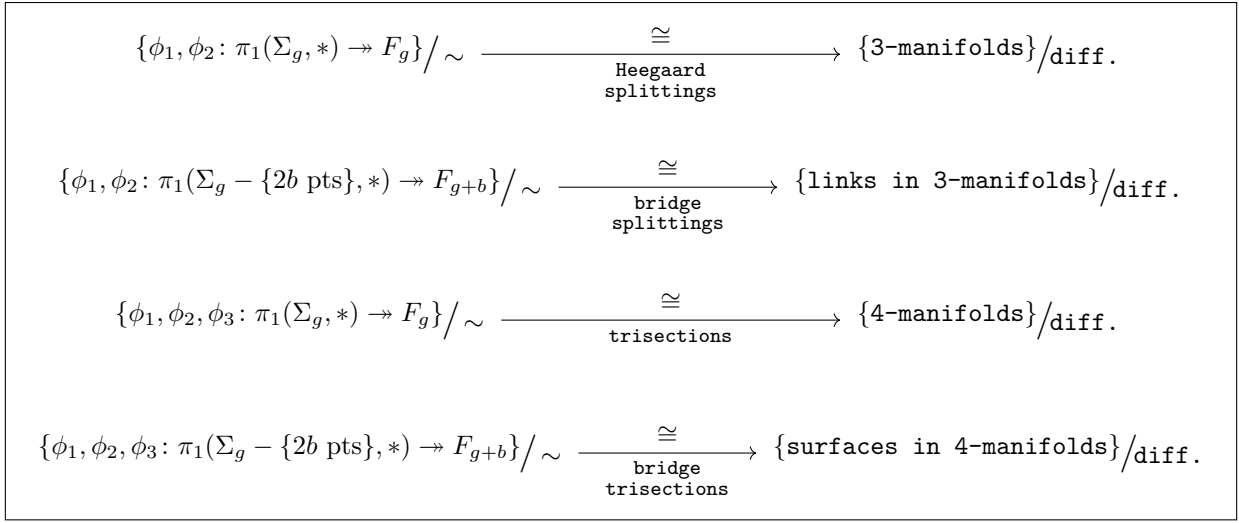
\begin{figure}
\begin{tikzcd}[framed]
    \sfrac{\{ \phi_1,\phi_2 \colon \pi_1(\Sigma_g, \ast) \twoheadrightarrow F_g \}}{\sim}
    \arrow{rr}{\cong}[swap]{ \normalsize \substack{\texttt{\textup{Heegaard}} \\ \texttt{\textup{splittings}}}}
    & &  \sfrac{\{ \texttt{\textup{3-manifolds}} \}}{\texttt{\textup{diff.}}} \\
    \sfrac{\{ \phi_1,\phi_2 \colon \pi_1(\Sigma_g - \{2b \text{ pts}\}, \ast) \twoheadrightarrow F_{g+b} \}}{\sim}
    \arrow{rr}{\cong}[swap]{ \normalsize \substack{\texttt{\textup{bridge}} \\ \texttt{\textup{splittings}}}}
    & & \sfrac{\{ \texttt{\textup{links in 3-manifolds}} \}}{\texttt{\textup{diff.}}} \\
    \sfrac{\{ \phi_1,\phi_2,\phi_3 \colon \pi_1(\Sigma_g, \ast) \twoheadrightarrow F_g \}}{\sim}
    \arrow{rr}{\cong}[swap]{\normalsize \substack{\texttt{\textup{trisections}}}}
    & & \sfrac{\{ \texttt{\textup{4-manifolds}} \}}{\texttt{\textup{diff.}}} \\ 
    \sfrac{\{ \phi_1,\phi_2,\phi_3 \colon \pi_1(\Sigma_g - \{2b \text{ pts}\}, \ast) \twoheadrightarrow F_{g+b} \}}{\sim}
    \arrow{rr}{\cong}[swap]{\normalsize  \substack{\texttt{\textup{bridge}} \\ \texttt{\textup{trisections}}}}
    & & \sfrac{\{ \texttt{\textup{surfaces in 4-manifolds}} \}}{\texttt{\textup{diff.}}}
\end{tikzcd}
    \caption{A summary of the main results of this paper; the equivalences via bridge splittings and bridge trisections are novel, while the other equivalences are a recasting and unification of previous work. All manifolds are smooth and, with the exception of surfaces in $4$-manifolds, oriented, and the diffeomorphisms are orientation preserving. All maps $\phi_i$ shown are surjective homomorphisms, and satisfy two algebraic conditions (see \autoref{def:bounding_hom}). The homomorphisms $\phi_i$ in the third and fourth rows need to satisfy the additional condition that they push out pairwise to free groups of an appropriate rank (see \autoref{sec:closed_4-manifolds} and \autoref{sec:links_4-manifolds}).
    \label{fig:chart}}
\end{figure}
\normalsize

In order to get off the ground constructing these spaces from appropriate group homomorphisms, we need to know how to recover a trivial tangle $T(\phi)$ with boundary points $\{ p_1, \ldots, p_{2b}\}$ in a handlebody $H(\phi)$ with boundary $\Sigma_g$ from a suitable homormophism \[\phi \colon \pi_1(\Sigma_g - \{p_1, \ldots, p_{2b} \}, \ast) \twoheadrightarrow F_{g+b},\] where $F_{g+b}$ is playing the role of the fundamental group of the complement of the trivial tangle. (For the precise algebraic conditions on $\phi$ needed for this construction, we defer to \autoref{def:bounding_hom}.) \autoref{sec:setup} is dedicated to this task. Our first main result (see \autoref{thm:main}), and the result underlying all of the constructions of spaces from group homomorphisms, is a method for constructing $H(\phi)$ and $T(\phi)$ algorithmically. This method is inspired by the procedure for computing the corresponding homomorphism given the topological data of the surface together with curves indicating the handlebody and the trivial tangle (see \autoref{lem:main}). When naively trying to construct diagrams for $H(\phi)$ and $T(\phi)$, we run into the possibility of constructing diagrams with too many curves. We fix this using bands to connect curves together, where the combinatorics of how the bands connect curves is guided by a process called Stallings folding \cite{stallings1983topology}, whose behaviour is guaranteed to serve our purposes by the conditions placed on $\phi$ (see the proof of \autoref{thm:main}).   

With this construction in hand, the constructions of the maps in \autoref{fig:chart} are straightforward and surjectivity follows from existence of the various geometric decompositions.  In \autoref{sec:3-manifolds} and \autoref{sec:4-manifolds} we discuss in detail the various geometric descriptions, the map in \autoref{fig:chart}, and the various algebraically defined relations that need to be collectively modded out by on the set of homomorphisms in order to obtain a bijection.  This later part involves setting up appropriate relations on the set of homomorphisms that mimic the geometric moves needed in the corresponding uniqueness theorem.  

For example, in the case of closed 3-manifolds, by way of Heegaard splittings, all such 3-manifolds can be described by a \textit{Heegaard diagram}, and two Heegaard diagrams result in the same 3-manifold if and only if they are related by a sequence of handleslides, diffeomorphisms of the surface applied to the diagram, and stabilizations (this is a diagrammatic restating of the Reidemeister-Singer theorem; see \autoref{thm:Alg_3_to_Man_3}).  In this case, we need to mod out our set of homomorphisms by an equivalence relation generated by three relations $\sim_h$, $\sim_m$, and $\sim_s$ ($h$ for handleslide, $m$ for mapping class, and $s$ for stabilization) that algebraically mimic the corresponding diagrammatic moves.  In \autoref{sec:closed_3-manifolds}, we carry out this process for closed 3-manifolds and in \autoref{sec:links_3-manifolds}, \autoref{sec:closed_4-manifolds}, and \autoref{sec:links_4-manifolds} we do the analogous procedure for links in 3-manifolds, closed 4-manifolds, and surfaces in 4-manifolds, respectively. In the ``relative'' cases of \autoref{sec:links_3-manifolds} and \autoref{sec:links_4-manifolds} there are additional relations needed corresponding to the different types of stabilizations available in these settings.   

It is unclear if this formalism will prove useful in deriving topological results (see subtitle). However, in \autoref{sec:examples} we give some additional examples and pose some questions regarding potential applications. One curious consequence of our work is that although smoothly knotted surfaces in the $4$-sphere cannot be distinguished by fundamental groups (or even their complements), they \emph{can} be distinguished by group trisections (see \autoref{cor:distinguish}).

\subsection*{Acknowledgments}

This project started at the 2020
Virtual Summer Trisectors Workshop,
which was financially supported by the NSF.
We are grateful for this wonderful opportunity to exchange ideas,
and would like to thank all of the participants for their input.
We want to give special credit
to Michelle Chu, David Gay, Gabriel Islambouli, Jason Joseph, Chris Leininger, 
Jeffrey Meier, Puttipong Pongtanapaisan, Arunima Ray, and Alexander Zupan. We also would like to sincerely thank the anonymous referees for their careful reading of our paper and their insightful comments. SB, MK, and BR would additionally like to thank the
Max Planck Institute for Mathematics in Bonn for its hospitality and financial support.

\tableofcontents

\section{Trivial tangles in handlebodies from algebra}
\label{sec:setup}
 
Let $\Sigma_g$ denote the genus $g$ oriented surface with basepoint $\ast$ and marked points $p_1, \ldots, p_{2b}$ as in \autoref{fig:surface}.
We abuse notation and let $p_1, \ldots, p_{2b}$ also denote the fundamental group elements as pictured.
Using the notation $[a,b] = aba^{-1}b^{-1}$ for the commutator of $a$ and $b$, we have
\[
    \pi_1(\Sigma_{g} - \{p_1, \ldots, p_{2b} \}, \ast)
    =
    \langle p_1, \ldots, p_{2b}, a_1, b_1, \ldots, a_g, b_g 
    \mid
    p_1 \cdots p_{2b} = [a_1,b_1] \cdots [a_g,b_g] \rangle,
\]
where the $a_{i}$, $b_{i}$ are generators of $\pi_1(\Sigma_g,\ast)$.
If $b = 0$, by convention this is the fundamental group of a closed genus $g$ surface
\[
    \pi_1(\Sigma_{g}, \ast)
    =
    \langle a_1, b_1, \ldots, a_g, b_g 
    \mid
    [a_1,b_1] \cdots [a_g,b_g] = 1 \rangle,
\]
while for $g = 0$ this is the group of a $2b$-times punctured sphere
\[
    \pi_1(\sphere{2} - \{p_1, \ldots, p_{2b} \}, \ast)
    =
    \langle p_1, \ldots, p_{2b} 
    \mid
    p_1 \cdots p_{2b} = 1 \rangle.
\]
Observe that for $b\geq 1$ this is a free group on $2g+2b-1$ generators, although it will be useful for us to instead remember the relation $p_1 \cdots p_{2b} = [a_1,b_1] \cdots [a_g,b_g]$, which we call the \emph{surface relation}.

\begin{figure}
    \centering
    \includegraphics[width=.9\linewidth]{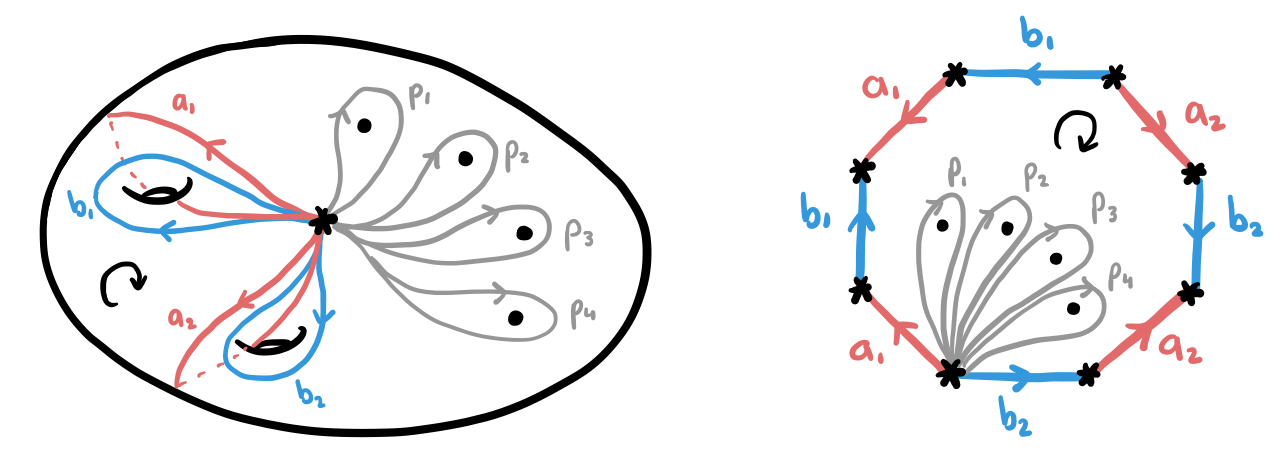}
    \caption{
    The genus $g$ oriented surface $\Sigma_g$ with basepoint $\ast$ and marked points $p_1, \ldots, p_{2b}$, represented two ways. Gluing the edges of the polygon on the right gives the surface on the left. The fundamental group $\pi_1(\Sigma_{g} - \{p_1, \ldots, p_{2b} \}, \ast)$ is generated by $a_{i}$, $b_{i}$, and (abusing notation) $p_{i}$. 
    \label{fig:surface}}
\end{figure}

We take \autoref{fig:surface} to be our standard model for the genus $g$ oriented surface $\Sigma_g$ with basepoint $\ast$ and marked points $p_1, \ldots, p_{2b}$. Proponents of the right-hand rule may be disappointed in our model, as each $a_i$ and $b_i$ pair violate this convention, but our choice of surface relation mandates the labels and orientations of the $a_i$ and $b_i$ curves. We choose to write the relation in this form for notational convenience on the algebraic side.

A genus $g$ \emph{handlebody} $H$ is a compact orientable 3-manifold
whose boundary is a genus $g$ closed surface, with the property
that $H$ can be cut along 2-dimensional disks such that the resulting space
is a set of 3-dimensional balls.
A $b$-component \emph{trivial tangle} $T$ in a handlebody $H$
is a collection of $b$ properly embedded arcs in $H$ such that
all of the arcs can be simultaneously isotoped into the boundary of $H$.
Given two handlebodies containing trivial tangles $(H_1, T_1)$ and $(H_2, T_2)$
with the property that $\partial H_1 = \partial H_2$ and $\partial T_1 = \partial T_2$,
we say that $(H_1, T_1)$ and $(H_2, T_2)$ are \emph{equivalent}
if there exists a diffeomorphism $H_1 \to H_2$ mapping $T_1$ to $T_2$
that is the identity on $\partial H_1$.
In the special case where $T_1$ and $T_2$ are empty,
we then say that the handlebodies $H_1$ and $H_2$ are equivalent.

We will be concerned with the set of equivalence classes of handlebodies and trivial tangles $(H,T)$ such that $\partial H = \Sigma_g$ and $\partial T = \{p_1, \ldots, p_{2b} \}$.
Any equivalence class of such a handlebody and trivial tangle can be described by a \emph{diagram} $\mathcal{D} = (C,S)$ on the surface $\Sigma_{g}$, made up of a collection of $g$ disjoint homologically linearly independent simple closed curves $C = \{C_1, \ldots, C_g\}$ (referred to as a \emph{cut system}) together with $b$ pairwise disjoint embedded arcs $S=\{S_1, \ldots, S_b\}$ whose endpoints are $\{p_1, \ldots, p_{2b}\}$ (referred to as a \emph{shadow diagram}).
In \cite{meier2018bridge4manifolds} this collection of cut system curves together with the shadow arcs is referred to as a ``curve-and-arc system.''

The handlebody $H$ can be constructed from the cut system by taking $\Sigma_{g} \times [0,1]$, attaching $g$ 3-dimensional 2-handles to $\Sigma_{g} \times \{1 \}$ along all of the curves $C_1,\ldots,C_g$,  and attaching a 3-ball to the resulting 2-sphere boundary component.
Given a shadow diagram in addition to the cut system, a trivial tangle in the resulting handlebody can be constructed by taking the arcs  of the tangle to be the union of $\{p_1, \ldots, p_{2b} \} \times [0, \frac{1}{2}]$ together with the arcs $S_1, \ldots, S_b$ in the shadow diagram considered as arcs in $\Sigma_{g} \times \{\frac{1}{2}\}$.
Conversely, given a handlebody $H$ and a trivial tangle $T$ we can obtain a diagram $\mathcal{D} = (C,S)$ for $(H,T)$ by choosing a set of $g$ disjoint embedded disks in $H$ that cut $H$ into a 3-ball and letting the cut system $C$ be the boundary of these disks, and taking an isotopy relative to the boundary of $T$ into $\Sigma_g$ and letting the shadow diagram $S$ denote the end result of this isotopy.   

We now give a name to these disks, as well as a few other disks referred to in some of the following proofs.
Given a handlebody $H$ and a trivial tangle $T$,
we refer to disjoint properly embedded disks bounded in $H$
by its cut system curves as \emph{cut disks},
and disjoint properly embedded disks that are the endpoint union
of a shadow arc and the associated tangle component as \emph{bridge disks}. 
One choice for these bridge disks is the track $S_{i} \times [0, \frac{1}{2}]$.
A \emph{bubble disk} in $H - T$ is a properly embedded disk
which encloses a bridge disk of the tangle strand; see \autoref{fig:bubble}.

\begin{figure}
    \centering
    \includegraphics[width=.4\linewidth]{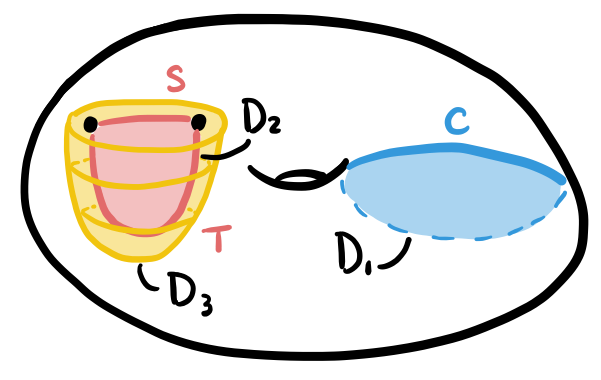}
    \caption{A cut disk $D_1$ bounded by a cut system curve $C$, a bridge disk $D_2$ bounded by a shadow arc $S$ and tangle strand $T$, and a bubble disk $D_3$.
    \label{fig:bubble}}
\end{figure}

\subsection{Bounding homomorphisms}

The following definition is motivated by \autoref{lem:main}, and examples are given in \autoref{ex:poincare} and \autoref{ex:main}.

\begin{definition}[Bounding homomorphism] \label{def:bounding_hom}
    A \emph{bounding homomorphism} is an epimorphism
    from a (possibly) punctured surface group to a free group
    \begin{equation*}
        \phi \colon
        \pi_1(\Sigma_{g} - \{p_1,\ldots,p_{2b} \}, \ast)
        \twoheadrightarrow
        \langle t_{1}, \ldots, t_{b},
        h_{1}, \ldots, h_{g} \rangle
    \end{equation*}
    with the following properties.
    \begin{enumerate}
        \item \label{def:bounding_hom_1} 
        The image of the subgroup generated by
        $a_1, b_1, \ldots, a_g, b_g$ surjects onto the quotient
        obtained by setting the $t_{i} = 1$.
        \item \label{def:bounding_hom_2} 
        Each $p_{i}$ maps to a conjugate of one
        of the $t_{j}$, where each of
        $t_{j}$ and its inverse $t_{j}^{-1}$ appears exactly once
        as the central letter.
        More precisely, there exists a bijection
        \[
        f \colon \{ p_{1}, \ldots, p_{2b} \}
        \rightarrow
        \{ t_{1}, t_{1}^{-1}, \ldots, t_{b}, t_{b}^{-1} \}
        \]
        and there are group elements
        $g_{i} \in \langle t_{1}, \ldots, t_{b},
        h_{1}, \ldots, h_{g} \rangle$
        with
        $\phi(p_{i}) = g_{i} f(p_{i}) g_{i}^{-1}$.
    \end{enumerate}
\end{definition}


There are two special cases worth mentioning. If $b = 0$, this is an
epimorphism from a closed surface group to a free group.
Topologically, this will correspond to having no tangle strands.
If $g = 0$, this corresponds to a trivial tangle in the 3-ball. The necessity of properties \eqref{def:bounding_hom_1} and \eqref{def:bounding_hom_2} will be seen in the proof of \autoref{thm:main}, but roughly speaking, property \eqref{def:bounding_hom_1} will allow us to distinguish between the handlebody and tangle, and property \eqref{def:bounding_hom_2} is a natural condition coming from the proof of \autoref{lem:main}. 

\begin{definition}[Topological realization]
    A \emph{topological realization} of a bounding homomorphism
    $\phi \colon
    \pi_1(\Sigma_{g} - \{p_1, \ldots ,p_{2b} \}, \ast)
    \twoheadrightarrow
    \langle t_{1}, \ldots, t_{b},
    h_{1}, \ldots, h_{g} \rangle$
    is a trivial tangle $T$ in
    a handlebody $H$ with $\partial H = \Sigma_g$ and $\partial T = \{p_1,\ldots,p_{2b}\}$
    such that there is an isomorphism
    $\psi \colon \pi_1(H - T, \ast)
    \xrightarrow{\cong}
    \langle t_1,\ldots, t_b, h_1,\ldots,h_g \rangle$
    that makes the following commute
    \begin{equation*}
    \begin{tikzcd}[row sep=tiny]
        & \pi_1(H - T, \ast) 
        \arrow{dd}{\psi}[swap]{\cong} \\
        \pi_1(\Sigma_{g} - \{p_1,\ldots,p_{2b} \}, \ast)
        \arrow[ur, "\iota_{*}", twoheadrightarrow]
        \arrow[dr, swap, "\phi", twoheadrightarrow] & \\
        & \langle t_1,\ldots, t_b, h_1,\ldots,h_g \rangle
    \end{tikzcd}
    \end{equation*}
    where the map
    $\iota_{*} \colon
    \pi_1(\Sigma_g - \{p_1,\ldots,p_{2b}\}, \ast) \twoheadrightarrow \pi_1(H-T, \ast)$
    is induced by inclusion.
\end{definition}

\begin{lemma}[Uniqueness of realization] \label{lem:uniqueness}
    Let $(H_{1}, T_{1})$ and $(H_{2}, T_{2})$ be two trivial tangles
    in handlebodies with $\partial H_1 = \partial H_2 = \Sigma_g$
    and $\partial T_1 = \partial T_2 = \{p_1,\ldots,p_{2b}\}$ so that there
    is an isomorphism $\rho$ between the fundamental groups of the
    tangle complements which makes the following diagram commute
    \begin{equation*}
    \begin{tikzcd}[row sep=tiny]
        & \pi_1(H_{1} - T_{1}, \ast) \arrow{dd}{\rho}[swap]{\cong} \\
    \pi_1(\Sigma_{g} - \{p_1,\ldots,p_{2b} \}, \ast)
    \arrow[ur, "\iota_{1*}", twoheadrightarrow] \arrow[dr, "\iota_{2*}", twoheadrightarrow] & \\
        & \pi_1(H_{2} - T_{2}, \ast)
    \end{tikzcd}
    \end{equation*}
    where the maps
    $\iota_{i*} \colon \pi_1(\Sigma_g - \{p_1,\ldots,p_{2b}\}, \ast) \twoheadrightarrow \pi_1(H_i-T_i, \ast)$, for $i=1,2$,
    are induced by inclusion.
    Then $(H_1,T_1)$ and $(H_2,T_2)$ are equivalent.  
\end{lemma}

\begin{proof}
    We will construct a diffeomorphism 
    $H_1 \rightarrow H_2$ mapping $T_1$ to $T_2$
    that extends the identity on the boundary $\partial H_1$,
    by defining it first on 2-cells in the complement of $T_1$ in $H_1$,
    and then extending over 3-balls.
    
    Fix a cut system for the handlebody $H_1$ and
    shadow arcs for the tangle $T_1$, and
    choose whiskers connecting each curve and arc to the basepoint.
    Let $\lambda$ be a based cut system curve 
    which thus bounds a cut disk in $H_1$.
    From commutativity of the diagram, we know that
    $\lambda$ is homotopically trivial in the tangle complement
    $H_2 - T_2$, and so by Dehn's lemma it bounds an embedded disk in $H_2 - T_2$.
    Extend the identity on the boundary to the cut disk bounded
    by $\lambda$ in $H_1 - T_1$ by mapping it to the disk
    obtained by Dehn's lemma in $H_2 - T_2$.
    In the same manner, extend the map to a complete
    system of cut disks for the handlebody $H_1$.
    
    Let $\eta$ be the based boundary of a closed
    tubular neighborhood of one of the shadow arcs of $T_1$.
    This curve bounds a bubble disk in $H_1 - T_1$; recall \autoref{fig:bubble}.
    Again from commutativity of the diagram, we have that
    $\eta$ is null-homotopic in $H_2 - T_2$, and by another application
    of Dehn's lemma bounds a disk.
    Use these disks to extend the map over all of the bubble
    disks in $H_1 - T_1$.
    
    To finish the construction, use the Alexander trick to
    extend the map over the 3-cells in $H_1$.
    Observe that each of the bubble disks cuts off a single
    tangle strand on one of its sides.
    Combined with the observation that there is a unique trivial
    1-strand tangle in the 3-ball, this shows that the
    diffeomorphism $H_1 \rightarrow H_2$ we constructed
    can be arranged to map $T_1$ to $T_2$.
\end{proof}

    Now we will set up some notation in preparation for the following lemma. Let $\mathcal{D} = (C,S)$ be a diagram for a handlebody and trivial tangle $(H,T)$
    with $\partial H = \Sigma_g$ and $\partial T = \{p_1,\ldots,p_{2b}\}$,
    together with an ordering of the curves in the cut system $C_1,\ldots,C_g$, an ordering of the arcs in the shadow diagram $S_1,\ldots,S_b$,
    and a choice of an orientation for each of the $C_i$ and $S_j$.  
    Observe that cutting the surface $\Sigma_g$ along the cut system
    curves and shadow arcs creates a connected, planar surface, and thus we will be able to choose dual loops $h_i$ and $t_j$ as follows.

    For each curve $C_i$, pick a closed loop $h_i$
    based at $\ast$ that does not intersect the arcs
    $S_{k}$ for any $k$ and the curves $C_{k}$ for $k \ne i$,
    and that intersects the curve $C_i$ in a single point.
    Orient $h_i$ so that at the point of intersection of $h_i$ and $C_i$,
    the orientation of $h_i$ followed by the orientation of $C_i$ agrees
    with the ambient orientation of $\Sigma_g$ (which we have assumed to be clockwise; see \autoref{fig:surface}). 
    
    Similarly, for each arc $S_j$, choose a loop $t_j$ based at $\ast$
    that intersects $S_{j}$ in exactly one point
    and does not intersect any of the other arcs or  curves.
    Orient $t_j$ so that at the point of intersection of $t_j$ and $S_j$,
    the orientation of $t_j$ followed by the orientation of $S_j$
    agrees with the ambient orientation of $\Sigma_g$. 
    Therefore, we now have
    $h_i, t_j \in \pi_1(\Sigma_g - \{ p_1,\ldots,p_{2b} \}, \ast)$. See \autoref{fig:lemma}.

\begin{lemma}[Diagrams to maps] 
    \label{lem:main} Using the notation defined above, we have the following. 

    \begin{enumerate}
        \item \label{lem:part_1} The loops representing $h_i$ and $t_j$ are well-defined in $\pi_1(H - T, \ast)$, independent of choices.
        
        \item \label{lem:part_2} The map 
        \begin{align*}
            \psi_{\mathcal{D}} \colon
            \pi_1(H - T, \ast) &\to \left< t_1,\ldots, t_b, h_1,\ldots,h_g \right> \\
            h_i &\mapsto h_i \\
            t_j &\mapsto t_j
        \end{align*}
        is an isomorphism.
    
        \item \label{lem:part_3} The composition of the map induced by inclusion
        and $\psi_{\mathcal{D}}$
        \[
            \phi \colon
            \pi_1(\Sigma_{g} - \{ p_1,\ldots,p_{2b} \}, \ast)
            \xrightarrow{\iota_{*}}
            \pi_1(H-T, \ast) 
            \xrightarrow{\psi_{\mathcal{D}}}
            \left< t_1,\ldots, t_b, h_1,\ldots,h_g \right>
        \]
        on an element $\lambda \in \pi_1(\Sigma_{g} - \{p_1,\ldots,p_{2b} \}, \ast)$
        is computed as follows.
        Represent $\lambda$ by a based, closed, immersed curve on
        $\Sigma_{g}$ that is transverse to the
        cut system curves $C_{i}$ and the shadow arcs $S_{i}$.
        We call this curve again $\lambda$.
        The image of $\lambda$ under this composition of maps is given
        by traversing $\lambda$ and building a word in the elements
        $t_1,\ldots, t_b, h_1,\ldots,h_g$ and their inverses as follows.
        Start with the empty word.
        For each intersection between $C_i$ and $\lambda$ we concatenate
        $h_i^{\pm 1}$ on the right,
        and for each intersection between $S_j$ and $\lambda$ we concatenate $t_j^{\pm 1}$ on the right,
        where the sign is determined by the sign of the intersection
        of the oriented curves as in \autoref{fig:sign}. 
    
        \item \label{lem:part_4} The map $\phi$ in
        \eqref{lem:part_3} is a bounding homomorphism,
        and with the choice of isomorphism $\psi_{\mathcal{D}}$,
        the trivial tangle $(H,T)$ is a topological realization of $\phi$. 
    \end{enumerate}
\end{lemma}

\begin{figure}[h]
    \centering
    \includegraphics[width=.45\linewidth]{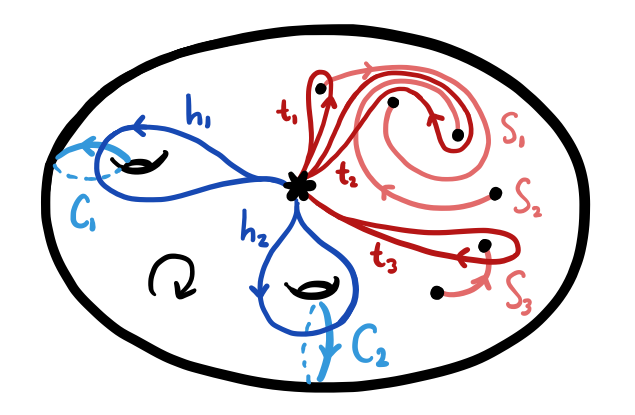}
    \caption{
    Finding loops $h_i$ dual to the cut system curves $C_i$, and loops $t_j$ dual to the shadow arcs $S_j$. 
    \label{fig:lemma}}
\end{figure}

\begin{figure}[h]
    \centering
    \includegraphics[width=.4\linewidth]{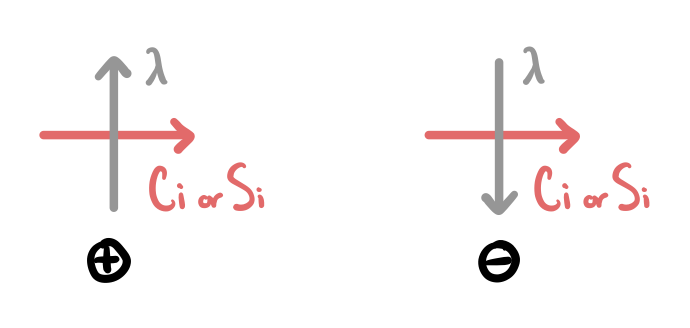}
    \caption{
        Sign convention for intersection points,
        where the pink arrow is a cut system curve $C_i$ or shadow arc $S_i$, and the gray arrow is an element in the fundamental group of the punctured surface, represented by a based,
        closed immersed curve $\lambda$ on the surface.
    \label{fig:sign}}
\end{figure}

\begin{proof}
Part \eqref{lem:part_1}:
Observe that the choices of these elements $h_i$ and $t_j$ are not unique
when considered as elements in $\pi_1(\Sigma_g, \ast)$;
see \autoref{fig:choice}. 
However we now prove that they are unique in the group $\pi_1(H - T, \ast)$. 
By choosing a collection of disjoint cut disks and bridge disks,
and cutting $H-T$ along these disks, we obtain a 3-ball as in \autoref{fig:disks}.
From this it follows that the choices of $h_i$ and $t_j$
are unique as elements of $\pi_1(H-T, \ast)$, because there is a unique
homotopy class of curves connecting points in a 3-ball.

\begin{figure}
    \centering
    \includegraphics[width=.75\linewidth]{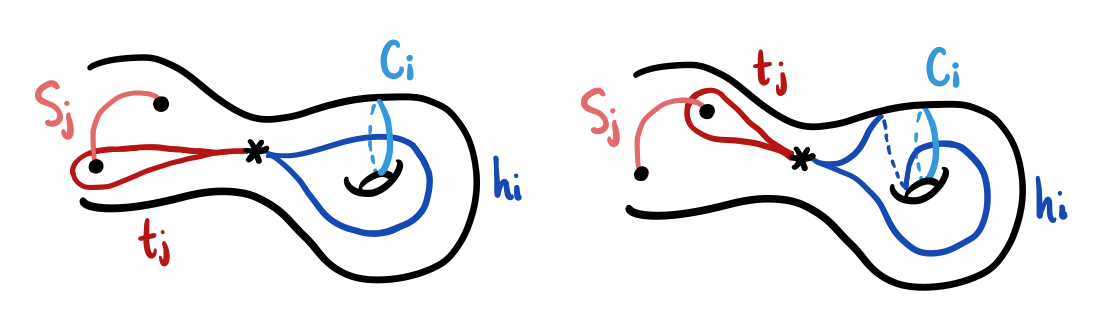}
    \caption{
        Inequivalent choices of $h_i$ and $t_j$ in
        the surface minus the points,
        which become isotopic in the handlebody minus the tangle determined
        by the $C_i$ and $S_j$.
    \label{fig:choice}}
\end{figure}

Part \eqref{lem:part_2}:
Now define the map 
\[
    \langle t_1,\ldots, t_b, h_1,\ldots,h_g \rangle \to \pi_1(H - T, \ast) 
\]
by sending $t_j \mapsto t_j$ and $h_i \mapsto h_i$.
We now argue that this map is an isomorphism.
First observe that $\pi_1(H - T, \ast)$ is a free group of
rank $g+b$, because $H-T$ deformation retracts onto
a spine obtained in the following way.
As above, cutting along (a choice of) cut disks
and bridge disks results in a 3-ball, and thus the $h_i$ and $t_j$ make up a
spine for the tangle complement $H - T$.
This also means that the homomorphism above is surjective,
and since free groups are Hopfian
it must be an isomorphism.

\begin{figure}
    \centering
    \includegraphics[width=.45\linewidth]{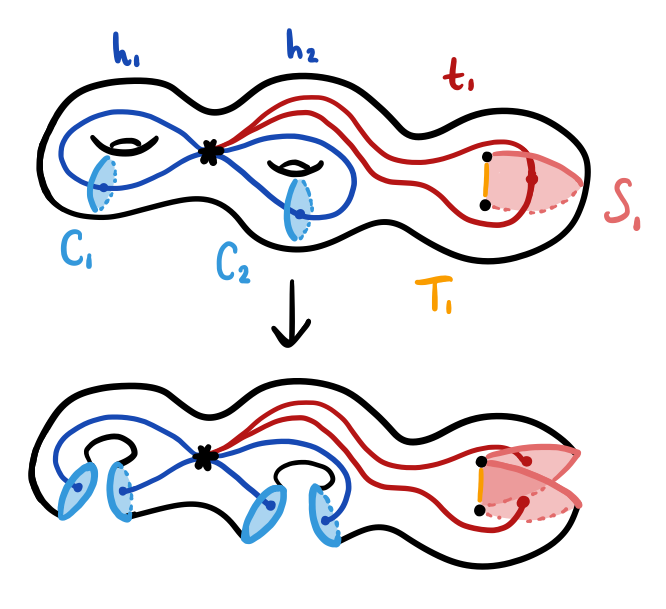}
    \caption{
        Cutting along cut disks (bounded by cut system curves $C_i$)
        and bridge disks (bounded by shadow arcs $S_i$ and tangles $T_i$) gives a 3-ball.
        The figure also shows how to build a spine of the
        tangle complement in the handlebody.
    \label{fig:disks}}
\end{figure}

Part \eqref{lem:part_3}: 
Using the spine from the proof of Part \eqref{lem:part_2}, we apply the cut system -- spine duality from \cite{johnson2006notes} to see that the image of an element $\lambda \in \pi_1(\Sigma_{g} - \{p_1,\ldots,p_{2b} \}, \ast)$ under these maps can be computed by recording intersections of $\lambda$ with the cut disks and bridge disks, since traversing an edge in the spine corresponds to hitting its dual disk.
For curves $\lambda$ that live on the surface $\Sigma_g$, these intersections occur on the boundaries of the disks. See \autoref{fig:disks}.

Part \eqref{lem:part_4}:
To check the first condition for a bounding homomorphism,
we glue 2-handles to the meridians of the tangle strands to kill
the normal closure of the generators $t_i$, and observe that the
$h_i$ make up a spine for the resulting handlebody. 

For the second condition, we represent
the generators $p_{i} \in \pi_1(\Sigma_{g} - \{p_1,\ldots,p_{2b} \}, \ast)$ in the following way.
Choose a simple system of rays (or whiskers) $r_{i}$ connecting
$\ast$ to each point in $\{p_1,\ldots,p_{2b} \}$.
Then $p_{i}$ corresponds to a loop that runs out along $r_{i}$,
goes around the puncture $p_{i}$ in a negative direction, as dictated by \autoref{fig:surface}, and returns along $r_{i}$.
On punctured surfaces this is also known as a Hurwitz arc system
\cite[Sec.\ 2.3]{kamada2002braid}.
The sequence in which the whisker of the
loop around $p_{i}$ intersects the cut system and tangle shadows will read off
a word $g_i$ in the free group.
Then running around the puncture reads off one of the generators $t_i$,
followed by returning to $\ast$ along $r_i$ contributing $g_{i}^{-1}$.
\end{proof}

\begin{example}[Handlebody with empty tangle]
    \label{ex:poincare}
    Here we give an example for the case $b = 0$. In this case the conditions on a bounding homomorphism $\phi \colon \pi_1(\Sigma_{g} , \ast) \twoheadrightarrow \langle h_{1}, \ldots, h_{g} \rangle $
    impose that it is an epimorphism from a surface group to a free group.
    Consider the genus 2 handlebody shown in \autoref{fig:poincare}, which is one of the handlebodies in a genus 2 Heegaard splitting of the Poincar{\'e} homology sphere.
    The bounding homomorphism for this handlebody can be read off by recording the sequence of intersections of the generators $a_{i}$ and $b_{i}$ with the curves $C_{j}$.
    \begin{align*}
        \pi_1(\Sigma_{2}, \ast)
        =\langle a_{1}, b_{1}, a_{2}, b_{2} \mid
        [a_{1}, b_{1}] [a_{2}, b_{2}] =1 \rangle
        & \twoheadrightarrow \langle h_{1}, h_{2} \rangle \\
        a_{1} & \mapsto h_{1}^{-1} \\
        b_{1} & \mapsto (h_{1} h_{2})^{5} h_{1}^{-2} \\
        a_{2} & \mapsto (h_{1} h_{2})^{5} h_{2}^{3} \\
        b_{2} & \mapsto h_{2}
    \end{align*}
\end{example}

\begin{figure}[h]
    \centering
    \includegraphics[width=0.5\linewidth]{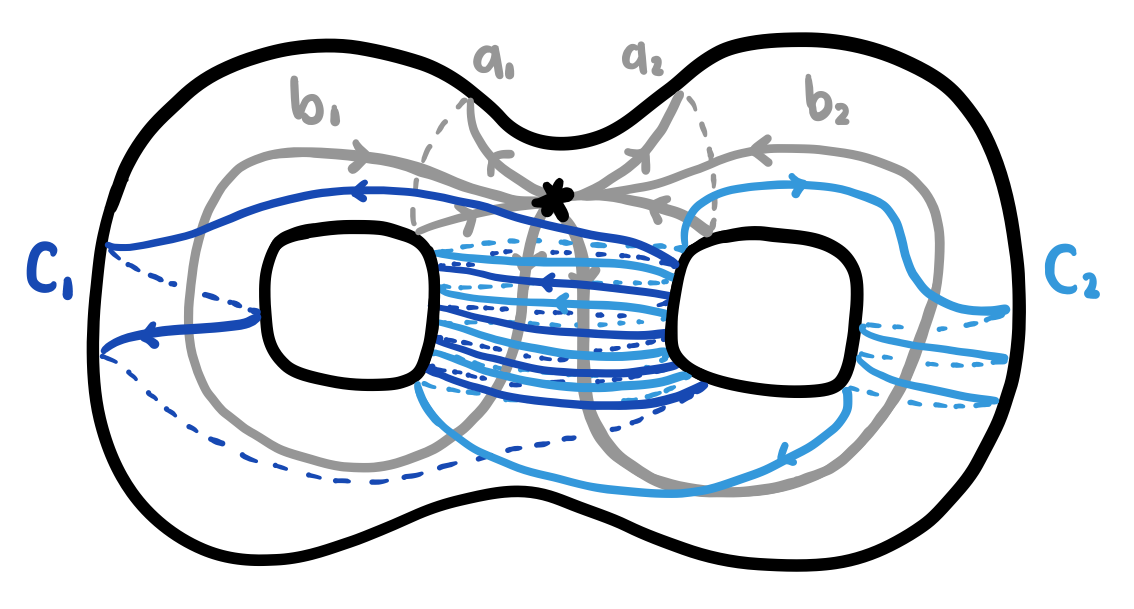}
    \caption{
        One of the handlebodies in a genus 2 Heegaard splitting of the Poincar{\'e} homology sphere, where the other handlebody is the solid genus 2 handlebody filling the interior. The generators $a_{1}, b_{1}, a_{2}, b_{2}$ of the fundamental group of the surface are shown in grey, and the two curves $C_{1}, C_{2}$ of the cut system are shown in light and dark blue.
        \label{fig:poincare}
        }
\end{figure}

\begin{example}[Running example]
    \label{ex:main}
    The following map is a bounding homomorphism which is realized by a trivial 2-bridge tangle in a solid genus 1 handlebody.
    It appears as the green tangle in the bridge trisection of $\RP{2}$ in $\CP{2}$ from \cite[Fig.~2]{meier2018bridge4manifolds} and \cite[Fig.~3]{joseph2022bridge}.
    See \autoref{fig:RP2_CP2}. We will use this bounding homomorphism as a running example in the proof of \autoref{thm:main}.
    \begin{align*}
        \pi_1(\Sigma_{1} - \{p_1,\ldots,p_{4} \}, \ast)
        =
        \langle p_{1}, p_{2}, p_{3}, p_{4},
        a_{1}, b_{1} 
        \mid
        p_1 p_2 p_3 p_4 = [a_1, b_1] \rangle
        & \twoheadrightarrow
        \langle t_{1}, t_{2},
        h_{1} \rangle \\
        p_{1} & \mapsto t_{2} h_{1} t_{1} h_{1}^{-1} t_{2}^{-1} \\
        p_{2} & \mapsto t_{2} \\
        p_{3} & \mapsto h_{1} t_{1}^{-1} h_{1}^{-1} \\
        p_{4} & \mapsto h_{1} t_{2}^{-1} h_{1}^{-1} \\
        a_{1} & \mapsto t_{2} h_{1} \\
        b_{1} & \mapsto h_{1}
    \end{align*}
\end{example}

\begin{figure}[h]
    \centering
    \includegraphics[width=.8\linewidth]{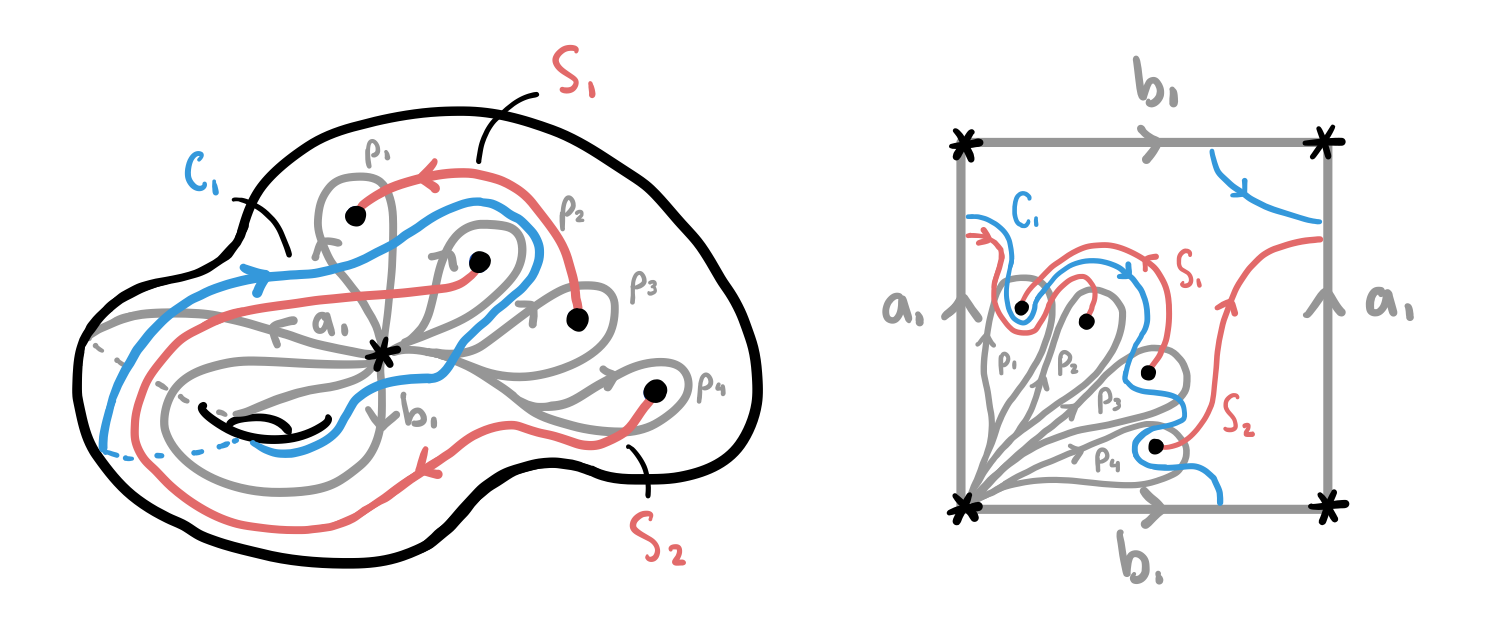}
    \caption{
    A trivial 2-bridge tangle in a solid genus 1 handlebody,
    which appears as the green tangle in the bridge trisection
    of $\RP{2}$ in $\CP{2}$ from \cite[Fig.~2]{meier2018bridge4manifolds} and \cite[Fig.~3]{joseph2022bridge}.
    \label{fig:RP2_CP2}}
\end{figure}

\subsection{Stallings folding}

We now discuss a technique due to Stallings called folding
\cite{stallings1983topology}
which in our context will be used to give
a topological realization of any bounding map.
There are several applications of folding
in the study of finitely generated free groups
(e.g. for the membership problem or determining
the index and normality of a subgroup,
see \cite[Chpt.~4]{clay2017officehours}). 
However for our purposes we only need one application, namely 
that folding gives a convenient algorithmic method to determine
if a set of elements $w_1, \ldots, w_k \in F_n$ generate $F_n$,
where $F_n$ is the free group generated by the elements $x_1,\ldots,x_n$.

We now describe this algorithm.
We begin by forming a directed graph $\Gamma$ with edges labeled  by elements
of $\{ x_1,\ldots,x_n \}$, where $\Gamma$ is topologically
a wedge of $k$ circles
and each of the circles is subdivided and labeled according to
the words $w_1,\ldots,w_k$ as in \autoref{fig:fold1}.
We can change $\Gamma$ by a move called a fold
to obtain a new such graph. 

\begin{definition}[Fold] \label{def:fold}
    A \emph{(Stallings) fold} is a move on a labeled, directed graph which takes two edges and replaces them with one single edge, with the following restrictions. The original edges must:
    \begin{enumerate}
        \item have the same label,
        \item share a vertex, and
        \item be oriented either both in or both out of the shared vertex.
    \end{enumerate}
    The label and orientation of the new edge are induced by those of the original edges.
    If the original edges share a second vertex, we call this a \emph{type I fold}, and if not, a \emph{type II fold}. In the case of the type II fold, the unshared vertices are identified together after replacing the original edges with the new edge. See \autoref{fig:fold2}. 
\end{definition}

\begin{figure}[h]
    \centering
    \includegraphics[width=.7\linewidth]{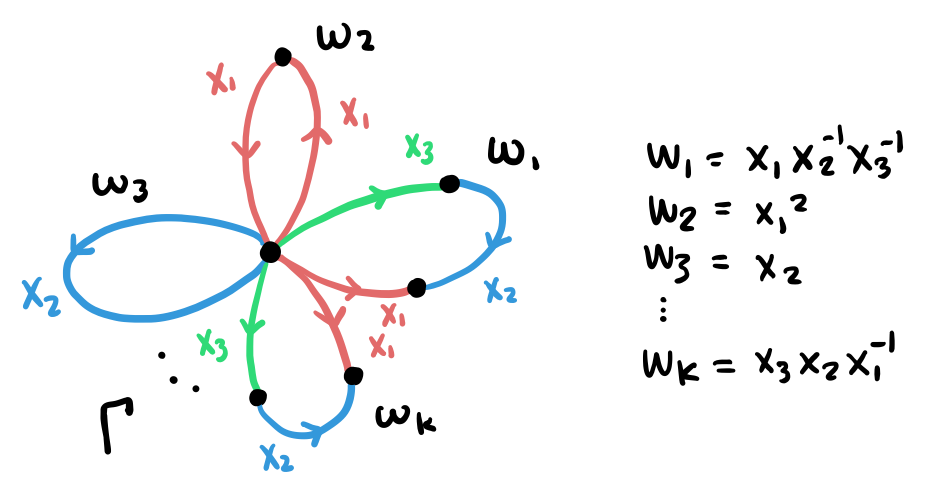}
    \caption{
     A directed graph $\Gamma$ with edges labeled by elements of
     $\{ x_1,\ldots, x_n\}$, where $\Gamma$ is topologically a wedge of $k$ circles and each of the circles is subdivided and labeled according to the words $w_1,\ldots,w_k$.
    \label{fig:fold1}}
\end{figure}

\begin{figure}[h]
    \centering
    \includegraphics[width=.8\linewidth]{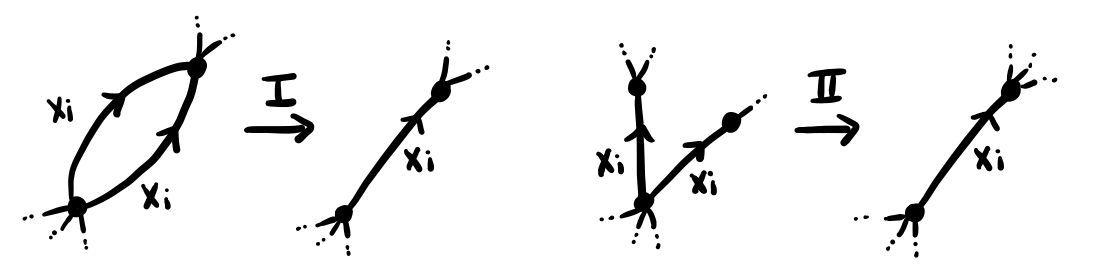}
    \caption{
        Types of folds. The labels and orientations of the edges being folded must match. 
    \label{fig:fold2}}
\end{figure}


\begin{lemma}[Stallings]
    \label{lem:fold}
    The elements $w_1,\ldots,w_k \in F_n$ generate $F_n$
    if and only if there exists a sequence of folds beginning with the graph $\Gamma$ given by
    $w_1,\ldots,w_k$ as in \autoref{fig:fold1}, and any such sequence of folds terminates in the graph $R_n$ in \autoref{fig:fold3}.  
\end{lemma}

\begin{figure}
    \centering
    \includegraphics[width=.35\linewidth]{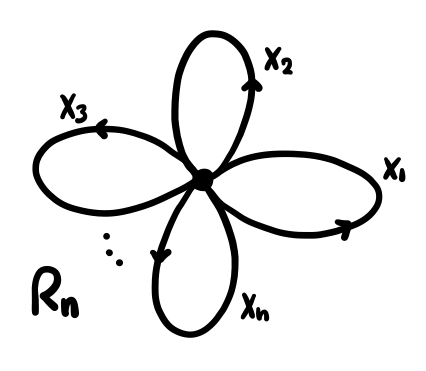}
    \caption{
        The directed graph $R_n$ is topologically a wedge of $n$ circles, with each circle uniquely labeled by an element $x_i$.
    \label{fig:fold3}}
\end{figure}


The proof of \autoref{lem:fold} follows as a special case of \cite[Thm.~4.7]{clay2017officehours}. Before we begin the proof of our main technical theorem, we mention one last lemma that will be used to check that a given set of closed curves constitutes a cut system.

\begin{lemma}
    \label{lem:cut}
    Let $w_1, \ldots, w_k$ be elements in a free group $F_n$
    with free generating set $x_1, \ldots, x_n$
    such that $w_1, \ldots, w_k$ generate $F_n$.
    Let $\exp_{x_i} \colon F_n \to \mathbb{Z}$
    denote the exponent sum homomorphism,
    that is, the signed count of occurrences of the letter $x_i$.
    Then the $n$ vectors 
    \[
        v_i = 
        (\exp_{x_i}(w_1), \exp_{x_i}(w_2), \ldots, \exp_{x_i}(w_k))
        \in \ZZ^{k}
    \]
    are linearly independent in $\ZZ^{k}$.  
\end{lemma}

\begin{proof}
    Consider the $(n \times k)$-matrix where the rows are given by the vectors $v_{i}$.
    The $j$th column is made up of all of the exponent sums of the word $w_{j}$,
    and thus computes the image of $w_{j}$ under the abelianzation map
    $\mathrm{ab} \colon F_{n} \rightarrow \ZZ^{n}$, where the images
    of the generators $x_{i}$ form the basis of the codomain.
    Since the words $w_1, \ldots, w_k$ generate $F_{n}$,
    the columns of the matrix generate $\ZZ^{n}$
    and thus its column rank is $n$.
    The claim now follows from the equality of row and column rank
    and the observation that $n \le k$.
\end{proof}

\begin{theorem}[Existence of realization] \label{thm:main}
    For every bounding homomorphism
    \begin{equation*}
        \phi \colon
        \pi_1(\Sigma_{g} - \{p_1,\ldots,p_{2b} \}, \ast)
        \twoheadrightarrow
        \langle t_{1}, \ldots, t_{b},
        h_{1}, \ldots, h_{g} \rangle
    \end{equation*}
    there exists a topological realization
    (where the realizing tangle $(H, T)$ is unique by \autoref{lem:uniqueness}) and further, we give a (polynomial-time) algorithm to construct a diagram for the topological realization.
\end{theorem}

\begin{proof}

The plan is to describe a topological realization of $\phi$ by an explicit diagram that we will produce in such a way that \autoref{lem:main} ensures that it is indeed a topological realization of $\phi$.
In the \hyperlink{proof:first_stage}{first stage} we produce a preliminary diagram $\mathcal{D}$ that, if it were to consist of only a cut system and a shadow diagram, would be a realization of $\phi$. This preliminary diagram will not necessarily be unique, and could possibly include ``extra'' components.
After this, in the \hyperlink{proof:second_stage}{second stage} we will apply Stallings folding from \autoref{lem:fold} to guide band sums which eliminate any extra
components of the preliminary diagram. At this point, we will have obtained a unique diagram, regardless of the choices made while producing the preliminary diagram.
Finally, in the \hyperlink{proof:third_stage}{third stage} we argue why the necessary bands always exist.

\vspace{2mm}
\hypertarget{proof:first_stage}{\textbf{First stage (Preliminary diagram):}}
For the first stage, it is necessary to initially make the distinction between words in the generators and inverses of generators in a free group, and elements of the free group.
For now, when we write $\phi(p_i), \phi(a_j), \phi(b_k)$, we mean the unique freely reduced words representing these elements.
To produce $\mathcal{D}$, look at two (not necessarily freely-reduced) words, the first given by concatenating the freely-reduced $\phi(p_i)$, namely  
\[
    w_1 = \phi(p_1) \phi(p_2) \cdots \phi(p_{2b})
\]
and the second given by concatenating the freely-reduced $\phi(a_i), \phi(b_j)$, namely
\[
    w_2 = \phi(a_1) \phi(b_1) \cdots \phi(a_g) \phi(b_g).
\]
Note that since $\phi$ is a homomorphism these words are equal as elements of
$\langle t_{1}, \ldots, t_{b},h_{1}, \ldots, h_{g} \rangle$.

The first step in drawing the preliminary diagram $\mathcal{D}$
is to represent $\Sigma_g$ by a polygon with edges $a_1,b_1,\ldots,a_g,b_g$
and punctures as in \autoref{fig:ex1}, which will begin our main running example for the proof. (If $g=0$, view the sphere as the plane with a point at infinity, place all of the punctures in a line, and place the basepoint at infinity. See \cite[Sec.~4.2.3]{blackwell2022combinatorial} for an example of this.) Mark each of the circles representing $a_1,b_1,\ldots,a_g,b_g$ on $\Sigma_g$ with ``oriented dashes'' labeled by the respective elements in $\phi(a_1),\phi(b_1),\ldots,\phi(a_g),\phi(b_g)$. Additionally, mark the loops around each punctured point $p_i$ on $\Sigma_g$ with ``oriented dashes'' labeled by the respective elements in $\phi(p_1),\ldots,\phi(p_{2b})$.

\begin{figure}
    \centering
    \includegraphics[width=.75\linewidth]{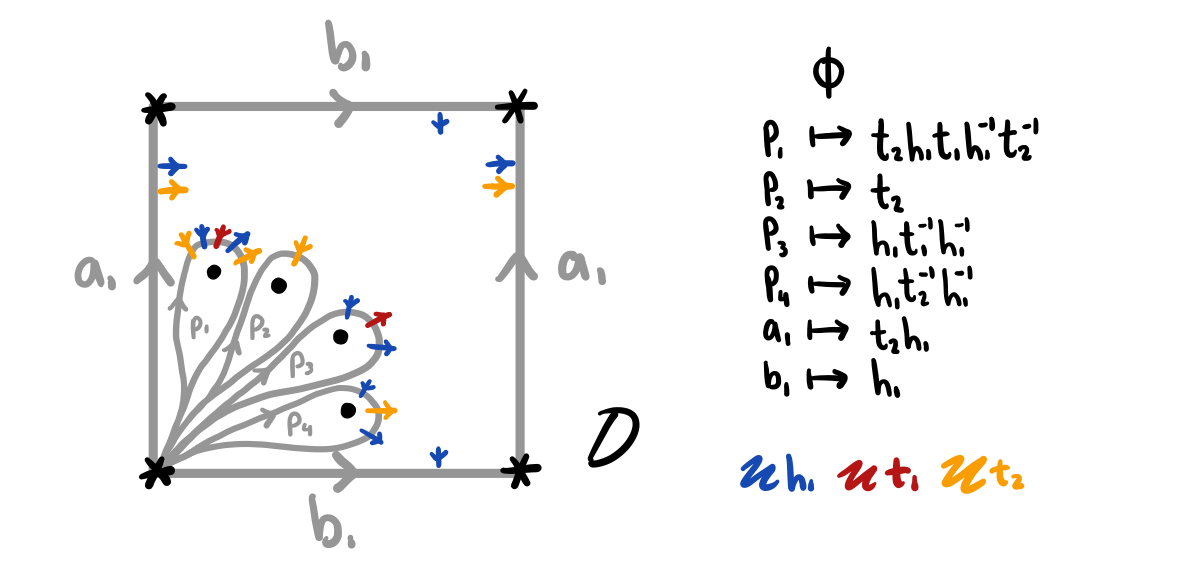}
    \caption{
        Our main running example for this proof
        comes from the bounding homomorphism in \autoref{ex:main}.
        Represent $\Sigma_1=T^2$ by a polygon with punctures. Mark each of the circles representing $a_1$ and $b_1$ with ``oriented dashes'' labeled (using color) by the respective elements in $\phi(a_1)$ and $\phi(b_1)$. Additionally, mark the loops around each punctured point $p_i$ with ``oriented dashes'' labeled by the respective elements in $\phi(p_1),\ldots,\phi(p_{4})$. Note that for the oriented dashes we use colors corresponding to $h_1$, $t_1$, and $t_2$ to see this correspondence easily, but keep in mind that throughout this running example we are really recovering the curves $C_1$, $S_1$, and $S_2$, which is why the colors in our end result, \autoref{fig:ex8}, differ slightly from those in \autoref{ex:main}.
    \label{fig:ex1}}
\end{figure}

The second step is to freely reduce both of the words $w_1$ and $w_2$, and as cancellations occur in the free reductions, draw arcs between the corresponding dashes as in \autoref{fig:ex2}.
The arcs retain the respective labelings and have orientations induced by the dashes.
Let $w_1'$ and $w_2'$ denote the resulting freely-reduced words, which are equal as freely-reduced words since they are equal as elements of the free group.
Therefore $(w_2')^{-1}w_1'$ freely reduces to yield the trivial word.

\begin{figure}
    \centering
    \includegraphics[width=.75\linewidth]{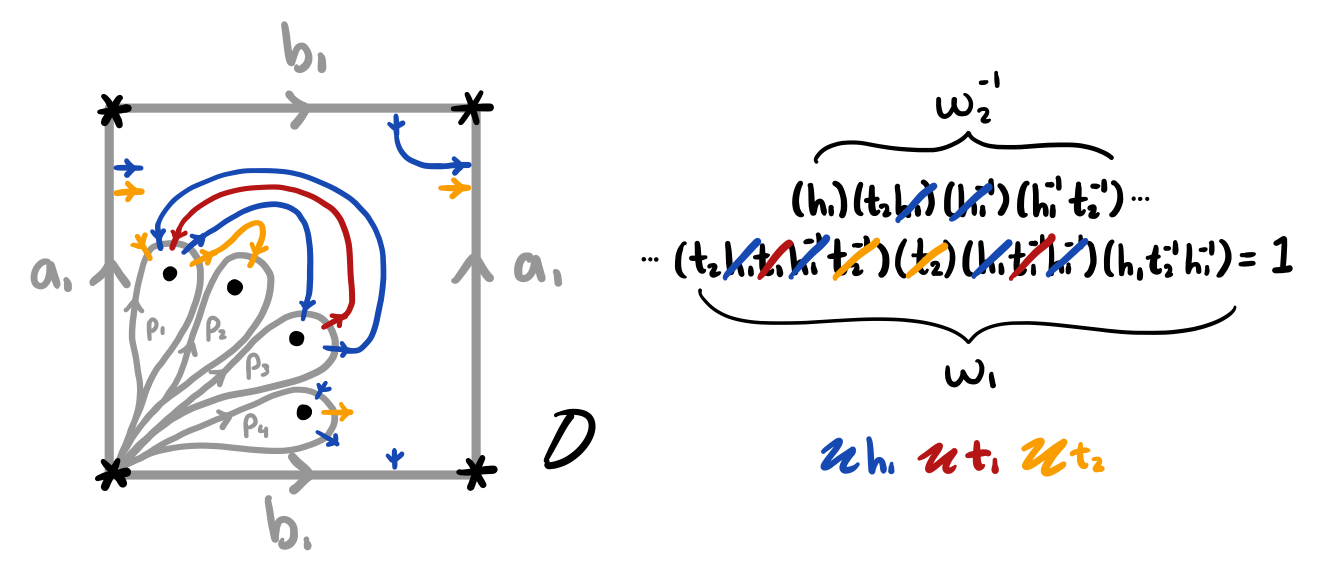}
    \caption{
        Running example. Drawing arcs between the dashes as corresponding cancellations occur in the free-reductions of the words $w_1$ and $w_2$. There are multiple ways to freely-reduce these words; this running example shows one possible choice. In fact, there exists a different choice of cancellations here which would avoid the need to proceed to the second stage.
    \label{fig:ex2}}
\end{figure}

The final step
is to continue to carry out this free reduction
down to the trivial word,
drawing arcs with each cancellation as in the second step.
See \autoref{fig:ex3}. This reduction will not necessarily be unique and could produce different diagrams, but this indeterminacy will be fixed in the \hyperlink{proof:second_stage}{second stage}.
The loops around the punctures each intersect an odd number of dashes, which are now each part of an arc.
Connect the middle dash to the punctured point, and connect the rest of the dashes in pairs that go around the opposite side of the puncture, as in \autoref{fig:ex4}.

\begin{figure}
    \centering
    \includegraphics[width=.8\linewidth]{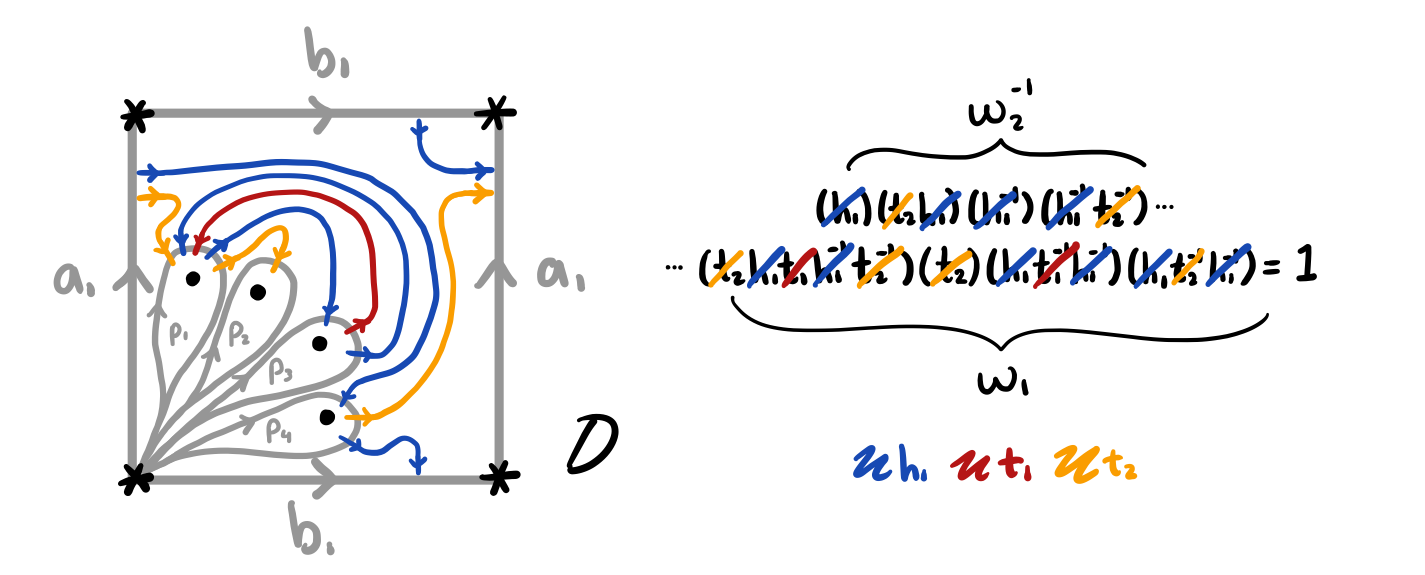}
    \caption{
        Running example. Continuing to draw arcs between dashes until all cancellations have occurred. 
    \label{fig:ex3}}
\end{figure}

\begin{figure}
    \centering
    \includegraphics[width=.6\linewidth]{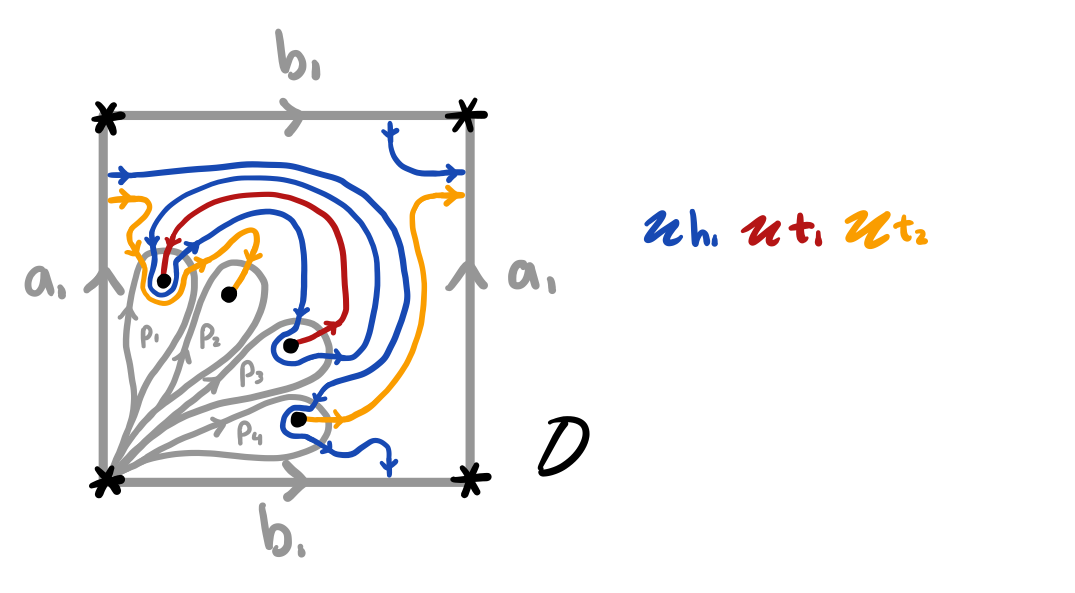}
    \caption{
        Running example. Connecting the middle dash
        to the puncture and the other dashes
        in pairs.
    }
    \label{fig:ex4}
\end{figure}

We consider the preliminary diagram $\mathcal{D}$ to be the resulting collection of disjoint, oriented arcs
and closed curves on $\Sigma_g$, each labeled by a generator of
$\langle t_{1}, \ldots, t_{b},h_{1}, \ldots, h_{g} \rangle$. In general $\mathcal{D}$ at this stage will consist
of too many closed loops and will not give a realization of $\phi$.
Note that by condition \eqref{def:bounding_hom_1} in the definition of a bounding homomorphism,
there is at least one closed curve with each label $h_j$.
Similarly, by condition \eqref{def:bounding_hom_2}, there is exactly one arc with each label $t_i$.
However, in general $\mathcal{D}$
will consist of additional closed curves with labels $h_j$ and $t_j$. 

Suppose that at this stage there are no such additional curves in $\mathcal{D}$,
namely, for each $h_i$ there is exactly one closed curve with that label
and for each $t_j$ there are no closed curves with that label. We now show that the curves labeled by the elements $h_i$ form a cut system, namely that they are homologically linearly independent.  

Let $C_1, \ldots, C_g$ denote the closed curves in this case,
with corresponding labels $h_1, \ldots, h_g$, respectively.
The curves $a_1,b_1, \ldots, a_g,b_g$ in \autoref{fig:surface}
give a basis for $H_1(\Sigma_g; \mathbb{Z})$. 
Using that the intersection product on $H_1(\Sigma_g; \mathbb{Z})$ is given by
$a_i \cdot b_j = \delta_{ij}$ we find that
by construction, in $H_1(\Sigma_g; \mathbb{Z})$ we have
\[
    C_i =
    \exp_{h_i}(\phi(b_1)) a_1 + \exp_{h_i}(\phi(a_1)) b_1
    + \cdots +
    \exp_{h_i}(\phi(b_g)) a_g + \exp_{h_i}(\phi(a_g)) b_g
\]
where
\[
    \exp_{h_i} \colon \langle t_1, \ldots, t_b, h_1, \ldots, h_g \rangle \to \ZZ
\]
maps an element to the exponent sum of $h_i$.
By property \eqref{def:bounding_hom_1}
of being a bounding homomorphism,
together with \autoref{lem:cut},
it follows that the elements of $C_i$ are linearly independent
in $H_1(\Sigma_g; \mathbb{Z})$ and therefore in this case the curves
$C_1, \ldots,C_g$ form a cut system.

From this it follows that, in this case, $\mathcal{D}$ is a diagram for a handlebody and trivial tangle. Furthermore, in this case, we know by \autoref{lem:main} that $\mathcal{D}$ is a diagram for a topological realization of $\phi$.

\vspace{2mm}
\hypertarget{proof:second_stage}{\textbf{Second stage (Stallings folding):}}
In the second stage, we will modify the preliminary diagram $\mathcal{D}$
in steps by performing orientation-preserving band sums
between components in order to eliminate the extra closed curves, as in \autoref{fig:ex8}. These bands may pass through the $a_i, b_j,$ and $p_k$ representing the elements in $\pi_1(\Sigma_g - \{p_1,\ldots,p_{2b} \}, \ast)$, but may not pass through the curves and arcs we drew in the previous stage.
Assuming for a moment that we successfully banded together all of the curves and arcs, so that there is exactly one curve/arc with each label, the above argument still applies and the resulting diagram will be a diagram for a cut system and trivial tangle (namely, it will consist of a cut system together with a shadow diagram) since each band introduces a pair of canceling intersections as in \autoref{fig:band1}.
Furthermore, because the added intersections cancel, applying \autoref{lem:main} shows that the resulting diagram is in fact a diagram for a topological realization of $\phi$.   

\begin{figure}[h]
    \centering
    \includegraphics[width=.3\linewidth]{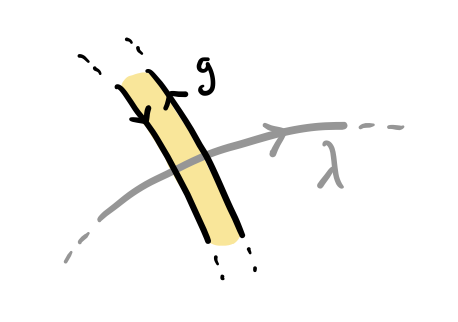}
    \caption{
        Following the curve $a$ running across a band,
        we read off the canceling pair of intersections
        $g g^{-1}$.
    \label{fig:band1}}
\end{figure}

The algorithm for finding these bands is as follows. Note that here and beyond, we use the word \textit{color} to describe the ``labels'' as mentioned in \autoref{def:fold}, and give the word \textit{label} a new, specific meaning. (In particular, we use \textit{color} colloquially to mean ``labeled by a color,'' as opposed to the graph-theoretic notion.)
Let $\Gamma$ be the graph that is topologically a wedge of circles such that each circle is \emph{colored} by the words $$\phi(p_{1}),\ldots,\phi(p_{2b}),\phi(a_1),\phi(b_1),\ldots,\phi(a_g),\phi(b_g)$$ as in \autoref{fig:ex5}. Add an additional \emph{label} to each edge of $\Gamma$,
namely label each edge of $\Gamma$ with the corresponding closed curve or arc
in the preliminary diagram $\mathcal{D}$ as in
\autoref{fig:ex6}.
We will modify the graph $\Gamma$ by folding and this will dictate how to modify the diagram $\mathcal{D}$ by band sums. At each stage, we will refer to the new graph again by $\Gamma$ and the new diagram again by $\mathcal{D}$.

\begin{figure}
    \centering
    \includegraphics[width=.7\linewidth]{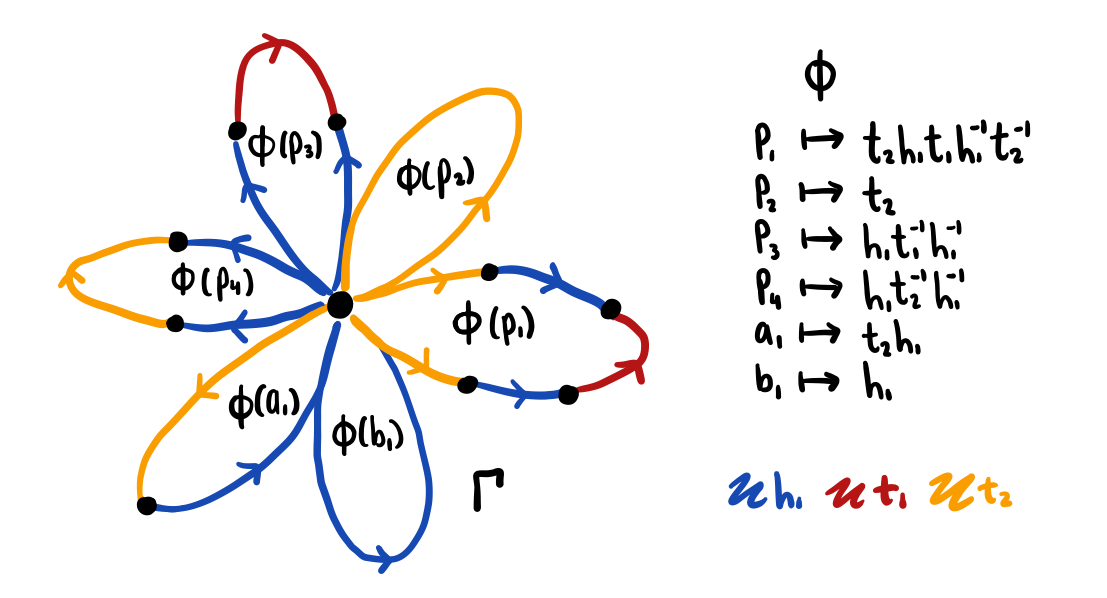}
    \caption{
        Running example. The graph $\Gamma$ is topologically a wedge of circles, where each circle is colored by the words $\phi(p_{1}),\ldots,\phi(p_{4}),\phi(a_1),\phi(b_1)$. Here colors are used to denote this coloring.
    \label{fig:ex5}}
\end{figure}

\begin{figure}
    \centering
    \includegraphics[width=.75\linewidth]{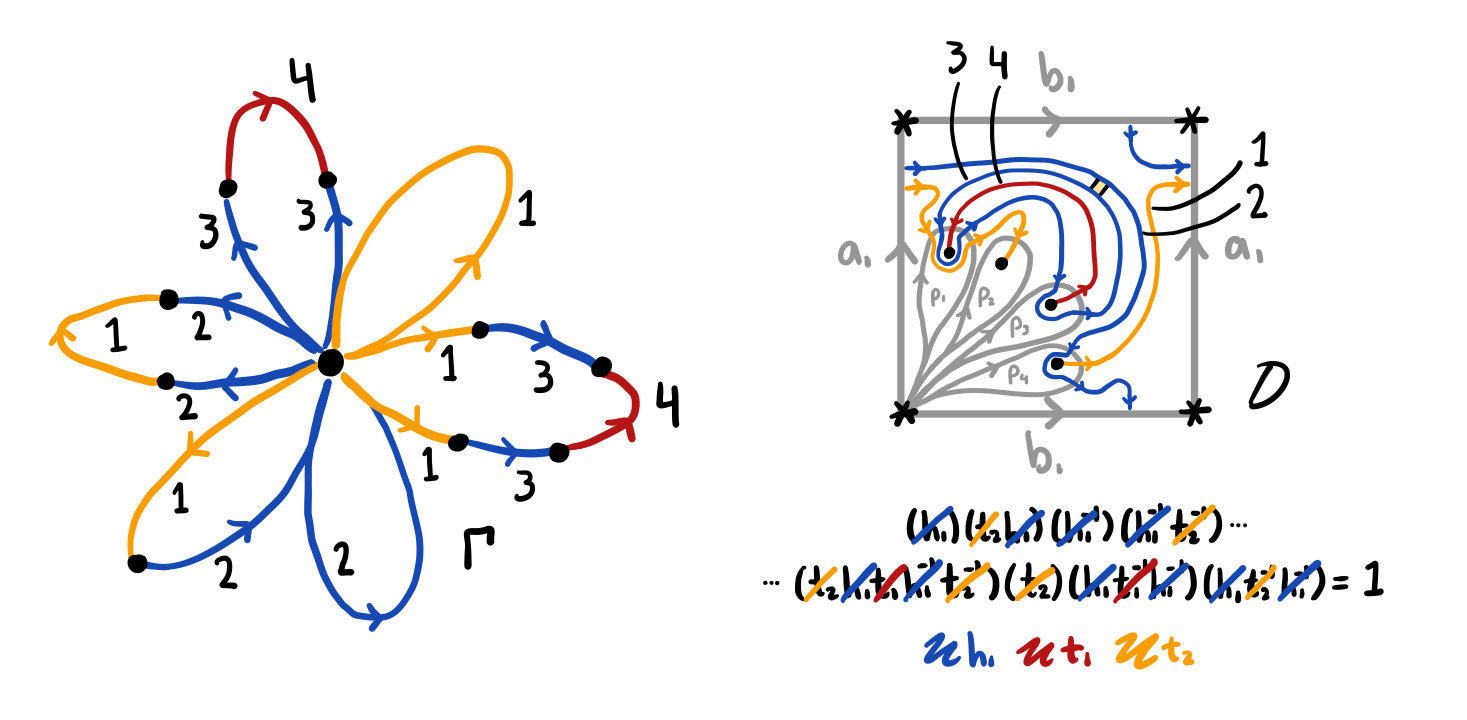}
    \caption{
        Running example. An additional label is added to each edge of $\Gamma$ which represents the corresponding curve/arc in $\mathcal{D}$.
        Here numbers are used to denote these labels.
    \label{fig:ex6}}
\end{figure}

Since the elements
\[
    \phi(p_{1}),\ldots,\phi(p_{2b}),\phi(a_1),\phi(b_1),\ldots,\phi(a_g),\phi(b_g)
\]
generate the free group $\langle t_{1}, \ldots, t_{b},h_{1}, \ldots, h_{g} \rangle$, by \autoref{lem:fold} there exists a sequence of foldings of the graph $\Gamma$
to the graph $R_n$ in \autoref{fig:fold3}, where $n=b+g$, the rank of the free group.
Choose a sequence of such foldings. Recall that foldings must occur between edges with the same \emph{coloring}, but not necessarily the same \emph{labels} (using the language of the previous paragraph).
Whenever two edges of $\Gamma$ with the same label are folded, the diagram $\mathcal{D}$
is left unchanged and the edge of $\Gamma$ resulting from the fold is given the same label as the edges it came from.
Whenever two edges of $\Gamma$ with different labels are folded,
an orientation-preserving band between the two corresponding curves/arcs is chosen,
disjoint from all of the other curves/arcs,
and the diagram $\mathcal{D}$ is modified by preforming a band sum
along this band as in \autoref{fig:band2}.
The edge of $\Gamma$ resulting from the fold is given a new labeling
that identifies the two banded together curves/arcs.
Any other occurrences of the involved labels are modified
as well.
See \autoref{fig:ex7} for a sequence of folds in our running example, and \autoref{fig:ex8} for the resulting band sum. See \autoref{ex:bandex} for a different example containing more complicated band sums.

\begin{figure}
    \centering
    \includegraphics[width=.5\linewidth]{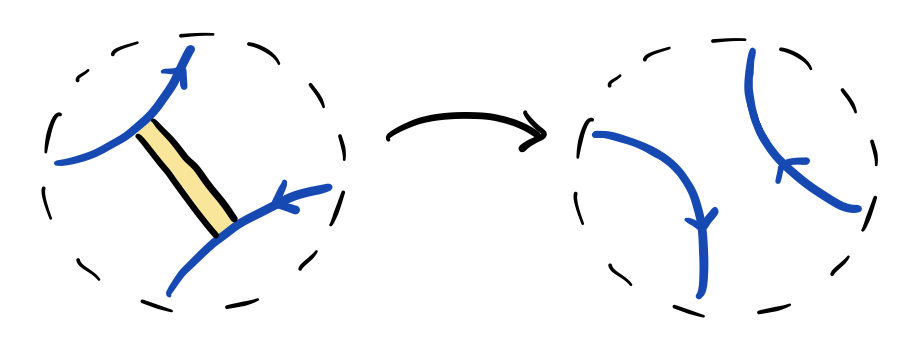}
    \caption{
        An orientation-preserving band sum between two curves/arcs in a neighborhood within the diagram $\mathcal{D}$. The curves/arcs are the same color but have different labels, that is, they are different curves/arcs in $\mathcal{D}$ that are colored by the same word.
    \label{fig:band2}}
\end{figure}

\begin{figure}
    \centering
    \includegraphics[width=1\linewidth]{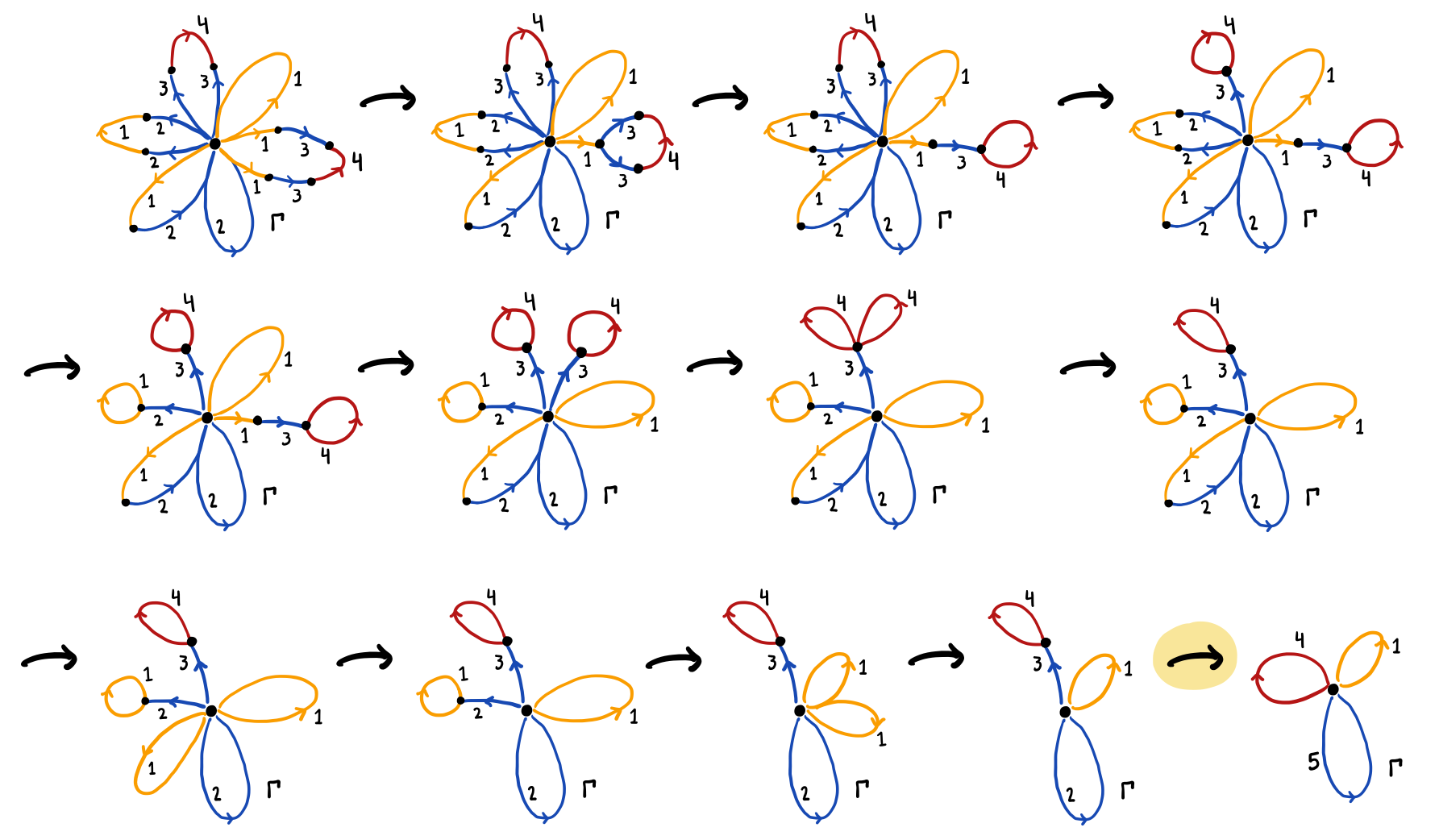}
    \caption{
        Running example. One possible sequence of Stallings folds. The diagram $\mathcal{D}$ remains unchanged throughout all of the folds except the last fold (highlighted), after which a band sum is performed as in \autoref{fig:ex8}.
    \label{fig:ex7}}
\end{figure}

\begin{figure}
    \centering
    \includegraphics[width=.9\linewidth]{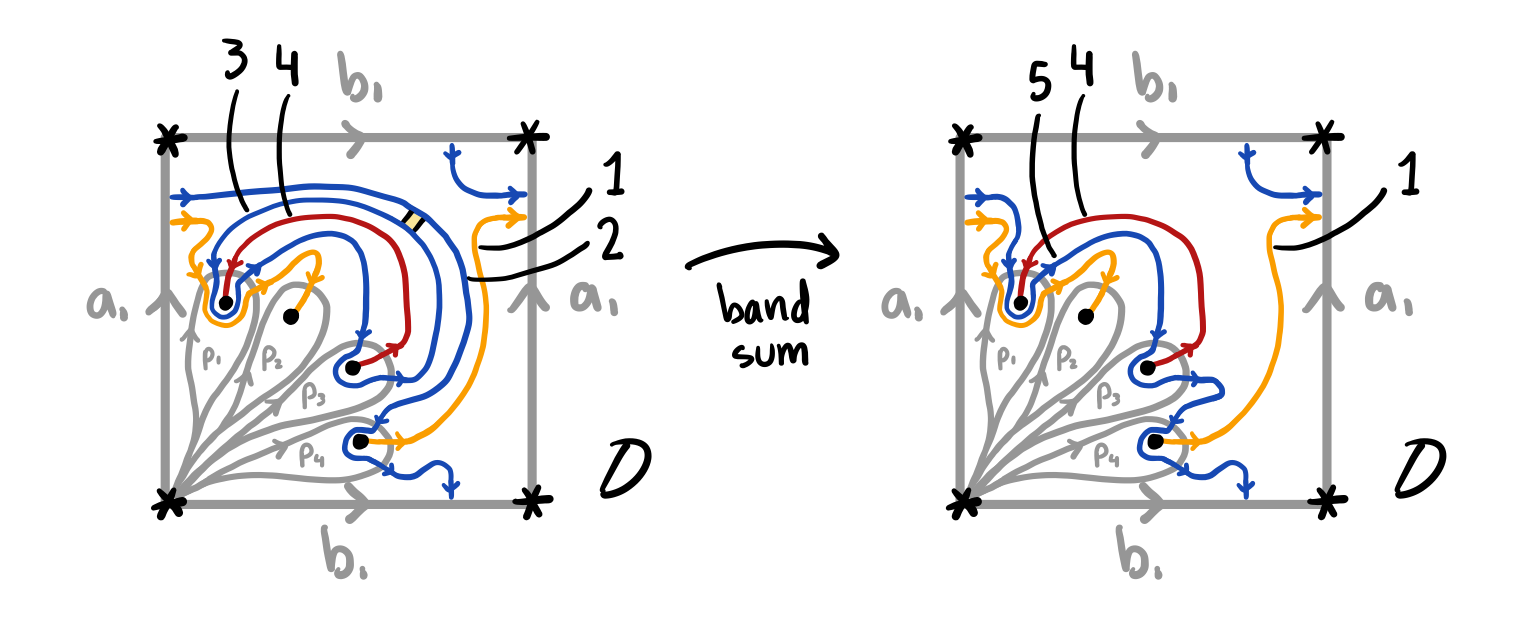}
    \caption{
        Running example. The curves 2 and 3 on the left are banded together to form the new curve 5 on the right. Now we are finished! In the notation of \autoref{ex:main}, the curve $5$ is $C_1$, the arc $4$ is $S_1$, and the arc $1$ is $S_2$.
    }
    \label{fig:ex8}
\end{figure}

Assuming for a moment that all the required bands do exist, in the final diagram there will be one curve/arc in the diagram for each edge of $R_n$, since the the graph $\Gamma$ folds down to $R_n$. It then follows that there will be exactly $b$ arcs and $g$ closed curves in the final diagram $\mathcal{D}$. The proceeding discussion then applies to see that the resulting closed curves form a cut system and the diagram $\mathcal{D}$ indeed does provide a topological realization of $\phi$. 

\vspace{2mm} 
\hypertarget{proof:third_stage}{\textbf{Third stage (Existence of bands):}}
We now tackle the problem of the existence of the bands, which will follow from the following claim.
Note that in the claim we are not assuming edges have the same color or label, even though our definition of folding requires edges to have the same color. Furthermore, we prove the existence of some ``extra'' orientation-reversing bands, which we do not need for the second stage of our proof, but we \textit{do} need as part of our inductive argument for the claim. Thus we prove that bands exist in a more general setting, which will imply that the specific bands we want in the previous part of the proof do indeed exist.

\begin{*claim}
    Given two incident edges $e_1$ and $e_2$ in the graph $\Gamma$
    at any state of the above procedure, there exists a band between the corresponding curves/arcs that label $e_1$ and $e_2$ which is disjoint from all other curves/arcs in the diagram $\mathcal{D}$. In particular:
    \begin{enumerate}
        \item If the incident edges are oriented such that they both go into their shared vertex, as in I of \autoref{fig:cases1}, then there exists an orientation-preserving band from the right of the curve/arc labeling $e_1$ to the right of the curve/arc labeling $e_2$.
        \item If the incident edges are oriented such that they both go out of their shared vertex, as in II of \autoref{fig:cases1}, then there exists an orientation-preserving band from the left of the curve/arc labeling $e_1$ to the left of the curve/arc labeling $e_2$.
        \item If the incident edges are oriented such that one goes into their shared vertex and one goes out, as in III of \autoref{fig:cases1}, then there exists an orientation-reversing band from the right of the curve/arc labeling $e_1$ to the left of the curve/arc labeling $e_2$.
    \end{enumerate}
\end{*claim}

\begin{figure}
    \centering
    \includegraphics[width=.7\linewidth]{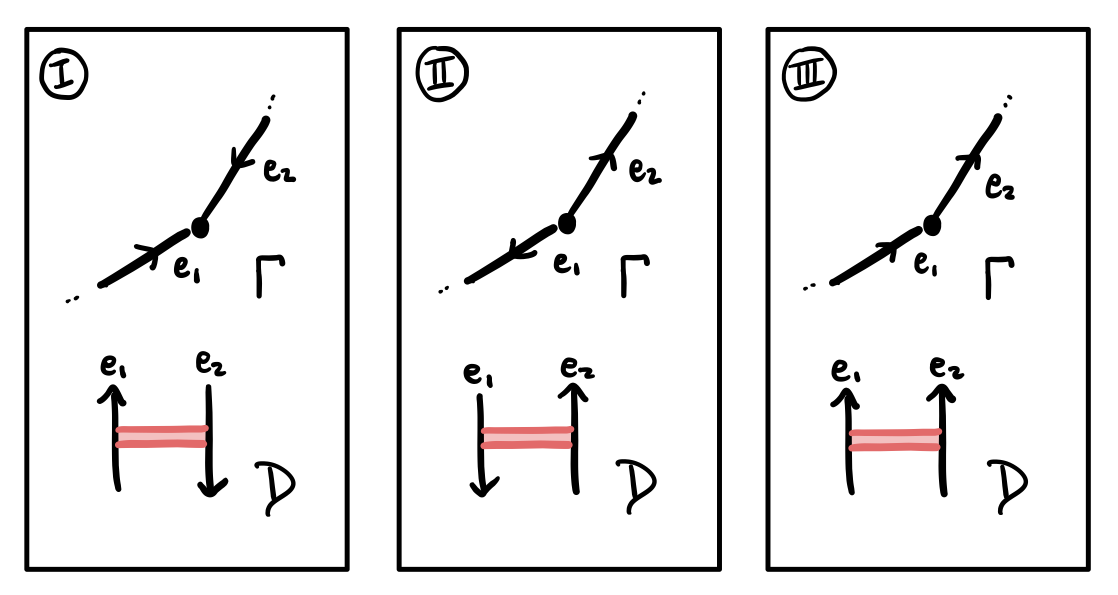}
    \caption{ 
        We consider three cases for how the edges $e_1$ and $e_2$ are oriented. Here $e_1$ and $e_2$ are the names of the edges in the graph $\Gamma$, not labels or colors, and we use the same notation to refer to the curves/arcs in $\mathcal{D}$ that correspond to the edges in $\Gamma$.
    \label{fig:cases1}}
\end{figure}

Here we are fixing some conventions on how the orientations of incident edges in $\Gamma$ correspond to orientations in $\mathcal{D}$, which can be done without loss of generality and in such a way that the cases I, II, and III are compatible with each other. Also note that in some cases, both an orientation-preserving and an orientation-reversing band might exist between the corresponding curves/arcs. The statement of the claim only contains the existence of those bands which are necessary for the proof.

We prove the claim by induction on the number of folds that have been performed in the algorithm, checking throughout that our three cases for orientations hold. We first check that the claim is valid for the preliminary diagram $\mathcal{D}$ before any folding has been performed. Our initial graph $\Gamma$ is topologically a wedge of circles, and incident edges can either be in the same circle or in different circles. 

\begin{figure}[p]
    \centering
    \includegraphics[width=.9\linewidth]{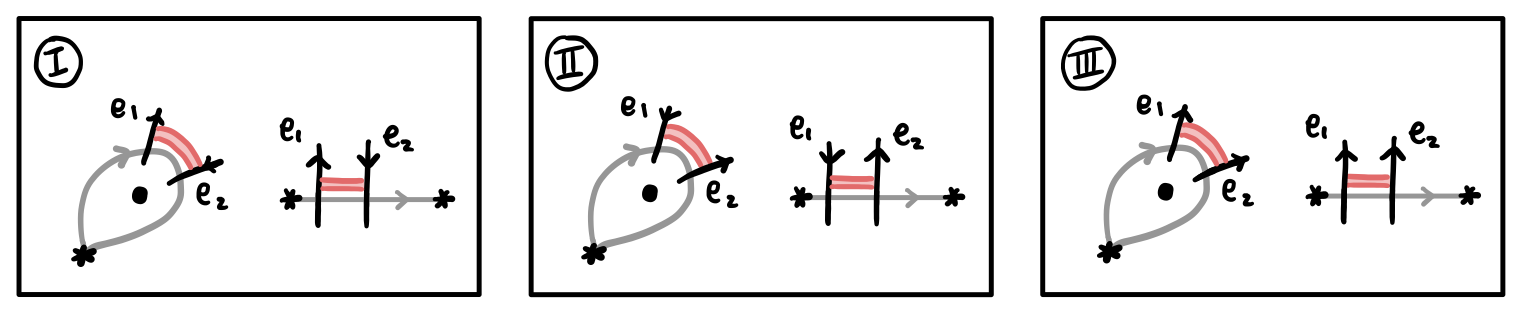}
    \caption{ 
        Some orientation checks for the base case of the claim, in which the edges $e_1$ and $e_2$ are in the same circle in the graph $\Gamma$. We abuse notation and use $e_1$ and $e_2$ to also refer to the oriented dashes in $\mathcal{D}$ that correspond to these edges in $\Gamma$. These figures show $e_1$ and $e_2$ in the diagram $\mathcal{D}$, where $\ast$ is the basepoint and the grey, unlabeled curves/arcs represent the generators $a_i$, $b_i$, and $p_i$, in the notation of the previous stages. Here we check whether a disjoint orientation-preserving band (in pink) between $e_1$ and $e_2$ can be found in cases (I) and (II), and whether a disjoint orientation-reversing band (in pink) between $e_1$ and $e_2$ can be found in case (III).
    \label{fig:cases2}}
\end{figure}

\begin{figure}[p]
    \centering
    \includegraphics[width=.9\linewidth]{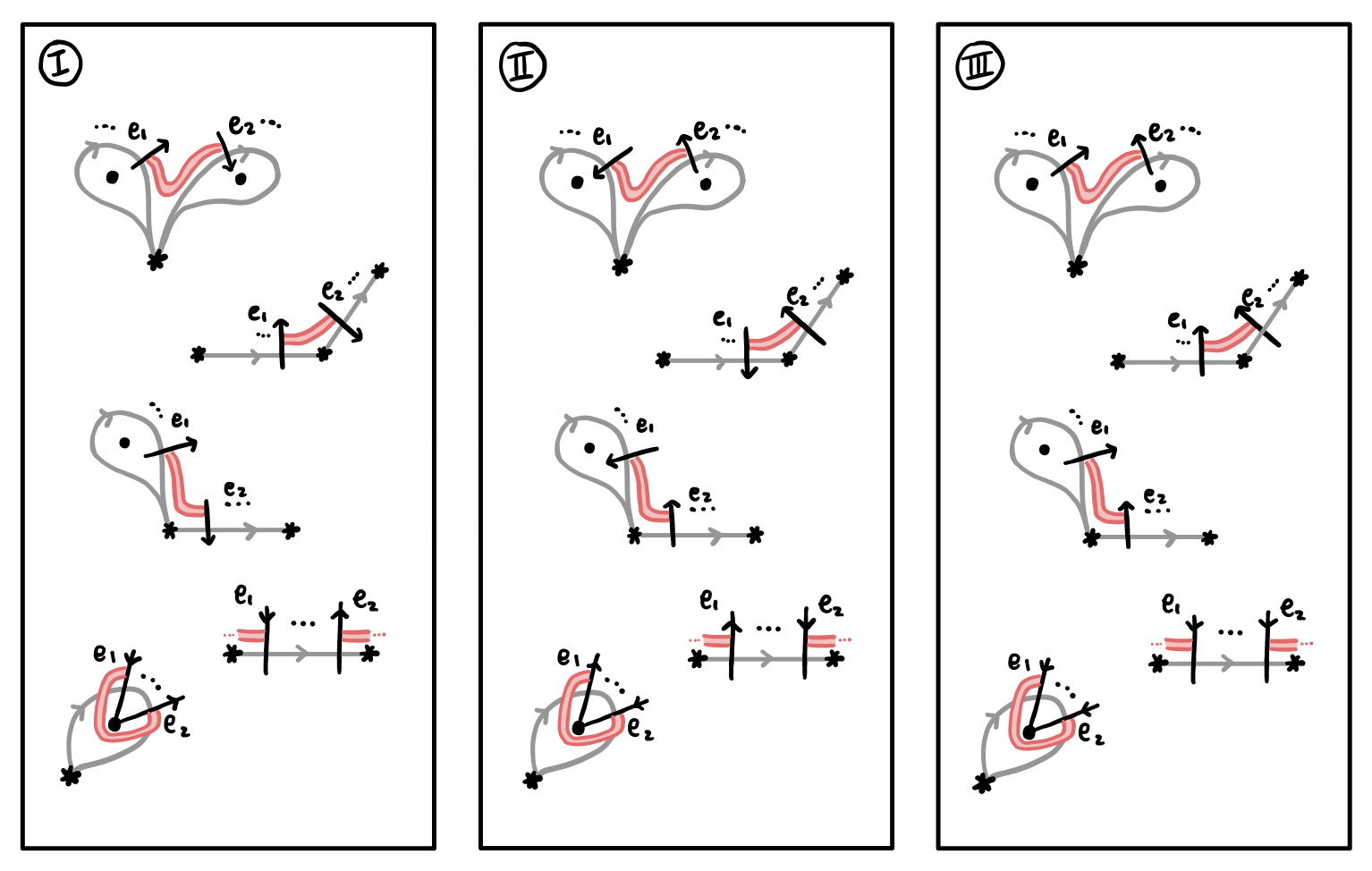}
    \caption{ 
        Some orientation checks for the base case of the claim, in which the edges $e_1$ and $e_2$ are in different circles in the graph $\Gamma$. We abuse notation and use $e_1$ and $e_2$ to also refer to the oriented dashes in $\mathcal{D}$ that correspond to these edges in $\Gamma$. These figures show $e_1$ and $e_2$ in the diagram $\mathcal{D}$, where $\ast$ is the basepoint and the grey, unlabeled curves/arcs represent the generators $a_i$, $b_i$, and $p_i$, in the notation of the previous stages. Here we check whether a disjoint orientation-preserving band (in pink) between $e_1$ and $e_2$ can be found in cases (I) and (II), and whether a disjoint orientation-reversing band (in pink) between $e_1$ and $e_2$ can be found in case (III).
    \label{fig:cases3}}
\end{figure}

If the edges $e_1$ and $e_2$ are in the same circle in $\Gamma$, then these correspond to oriented dashes in $\mathcal{D}$ that are right next to each other (except in one case mentioned below). Thus we can draw a band between them which is disjoint from the rest of the diagram. See \autoref{fig:cases2} for some of the orientation checks. If the edges $e_1$ and $e_2$ are in different circles in $\Gamma$, then they are both connected to the central vertex. They are therefore labeled by the ``outermost'' curves/arcs in $\mathcal{D}$ and are connected to the basepoint by an arc which is disjoint from the other curves/arcs. Thickening and joining these arcs then gives a band. See \autoref{fig:cases3} for some of the orientation checks. (The case where the edges are in the same circle, but both connected to the central vertex and not sharing a second vertex is included in \autoref{fig:cases3} rather than \autoref{fig:cases2}. Specifically, see the bottom-most two pictures in \autoref{fig:cases3}.)

For the inductive step, we verify that the validity of the claim is preserved after folding has occurred. Let $\mathcal{D}$ be the diagram before the fold, $\mathcal{D}'$ be the diagram after the fold, $f_1$ and $f_2$ be the edges to be folded, and $f$ be the new folded edge. We need to show that any curves/arcs that label edges that are newly incident after the fold still have a band between them. For ease of explanation, we will slightly abuse notation and use $e_1$, $e_2$, $f_1$, $f_2$, and $f$ to also refer to the curves/arcs in the diagram that are labeled by these edges.

We first handle the case where the two folded edges $f_1$ and $f_2$ have the same label. Note that in this case $\mathcal{D} = \mathcal{D}'$. Since $f_1$ and $f_2$ have the same label, they correspond to the same curve/arc in $\mathcal{D}$, and existing bands will suffice in all cases except those in which $e_1$ and $e_2$ are connected to the non-shared vertex of $f_1$ and $f_2$ respectively, and are not incident before the fold. In these cases, observe that a band from $e_1$ to $e_2$ can be created by taking the existing band from $e_1$ to $f_1$, following along $f_1=f_2=f$, and continuing along the existing band from $f_2$ to $e_2$. See \autoref{fig:cases4} for the orientation checks.

\begin{figure}[h]
    \centering
    \includegraphics[width=1\linewidth]{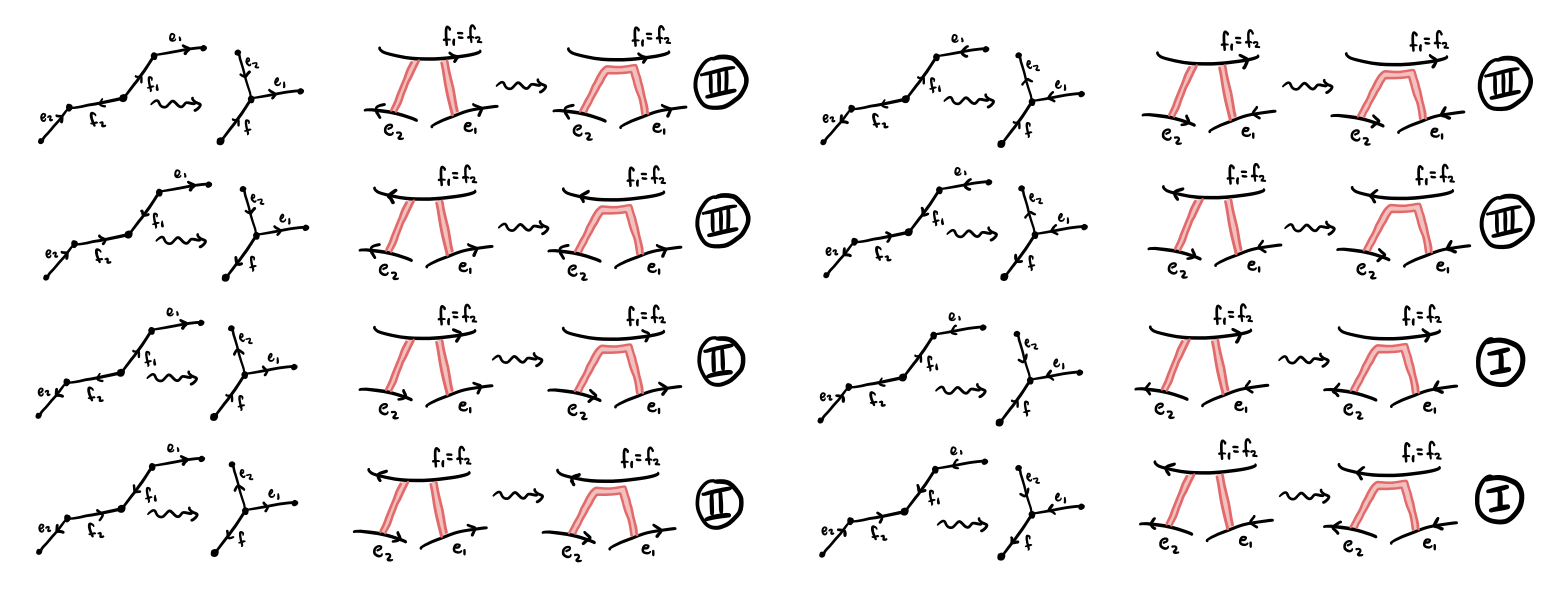}
    \caption{ 
        Orientation checks for the inductive step of the claim when the edges $f_1$ and $f_2$ have the same label. We abuse notation and use $e_1$, $e_2$, $f_1$, $f_2$, and $f$ to also refer to the curves/arcs in the diagram that correspond to these edges in $\Gamma$. These figures show the possible orientations for $e_1$, $e_2$, $f_1$, $f_2$, and $f$ relative to each other in $\Gamma$ (on the left), and how in the corresponding diagrams (on the right), a disjoint orientation-preserving band between $e_1$ and $e_2$ can be found inductively when they are oriented as in cases (I) and (II), and a disjoint orientation-reversing band between $e_1$ and $e_2$ can be found inductively when they are oriented as in case (III). These pictures suffice for all cases in which $e_1$ and $e_2$ are connected to the non-shared vertex of $f_1$ and $f_2$ respectively, and are not incident before the fold. 
    \label{fig:cases4}}
\end{figure}

Finally, assume that the two edges $f_1$ and $f_2$ have different labels. In this case, the diagram $\mathcal{D}'$ differs from the diagram $\mathcal{D}$ by an orientation-preserving band sum between $f_1$ and $f_2$, which merges these curves/arcs into the same component $f$. Therefore, as in the previous case, existing bands will suffice in all cases except those in which $e_1$ and $e_2$ are connected to the non-shared vertex of $f_1$ and $f_2$ respectively, and are not incident before the fold. In these cases, observe that a band from $e_1$ to $e_2$ can be created by taking the existing band from $e_1$ to $f_1$, following through the ``tunnel'' created by the band sum, and continuing along the existing band from $f_2$ to $e_2$. See \autoref{fig:cases5} for the orientation checks. 
\end{proof}

\begin{figure}[h]
    \centering
    \includegraphics[width=1\linewidth]{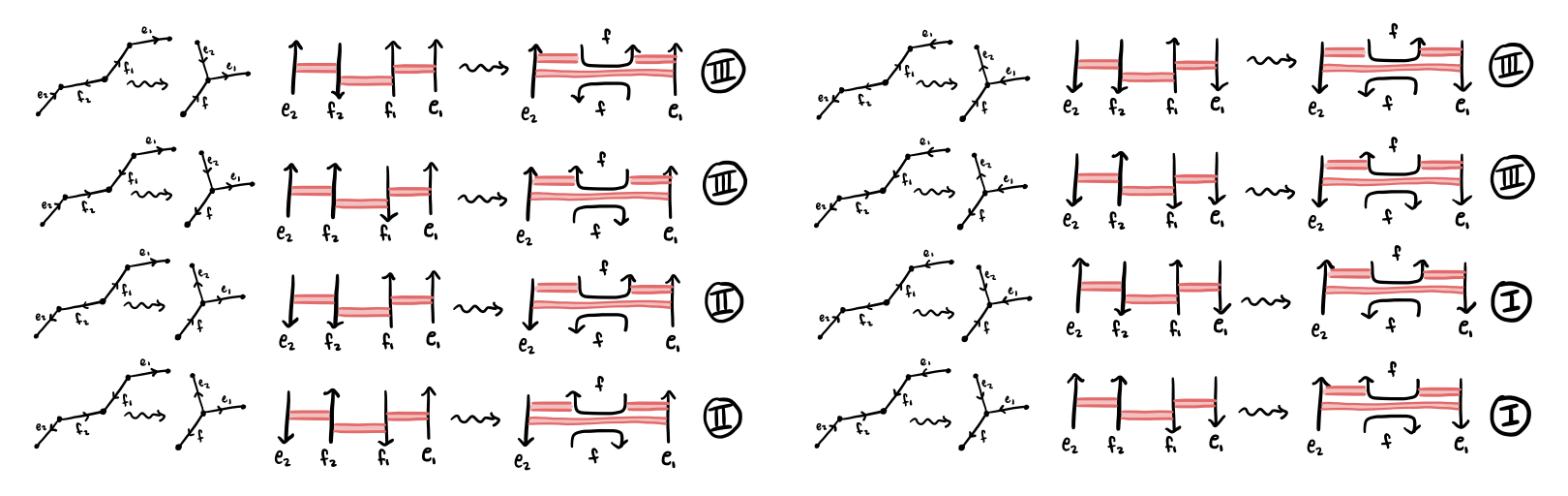}
    \caption{ 
        Orientation checks for the inductive step of the claim when the edges $f_1$ and $f_2$ have different labels. We abuse notation and use $e_1$, $e_2$, $f_1$, $f_2$, and $f$ to also refer to the curves/arcs in the diagram that correspond to these edges in $\Gamma$. These figures show the possible orientations for $e_1$, $e_2$, $f_1$, $f_2$, and $f$ relative to each other in $\Gamma$ (on the left), and how in the corresponding diagrams (on the right), a disjoint orientation-preserving band between $e_1$ and $e_2$ can be found inductively when they are oriented as in cases (I) and (II), and a disjoint orientation-reversing band between $e_1$ and $e_2$ can be found inductively when they are oriented as in case (III). These pictures suffice for all cases in which $e_1$ and $e_2$ are connected to the non-shared vertex of $f_1$ and $f_2$ respectively, and are not incident before the fold.
    \label{fig:cases5}}
\end{figure}

\begin{example}[More complicated band sums] \label{ex:bandex}
In \autoref{fig:bandex} we present an example containing more complicated band sums (compared to our running example). Here our surface is $\sphere{2}$. In the top box we start with a preliminary diagram $\mathcal{D}$ coming from a given $\phi$ (and a choice of cancellation), and from this we produce a graph $\Gamma$. The middle box shows a sequence of Stallings folds, which results in three band sums (corresponding to the highlighted folds). The bottom box shows the result of the band sums in the diagram $\mathcal{D}$. Note that in this example there is a necessary order for the band sums; $1$ and $3$ cannot be banded together until $2$ and $5$, and then $6$ and $4$, are banded together.

\begin{figure}[p]
    \centering
    \includegraphics[width=.9\linewidth]{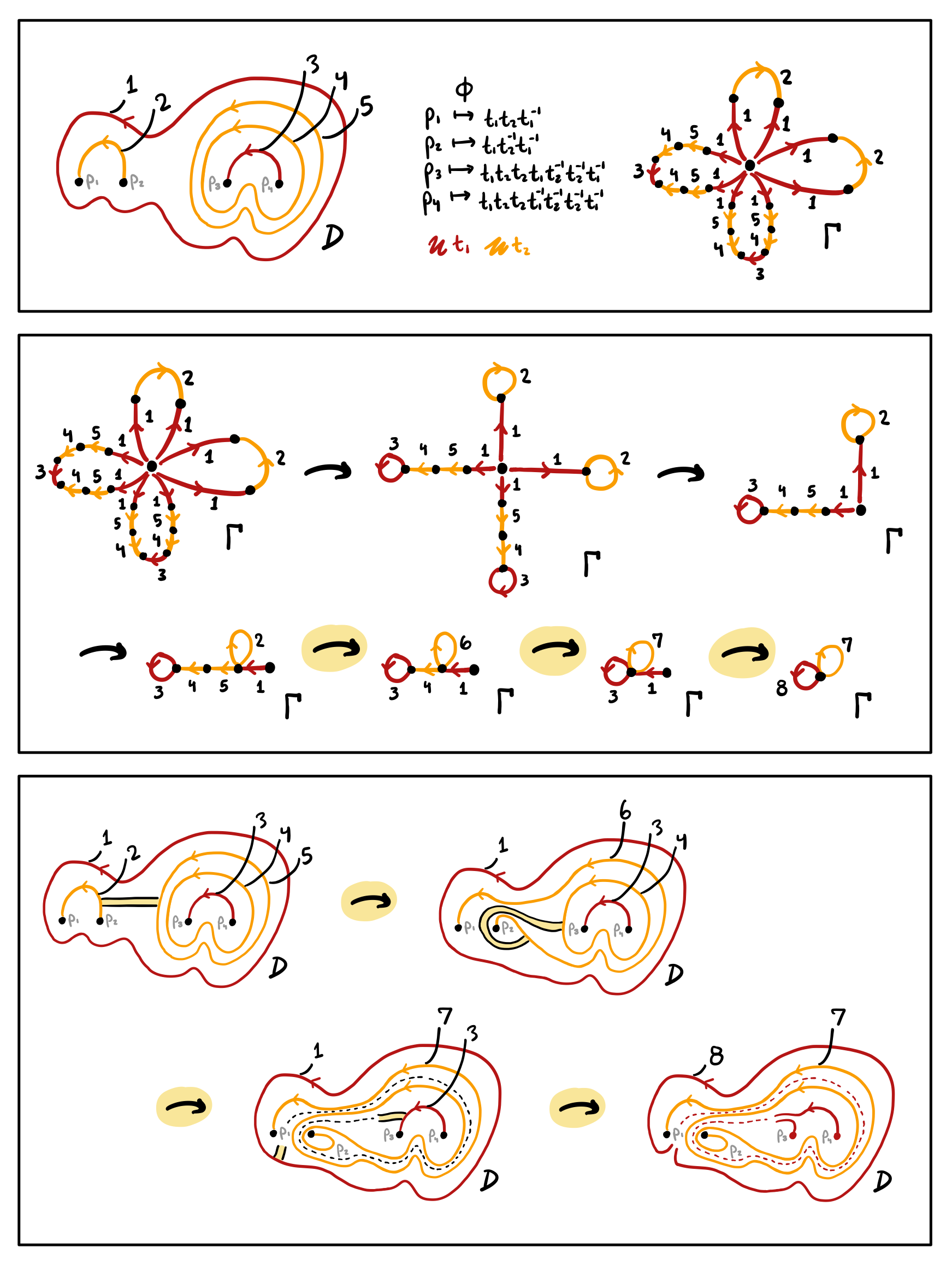}
    \caption{
        An example containing more complicated band sums. Note that in this example there is a necessary order for the band sums; $1$ and $3$ cannot be banded together until $2$ and $5$, and then $6$ and $4$, are banded together.
    \label{fig:bandex}}
\end{figure}
\end{example}

\section{Closed 3-manifolds and bridge split links}
\label{sec:3-manifolds}

In this section and the following, we translate fundamental topological theorems into the algebraic setup we described in \autoref{sec:setup}, and obtain correspondences between topology and algebra.
In each subsection we take as input a topological theorem and show how this translates to algebra.
Because we pass through diagrams in between topology and algebra and multiple notions of equivalence are involved, the proofs are rather technical.
We include full details in \autoref{sec:closed_3-manifolds} and omit some details further on, as each subsection builds from the previous and the proofs follow similarly.

As a warm-up, we begin in \autoref{sec:closed_3-manifolds}
with the case of closed 3-manifolds where our topological input theorem is the Reidemeister-Singer theorem.
Then in \autoref{sec:links_3-manifolds} we show how link theory in 3-manifolds
can be translated into the algebra of bounding homomorphisms
up to stabilization, starting from the observation that
a pair of bounding homomorphisms determines a
link in bridge position in a Heegaard split 3-manifold. We then state a correspondence theorem in this setting.

In \autoref{sec:closed_4-manifolds} we recall the
4-dimensional story of group trisections of closed 4-manifolds.
Then in \autoref{sec:links_4-manifolds}
we consider the case of surfaces inside 4-manifolds, where a triple of bounding
homomorphism with a pairwise freeness condition
determines a bridge trisected surface
in a trisected 4-manifold.

\subsection{Closed 3-manifolds}
\label{sec:closed_3-manifolds}

This section is heavily inspired
by Jaco's announcement \cite{jaco1970stable} of a result similar to \autoref{thm:Alg_3_to_Man_3}. Our topological input theorem here is the Reidemeister-Singer theorem.

\begin{ttheorem}[\cite{reidemeister1933topologie, singer1933heegaard}]
    Any two Heegaard splittings of a fixed $3$-manifold
    become isotopic after some number of
    stabilizations.
\end{ttheorem}

Let $\texttt{Man}^3$ denote the set of all
closed, connected, oriented 3-manifolds considered up to
orientation-preserving diffeomorphism
(or equivalently homeomorphism).
Let $\texttt{Alg}^3$ denote the set of pairs of homomorphisms
$(\phi_1, \phi_2)$ where $\phi_i \colon \pi_1(\Sigma_g, \ast) \twoheadrightarrow F_g $ for $i=1,2$
are surjections. We will call such pairs $(\phi_1, \phi_2)$ \emph{splitting homomorphisms} \cite{stallings1966nottoprove, jaco1969splitting}.
Given a single such surjection $\phi$, which is a bounding homomorphism for the special case where $b=0$,
using \autoref{thm:main}
we obtain a handlebody which we will denote by $H(\phi)$
such that the following diagram commutes
for an isomorphism $\psi$ as in \autoref{lem:main}.
\begin{equation*}
    \begin{tikzcd}[row sep=tiny]
        & \pi_1(H(\phi), \ast)
        \arrow{dd}{\psi}[swap]{\cong} \\
    \pi_1(\Sigma_{g}, \ast)
    \arrow[ur, "\iota", twoheadrightarrow] \arrow[dr, swap, "\phi", twoheadrightarrow] & \\
        & F_{g}
    \end{tikzcd}
\end{equation*}
By \autoref{lem:uniqueness}, $H(\phi)$ is the unique handlebody
bounding $\Sigma_g$ with the property that there exists
a vertical isomorphism $\psi$ in this diagram making it commute.  

Therefore, given $(\phi_1,\phi_2) \in \texttt{Alg}^3$ we can form two handlebodies $H(\phi_1), H(\phi_2)$ with boundary $\Sigma_g$, and thus we obtain a compact 3-manifold $M(\phi_1, \phi_2) = H(\phi_1) \cup_{\Sigma_g} - H(\phi_2)$, which is given the orientation that naturally results from gluing
the orientations of $H(\phi_1)$ and $-H(\phi_2)$. We thus have a map
\begin{align*}
    M \colon \texttt{Alg}^3 &\to \texttt{Man}^3 \\
    (\phi_1, \phi_2) &\mapsto M(\phi_1, \phi_2).
\end{align*}

We will define three relations on $\texttt{Alg}^3$, denoted $\sim_h$, $\sim_m$, and $\sim_s$, so that this map $M$ descends to the quotient of $\texttt{Alg}^3$
by these relations. Both of the relations $\sim_h$ and $\sim_m$ are actually equivalence relations on the set $\Alg^3$, while $\sim_s$ is not.
We will abuse notation and also denote by $\sim_s$ the equivalence relation on $\Alg^3$
generated by the relation $\sim_s$ (that is, the smallest equivalence relation containing $\sim_s$).
The $h$ here stands for ``handleslide,'' the $m$ for ``mapping class,'' and the $s$ for ``stabilization.'' The proof of \autoref{thm:Alg_3_1_to_Man_3_1} motivates this choice of notation. 

We say $(\phi_1, \phi_2) \sim_h (\phi_1', \phi_2')$ if for $i = 1,2$ there exist isomorphisms $h_i \colon F_g \to F_g$ such that the following diagram commutes. 
\begin{equation*}
    \begin{tikzcd}[row sep=tiny]
        & F_g \arrow{dd}{h_{i}}[swap]{\cong}\\
    \pi_1(\Sigma_{g}, \ast)
    \arrow[ur, "\phi_i", twoheadrightarrow] \arrow[dr, swap, "\phi_i'", twoheadrightarrow] & \\
        & F_{g}
    \end{tikzcd}
\end{equation*}

We call an automorphism $m \colon \pi_1(\Sigma_{g}, \ast) \to \pi_1(\Sigma_{g}, \ast)$ \emph{orientation-preserving} if the induced automorphism $H_2(\pi_1(\Sigma_{g}, \ast); \mathbb{Z}) \to H_2(\pi_1(\Sigma_{g}, \ast); \mathbb{Z})$ is the identity.  We write $(\phi_1, \phi_2) \sim_m (\phi_1', \phi_2')$ if there exists an orientation-preserving isomorphism $m \colon \pi_1(\Sigma_{g}, \ast) \to \pi_1(\Sigma_{g}, \ast)$ so that for $i = 1,2$ the following diagram commutes. 
\begin{equation*}
    \begin{tikzcd}[row sep=tiny]
        \pi_1(\Sigma_{g}, \ast) \arrow{dd}{m}[swap]{\cong} 
        \arrow[dr, "\phi_i", twoheadrightarrow] &  \\
                & F_{g} \\
        \pi_1(\Sigma_{g}, \ast) \arrow[ur, swap, "\phi_i'", twoheadrightarrow] &
    \end{tikzcd}
\end{equation*}

Next we define $\sim_s$. We note that $\phi_i'$ will be a map from $\pi_1(\Sigma_g,\ast) \twoheadrightarrow F_g$ while $\phi_i$ will be a map from $\pi_1(\Sigma_{g+1},\ast) \twoheadrightarrow F_{g+1}$. Let $a_{i}$, $b_{i}$ be the generators of $\pi_1(\Sigma_g,\ast)$ (and, abusing notation, $\pi_1(\Sigma_{g+1},\ast)$), and $h_i$ be the generators of $F_g$ (and, abusing notation, $F_{g+1}$). We say $(\phi_1, \phi_2) \sim_s (\phi_1', \phi_2')$ if $\phi_i(a_j) = \phi_i'(a_j)$ and $\phi_i(b_j) = \phi_i'(b_j)$ for $i = 1,2$ and $j= 1,\ldots,g$ (where we are identifying $F_g$ naturally as a subset of $F_{g+1}$), and the rest of the generators are mapped as follows. 
 \begin{align*}
     \phi_1(a_{g+1}) &= h_{g+1}\\ 
     \phi_1(b_{g+1}) &= 1 \\
     \phi_2(a_{g+1}) &= 1 \\ 
     \phi_2(b_{g+1}) &= h_{g+1}
 \end{align*}

Let $\sim$ denote the equivalence relation on $\Alg^3$
generated by $\sim_h$, $\sim_m$, and $\sim_s$. Now we proceed to the main result of this section. A result that is similar in spirit was announced in \cite{jaco1970stable}.

\begin{theorem}
    \label{thm:Alg_3_to_Man_3}
    The map $M \colon \Alg^3 \to \Man^3$ descends to $\Alg^3/ \sim$ and the resulting map is a bijection. 
\end{theorem}

\begin{proof}
    We consider an intermediate set $\texttt{Diag}^3$ whose elements are \emph{Heegaard diagrams}, that is, tuples $(\Sigma_g, \alpha, \beta)$
    where $\alpha$ and $\beta$ are cut systems on $\Sigma_g$ (which are only considered up to isotopy). We will refer to these simply as \emph{diagrams}.
    Then the map $M$ factors as shown below.
    \begin{equation*}
    \begin{tikzcd}
        \Alg^3
        \arrow[rr, "M"] \arrow[dr, "D"] 
        & & \Man^3  \\
        & \texttt{Diag}^3 \arrow[ur, "R"] &
    \end{tikzcd}
    \end{equation*}
    The map $R \colon \Diag^3 \to \Man^3$
    is the topological realization
    of a diagram
    $(\Sigma_g, \alpha, \beta)$, where we cross $\Sigma_g$ with an interval,
    glue disks on the respective sides to $\alpha$ and $\beta$,
    and then glue $3$-balls to the resulting sphere boundary components. The map
    $D \colon \texttt{Alg}^3 \to \texttt{Diag}^3$ is the construction of $M$
    using \autoref{thm:main}, but where we stop at just a diagram (rather than realizing the manifold) with $\alpha$ corresponding to $\phi_1$ and $\beta$ corresponding to $\phi_2$.
    We use the notation $D(\phi_1)$ and $D(\phi_2)$ to denote these cut systems so that
    $D \colon (\phi_1, \phi_2) \mapsto (\Sigma_g, D(\phi_1), D(\phi_2))$.
    
    The following commutative diagram is a guide to the logic of the proof. The goal is to define a bijection $(\Alg^3 / \sim_{h}, \sim_{m}, \sim_{s}) \to \Man^3$, so we must show that this map, which passes through an intermediate set of diagrams, is well-defined, injective, and surjective. We do this by descending by quotients on the algebraic and diagrammatic sides, and checking each time that the relevant map factors through and a bijection between the quotients is achieved.

    \begin{equation*}
    \begin{tikzcd}
        \Alg^3
        \arrow[rr, "M"] \arrow[dr, "D"] \arrow[dd]
        & & \Man^3  \\
        &
        \Diag^3 \arrow[ur, "R"]
        \arrow[d, "p_1"]
        & \\
        \Alg^3 / \sim_{h} \arrow[d] \arrow[r, "D_1"] &
        \Diag^3 / \sim_{h} \arrow[d, "p_2"]  \arrow[ruu, bend right] & \\
        \Alg^3 / \sim_{h}, \sim_{m} \arrow[d] \arrow[r, "D_2"] &
        \Diag^3 / \sim_{h}, \sim_{m} \arrow [d, "p_3"] \arrow[ruuu, bend right] & \\
        \Alg^3 / \sim_{h}, \sim_{m}, \sim_{s} \arrow[r, "D_3"] &
        (\Diag^3 / \sim_{h}, \sim_{m}) / \sim_{s} \arrow[ruuuu, bend right] & \\
    \end{tikzcd}
    \end{equation*}
    
    We abuse notation and use the symbols $\sim_h$, $\sim_m$, $\sim_s$ to denote equivalence relations on both the algebraic and diagrammatic sides. Below in \autoref{tab:notation} we summarize the notation used throughout the proof, with precise definitions for the relations on the diagrammatic side following.
    
    \begin{table}[h!]
        \centering
        \begin{tabular}{ p{2cm} p{11cm} }
        $\sim_h$ & an equivalence relation on $\Alg^3$ (as defined above) \\
        $\sim_m$ & an equivalence relation on $\Alg^3$ (as defined above) \\
        $\sim_s$ & an equivalence relation on $\Alg^3$ (as defined above) \\
        $(\phi_1, \phi_2)$ & an element of $\Alg^3$ \\
        $[\phi_1,\phi_2]$ & an equivalence class in  $\Alg^3/ \sim_h$  \\ 
        $\llbracket \phi_1, \phi_2 \rrbracket$ & an equivalence class in  $\Alg^3 / \sim_{h}, \sim_{m}$  \\ 
          & \\
        $\sim_h$ & an equivalence relation on $\Diag^3$ (generated by handleslides) \\
        $\sim_m$ & an equivalence relation on $\Diag^3$ (generated by mapping classes) \\
        $\sim_s$ & an equivalence relation on $\Diag^3 / \sim_{h}$ (generated by stabilizations) \\
        $(\Sigma_g, \alpha, \beta)$ & an element of $\Diag^3$ \\
        $[\Sigma_g, \alpha, \beta]$ & an equivalence class in  $\Diag^3/ \sim_h$  \\
        $\llbracket \Sigma_g, \alpha, \beta \rrbracket$ & an equivalence class in  $\Diag^3 / \sim_{h}, \sim_{m}$ \\ & \\
        \end{tabular}
        \caption{
        A summary of the notation used throughout the proof of \autoref{thm:Alg_3_to_Man_3}.
        \label{tab:notation}
        }
        
    \end{table}
    Given two diagrams $(\Sigma_g, \alpha, \beta)$ and $(\Sigma_g, \alpha', \beta')$, we write $(\Sigma_g, \alpha, \beta) \sim_h (\Sigma_g, \alpha', \beta')$ if there is a sequence of handleslides from the curves $\alpha$ to $\alpha'$ and similarly from $\beta$ to $\beta'$.
    We write $(\Sigma_g, \alpha, \beta) \sim_m (\Sigma_g, \alpha', \beta')$ if there exists a single mapping class $\Sigma_g \to \Sigma_g$ taking $\alpha$ to $\alpha'$ and $\beta$ to $\beta'$ simultaneously. Let $\llbracket \Sigma_g, \alpha, \beta \rrbracket$ denote an equivalence class of a diagram $(\Sigma_g, \alpha, \beta)$ under the equivalence relation generated by $\sim_h$ and $\sim_m$. 
   
    Recall that stabilizing a Heegaard diagram entails performing a connect sum with the standard genus $1$ diagram of $S^3$. In order to connect sum in a controlled manner, recall that we have fixed a standard model for the closed genus $g$ surface, and we additionally fix a disk on this model where the connect sum will be performed. See \autoref{fig:model}. Because we will mod out by handleslides and mapping classes first, we can assume our Heegaard diagram looks like the standard model. We write that $\llbracket \Sigma_{g+1}, \alpha, \beta \rrbracket \sim_s \llbracket \Sigma_g, \alpha', \beta' \rrbracket$ if we obtain $\alpha$ and $\beta$ on $\Sigma_{g+1}$ from $\alpha'$ and $\beta'$ on $\Sigma_g$ by:
    \begin{enumerate}
        \item choosing an isotopy of $\alpha'$ and $\beta'$ such that they do not intersect the connect sum disk, and
        \item modifying $\Sigma_g$ to be $\Sigma_{g+1}$ (using the fixed disk for the connect sum) and adding the two new curves in the standard genus $1$ diagram of $S^3$ to $\alpha'$ and $\beta'$.
    \end{enumerate} 
    (Note that this description incorporates both stabilization and destabilization, depending on which diagram is seen as the original and which is the modified one.)
    This operation is not well defined in $\Diag^3$ because of the choice of isotopy; for instance, see \autoref{fig:slide1}. However once we quotient by $\sim_h$ this \textit{is} well defined, as we are able to use handleslides to ``move'' the curve over the attached handle. See \autoref{fig:slide2}. Thus we write  $(\Diag^3 / \sim_{h}, \sim_{m})/ \sim_{s}$ rather than $\Diag^3 / \sim_{h}, \sim_{m}, \sim_{s}$ because unique stabilizations only occur after modding out by handleslides, and additionally we wish to assume our Heegaard diagram looks like our standard model equipped with our fixed disk.
    
    \begin{figure}
    \centering
    \includegraphics[width=.35\linewidth]{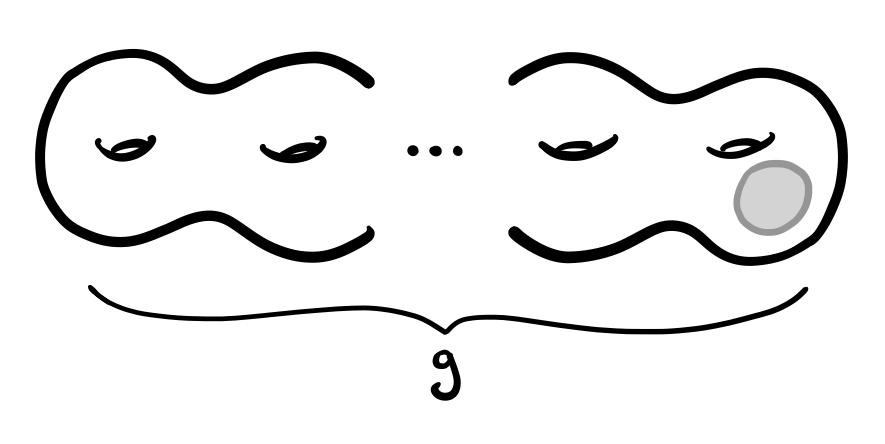}
    \caption{
    Our standard model for the closed genus $g$ surface, with a fixed disk for stabilization indicated.
    \label{fig:model}}
    \end{figure}
    
    \begin{figure}
    \centering
    \includegraphics[width=.45\linewidth]{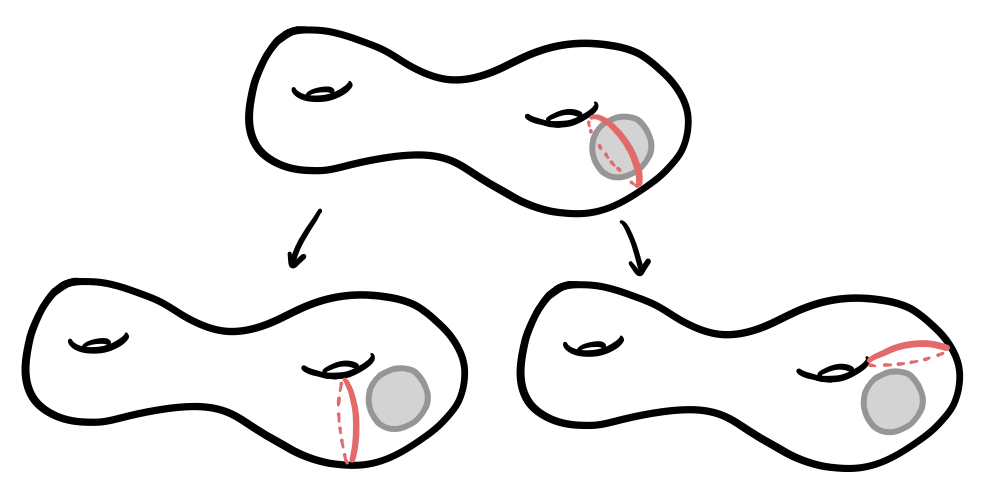}
    \caption{
    A curve which intersects the disk for stabilization, and two choices of isotopy for moving the curve off of the disk.
    \label{fig:slide1}}
    \end{figure}

    \begin{figure}
    \centering
    \includegraphics[width=1\linewidth]{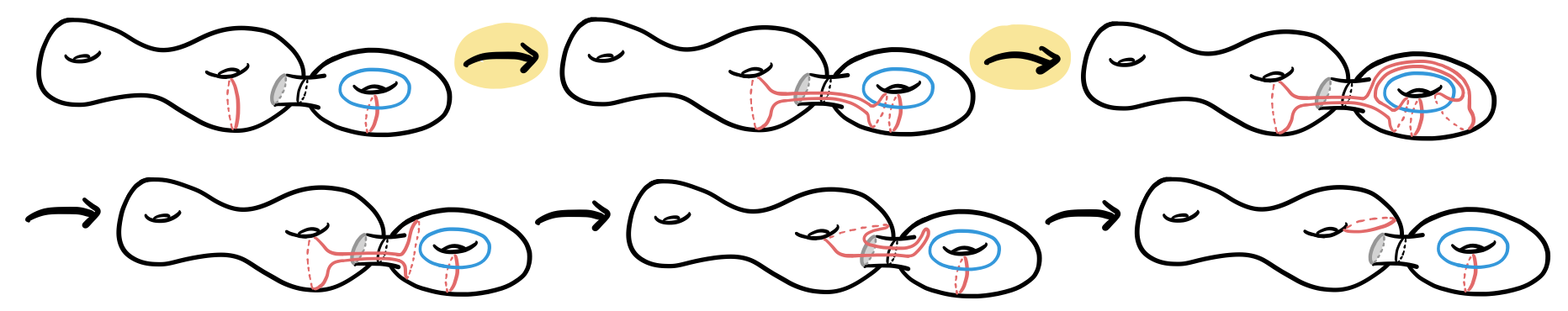}
    \caption{
    Using handleslides, denoted by the highlighted arrows, to move the curve over the attached handle.
    \label{fig:slide2}}
    \end{figure}

    \begin{claim}
        The map $p_1 \circ D$ factors through $\Alg^3/ \sim_h$
        and the resulting map
        $D_1$ is bijective.
    \end{claim}
    
    Recall that two handlebodies $H_1$ and $H_2$ bounding a given surface are equivalent if and only if their diagrams differ by handleslides \cite{johnson2006notes}.
    Given two splitting homomorphisms
    $(\phi_1, \phi_2), (\phi_1', \phi_2') \in \Alg^3$
    with $(\phi_1,\phi_2) \sim_h (\phi_1',\phi_2')$,
    by \autoref{lem:uniqueness} we see that
    $H(\phi_1) = H(\phi_1')$ and $H(\phi_2) = H(\phi_2')$.
    Therefore $(\Sigma_g, D(\phi_1),D(\phi_2)) \sim_h (\Sigma_g, D(\phi_1'), D(\phi_2'))$
    and thus the map factors through $\Alg^3/ \sim_h$, as desired.

    To see that the map $D_1 \colon \Alg^3/ \sim_h \to \Diag^3/\sim_h$ is injective,
    suppose that $D_1([\phi_1, \phi_2]) = D_1([\phi_1',\phi_2'])$ where 
    $[\phi_1,\phi_2],[\phi_1',\phi_2'] \in \Alg^3/ \sim_h$. 
    Then $H(\phi_1) = H(\phi_1')$ and $H(\phi_2) = H(\phi_2')$ and therefore by the definition of equivalence of handlebodies, we have for $i=1,2$ the commutative diagram
    \begin{equation*}
        \begin{tikzcd}[row sep=tiny]
            & \pi_1(H(\phi_i),\ast) \arrow{dd}{h_{i}}[swap]{\cong}\\
        \pi_1(\Sigma_{g}, \ast)
        \arrow[ur, twoheadrightarrow] \arrow[dr, twoheadrightarrow] & \\
            & \pi_1(H(\phi_i'),\ast)
        \end{tikzcd}
    \end{equation*}
    for some isomorphisms $h_{i}: \pi_1(H(\phi_i), \ast) \to \pi_1(H(\phi_i'), \ast)$, where the other maps are induced by inclusion.  From this, it follows that $(\phi_1,\phi_2) \sim_h (\phi_1', \phi_2')$. Thus $D_1$ is injective.

    To see that the map $D_1 \colon \Alg^3/ \sim_h \to \Diag^3/\sim_h$
    is surjective, we define a section
    \[\sigma \colon \texttt{Diag}^3/ \sim_h \to \texttt{Alg}^3/\sim_h.\]
    Let $[\Sigma_g, \alpha, \beta]$ denote the equivalence class of
    $(\Sigma_g, \alpha, \beta)$ in $\Diag^3 / \sim_h$.
    We define $\sigma([\Sigma_g, \alpha, \beta])$ by taking the diagram
    $(\Sigma_g, \alpha, \beta)$ with a particular choice of the curves
    $\alpha$ and $\beta$ (so they are no longer isotopy classes but fixed curves). 
    We then reverse the construction of the map $D$. That is, we consider each of the sets of curves $\alpha$ and $\beta$ separately and apply the construction just as in \autoref{lem:main} to obtain maps $\phi_1,\phi_2: \pi_1(\Sigma_g, \ast) \to F_g$, where here we have chosen orientations for each of the curves in $\alpha$ and $\beta$.  These maps $\phi_1, \phi_2$ are independent of the choice of representatives of the isotopy classes of the curves in $\alpha$ and $\beta$, as well as the choice of orientations, when we consider the result $[\phi_1,\phi_2]$ in $\Alg^3/\sim_h$, giving a map $\Diag^3 \to \Alg^3/\sim_h$. (Note that we have also implicitly chosen an ordering of the curves in each cut system in this construction; however because there are automorphisms of the free group permuting all of the canonical generators, this choice does not matter.) This map factors through to give a map $\sigma \colon \Diag^3/\sim_h \to \Alg^3/\sim_h$ which is a section to $D_1$ by \autoref{lem:uniqueness} together with the fact that two cut systems determine the same handlebody if and only if they differ by handleslides.
    Thus $D_1$ is surjective.

    \begin{claim}
        The map $p_2 \circ D_1$ factors through $\texttt{Alg}^3/ \sim_h, \sim_m$ and the resulting map $D_2$ is bijective.
    \end{claim}
    
    For well-definedness, suppose that $(\phi_1,\phi_2) \sim \cdots \sim (\phi_1', \phi_2')$ where each $\sim$ is either $\sim_h$ or $\sim_m$.  We must show that $(\Sigma_g, D(\phi_1), D(\phi_2)) \sim \cdots \sim (\Sigma_g, D(\phi_1'), D(\phi_2'))$ where each $\sim$ is either $\sim_h$ or $\sim_m$. By the previous step of the proof, we know that every $\sim_h$ equivalence of splitting homomorphisms produces diagrams that are equivalent with respect to $\sim_h$.  Assume $(\phi_1, \phi_2) \sim_m (\phi_1', \phi_2')$.  By the Dehn-Nielsen-Baer theorem, there exists an orientation-preserving diffeomorphism $\Sigma_g \to \Sigma_g$ that fixes the basepoint and realizes the isomorphism $\pi_1(\Sigma_g, \ast) \to \pi_1(\Sigma_g, \ast)$ that is contained in the assumption that $(\phi_1, \phi_2) \sim_m (\phi_1', \phi_2')$ \cite{farb2012primer}.  This then implies that $(\Sigma_g, D(\phi_1), D(\phi_2))$ and $(\Sigma_g, D(\phi_1'), D(\phi_2'))$ are equivalent using $\sim_h$ and $\sim_m$. Therefore, the resulting map $D_2$ is well-defined.  
    
    Assume $D_2(\llbracket \phi_1, \phi_2 \rrbracket) = D_2(\llbracket \phi_1', \phi_2' \rrbracket)$ where 
    $\llbracket \phi_1,\phi_2 \rrbracket,\llbracket\phi_1',\phi_2'\rrbracket \in \Alg^3/ \sim_h,\sim_m$.  We must show that $(\phi_1,\phi_2) \sim \cdots \sim (\phi_1', \phi_2')$ where each $\sim$ is either $\sim_h$ or $\sim_m$.  By assumption, we have that $(\Sigma_g, D(\phi_1), D(\phi_2)) \sim \cdots \sim (\Sigma_g, D(\phi_1'), D(\phi_2'))$ where each $\sim$ is either $\sim_h$ or $\sim_m$. Let $(\Sigma_g, \alpha, \beta)$ and  $(\Sigma_g, \alpha', \beta')$ be two diagrams in the above chain of relations such that $(\Sigma_g, \alpha, \beta) \sim_m (\Sigma_g, \alpha', \beta')$. Then there exists an orientation-preserving diffeomorphism 
    \begin{align*}
        F \colon \Sigma_g &\to \Sigma_g \\
            \alpha &\mapsto \alpha' \\
            \beta  &\mapsto \beta'
    \end{align*}
    which we can assume fixes the basepoint $\ast$. Let
    $(f_1, f_2) = 
    \sigma(\Sigma_g, \alpha, \beta)$ and 
    $(f_1', f_2') = \sigma(\Sigma_g, \alpha', \beta')$ where $\sigma$ is the map from the proceeding claim. 
    (Note that $\sigma$ is technically defined on $\texttt{Diag}^3/ \sim_h$, but we can similarly apply the same construction to any specific diagram with curves transverse to the generators $a_1,\ldots,b_g$ which are not considered up to isotopy.)
    Then we have for $i=1,2$ the commutative diagram
    \begin{equation*}
    \begin{tikzcd}[row sep=tiny]
        \pi_1(\Sigma_{g}, \ast) \arrow{dd}{\pi_1(F,\ast)}[swap]{\cong} 
        \arrow[dr, "f_i", twoheadrightarrow] &  \\
                & F_{g} \\
        \pi_1(\Sigma_{g}, \ast) \arrow[ur, swap, "f_i'", twoheadrightarrow] &
    \end{tikzcd}
    \end{equation*}
    for $i = 1,2$ where $\pi_1(F, \ast)$ is orientation-preserving. Therefore, we have $(f_1,f_2) \sim_m (f_1',f_2')$, so the chain of equivalences from $(\Sigma_{g}, D(\phi_{1}), D(\phi_{2}))$ to $(\Sigma_{g}, D(\phi_{1}'), D(\phi_{2}'))$ can be converted to a chain of equivalences from $(\phi_1,\phi_2)$ to $(\phi_{1}', \phi_{2}')$, and the map $D_2$ is injective. 
    
    It follows similarly that \[\sigma \colon \texttt{Diag}^3/ \sim_h \to \texttt{Alg}^3/\sim_h\] descends to a map \[\sigma \colon \texttt{Diag}^3/\sim_h, \sim_m \to \texttt{Alg}^3/\sim_h, \sim_m,\] and that it is a section for $D_2$. 

    \begin{claim}
        The map $p_3 \circ D_2$ factors through $\texttt{Alg}^3/ \sim_h, \sim_m, \sim_s$ and the resulting map $D_3$ is bijective.
    \end{claim}
    
    For well-definedness, suppose $(\phi_1, \phi_2) \sim_s (\phi_1', \phi_2')$. Then, by construction of the map $D$ we will have $\llbracket \Sigma_{g+1}, D(\phi_1), D(\phi_2) \rrbracket \sim_s \llbracket \Sigma_g, D(\phi_1'), D(\phi_2') \rrbracket$.  Well-definedness therefore follows.
    
    Similarly, by construction of $D$, if \[\llbracket \Sigma_{g+1}, D(\phi_1), D(\phi_2) \rrbracket \sim_s \llbracket \Sigma_g, D(\phi_1'), D(\phi_2') \rrbracket,\] then $\llbracket \phi_1,\phi_2 \rrbracket \sim_s \llbracket \phi_1',\phi_2' \rrbracket$, so $D_3$ is injective. If $\llbracket \Sigma_{g+1}, \alpha, \beta \rrbracket \sim_s \llbracket \Sigma_g, \alpha',\beta' \rrbracket$, then again by construction $\llbracket \sigma(\Sigma_{g+1}, \alpha, \beta) \rrbracket \sim_s \llbracket \sigma(\Sigma_{g}, \alpha',\beta') \rrbracket$, so $\sigma$ factors through to give a section of $D_3$.  

    We note that this claim follows more immediately than the previous ones since the definition of the algebraic relation $\sim_s$ is explicit in the sense that it does not involve any choices, as compared to the definitions of the algebraic relations $\sim_h$ and $\sim_m$. In later sections, we will define other notions of stabilization and they will be similarly explicit.

    \begin{claim}
        The map $R$ factors through ($\texttt{Diag}^3/ \sim_h, \sim_m)/ \sim_s$ and the resulting map is bijective.
    \end{claim}
    
Every closed, orientable 3-manifold admits a Heegaard decomposition (for example, by taking a triangulation and taking the Heegaard splitting surface to be the boundary of a regular neighborhood of the 1-skeleton). Let $Y$ be a closed, orientable 3-manifold and let $S \subset Y$ be a Heegaard splitting surface of genus $g$. Choose an identification of $S$ with $\Sigma_g$.  We have $Y = H_1 \cup_S H_2$ for two handlebodies $H_1$ and $H_2$, and by choosing collections of $g$ disjoint properly embedded disks $\mathbb{D}_1$ and $\mathbb{D}_2$ in $H_1$ and $H_2$ respectively that cut $H_1$ and $H_2$ into a $3$-ball, then by looking at $(S, \partial \mathbb{D}_1, \partial \mathbb{D}_2)$, we have a diagram whose topological realization is $Y$. (Here the topological realization is as before; thicken $S$, glue disks to $\mathbb{D}_1$ and $\mathbb{D}_2$ on their respective sides, and glue in 3-balls to the resulting spheres.) Using the identification of $S$ with $\Sigma_g$, we obtain a diagram $(\Sigma_g, \alpha, \beta) \in \texttt{Diag}^3$ where $\alpha$ and $\beta$ are the respective images of $\partial \mathbb{D}_1$ and $\partial \mathbb{D}_2$, and the image of $(\Sigma_g, \alpha, \beta)$ in $\texttt{Man}^3$ is $Y$.  Therefore the map $R: \texttt{Diag}^3 \to \texttt{Man}^3$ is surjective.  

The factored through map $R: (\texttt{Diag}^3/ \sim_h, \sim_m)/ \sim_s \to \Man^3$ is injective by the Reidemeister-Singer theorem, and hence a bijection. By composing $D_3$ with this map, we obtain the theorem.
\end{proof}

\subsection{Bridge split links in 3-manifolds}
\label{sec:links_3-manifolds}

In \autoref{sec:closed_3-manifolds} we used that any pair of Heegaard splittings of the same fixed $3$-manifold become isotopic after some number of stabilization operations (which corresponds to connect summing with the genus $1$ splitting of the $3$-sphere).
The goal of this section will be translating the corresponding uniqueness up to perturbation statement for bridge splittings of links in $3$-manifolds into the algebra of bounding homomorphisms.

\begin{ttheorem}[\cite{hayashi1998stable,zupan2013bridgecomplexity}]
    Let $L$ be a link in a fixed Heegaard split $3$-manifold.
    Then any two bridge splittings of $L$ become isotopic after some number of perturbations.
\end{ttheorem}

Let $\Man^{(3,1)}$ denote the set of closed, connected, oriented 3-manifolds $M$ together with a link $L \subset M$ modulo orientation-preserving diffeomorphisms preserving the links. Let $\Alg^{(3,1)}$ denote the set of pairs $(\phi_1,\phi_2)$ such that $\phi_1,\phi_2 \colon \pi_1(\Sigma_g - \{p_1,\ldots,p_{2b} \}, \ast) \twoheadrightarrow F_{g+b}$ are bounding homomorphisms. Throughout this section, let $a_{j}$, $b_{j}$, and $p_{k}$ denote the generators of $\pi_1(\Sigma_g - \{p_1,\ldots,p_{2b} \}, \ast)$, where $a_{j}$, $b_{j}$ are the surface generators (for $j=1,\ldots,g$) and $p_{k}$ are the puncture generators (for $k=1,\ldots,2b$), and let $h_j$, $t_{\ell}$ denote the generators of $F_{g+b}$ (for $\ell=1,\ldots,b$). 

We define the map $(M, L) \colon \Alg^{(3,1)} \to \Man^{(3,1)}$ as follows. Given $(\phi_1,\phi_2) \in \Alg^{(3,1)}$, let $H(\phi_1)$ and $H(\phi_2)$ be the handlebodies bounding $\Sigma_g$ that result from the application of \autoref{thm:main}, and further let $T(\phi_1) \subset H(\phi_1)$ and $T(\phi_2) \subset H(\phi_2)$ be the resulting trivial tangles in these handlebodies.  We then define $(M,L)(\phi_1,\phi_2)$ to be the 3-manifold $H(\phi_1) \cup_{\Sigma_g} H(\phi_2)$ (with the orientation as in \autoref{sec:closed_3-manifolds}) together with the link $T(\phi_1) \cup_{\{p_1,\ldots,p_{2b} \}} T(\phi_2)$.

We now define the analogues of the equivalence relations $\sim_h,\sim_m$, and $\sim_s$ on $\Alg^{3}$ in this setting. Given $(\phi_1,\phi_2), (\phi_1', \phi_2') \in \Alg^{(3,1)}$, we write $(\phi_1, \phi_2) \sim_h (\phi_1', \phi_2')$ if for $i = 1,2$ there exist isomorphisms $h_i \colon F_{g+b} \to F_{g+b}$ such that the following diagram commutes. 
\begin{equation*}
    \begin{tikzcd}[row sep=tiny]
        & F_{g+b} \arrow{dd}{h_{i}}[swap]{\cong}\\
    \pi_1(\Sigma_{g}- \{p_1,\ldots,p_{2b}\}, \ast)
    \arrow[ur, "\phi_i", twoheadrightarrow] \arrow[dr, swap, "\phi_i'", twoheadrightarrow] & \\
        & F_{g+b}
    \end{tikzcd}
\end{equation*}

Let $m \colon \pi_1(\Sigma_g - \{p_1,\ldots,p_{2b} \}, \ast) \to \pi_1(\Sigma_g - \{p_1,\ldots,p_{2b} \}, \ast)$ be an automorphism that preserves the conjugacy classes of $p_1,\ldots,p_{2b}$ setwise.  Then $m$ descends to an automorphism $\pi_1(\Sigma_g, \ast) \to \pi_1(\Sigma_g,\ast)$, by the surjective map $\pi_1(\Sigma_g - \{p_1,\ldots,p_{2b} \}, \ast) \twoheadrightarrow \pi_1(\Sigma_g, \ast)$ which sends each $p_i$ to the identity and is the identity on all of the elements $a_1,b_1,\ldots,a_g,b_g$. We call $m$ \emph{orientation-preserving} if this corresponding automorphism $\pi_1(\Sigma_g, \ast) \to \pi_1(\Sigma_g,\ast)$ is orientation-preserving. Given $(\phi_1,\phi_2), (\phi_1', \phi_2') \in \Alg^{(3,1)}$, we write $(\phi_1,\phi_2) \sim_m (\phi_1',\phi_2')$ if there exists an orientation-preserving isomorphism $m \colon \pi_1(\Sigma_{g} - \{p_1,\ldots,p_{2b}\}, \ast) \to \pi_1(\Sigma_{g}- \{p_1,\ldots,p_{2b}\}, \ast)$ so that for $i = 1,2$ the following diagram commutes. 
\begin{equation*}
    \begin{tikzcd}[row sep=tiny]
        \pi_1(\Sigma_{g}- \{p_1,\ldots,p_{2b}\}, \ast) \arrow{dd}{m}[swap]{\cong} 
        \arrow[dr, "\phi_i", twoheadrightarrow] &  \\
                & F_{g+b} \\
        \pi_1(\Sigma_{g}- \{p_1,\ldots,p_{2b}\}, \ast) \arrow[ur, swap, "\phi_i'", twoheadrightarrow] &
    \end{tikzcd}
\end{equation*}

While $\sim_h$ and $\sim_m$ as defined in this section are very similar to $\sim_h$ and $\sim_m$ as defined in \autoref{sec:closed_3-manifolds}, the analogue of $\sim_s$ is a bit more complicated.  We will have one such relation $\sim_{s_g}$ which is directly analogous to $\sim_s$ in \autoref{sec:closed_3-manifolds}; namely it captures the idea of increasing the genus of the Heegaard splitting while leaving everything else fixed.  In addition, there are two relations $\sim_{s_b^1}$ and $\sim_{s_b^2}$ which will correspond to the idea of modifying a link in bridge position by perturbation.

Given $(\phi_1,\phi_2), (\phi_1', \phi_2') \in \Alg^{(3,1)}$, we now define $\sim_{s_g}$.  We note that $\phi_i'$ will be a map from $\pi_1(\Sigma_g - \{p_1,\ldots,p_{2b} \},\ast) \twoheadrightarrow F_{g+b}$ while $\phi_i$ will be a map from $\pi_1(\Sigma_{g+1}-\{p_1,\ldots,p_{2b} \},\ast) \twoheadrightarrow F_{g+1+b}$. We say $(\phi_1, \phi_2) \sim_{s_g} (\phi_1', \phi_2')$ if $\phi_i(a_j) = \phi_i'(a_j)$, $\phi_i(b_j) = \phi_i'(b_j)$, and $\phi_i(p_k) = \phi_i'(p_k)$ for $i = 1,2$, $j= 1,\ldots,g$, and $k = 1,\ldots,2b$ (where we are identifying $F_{g+b}$ naturally as a subset of $F_{g+1+b}$, identifying the $h_i$ generators in $F_{g+b}$ with $h_i$ in $F_{g+1+b}$ and similarly with $t_i$), and the rest of the generators are mapped as follows. 
 \begin{align*}
     \phi_1(a_{g+1}) &= h_{g+1}\\ 
     \phi_1(b_{g+1}) &= 1 \\
     \phi_2(a_{g+1}) &= 1 \\ 
     \phi_2(b_{g+1}) &= h_{g+1}
 \end{align*}

Finally, we define $\sim_{s_b^1}$ and $\sim_{s_b^2}$, which correspond to perturbation of the tangle strands. See \autoref{fig:stab1}. Suppose now that $b > 0$.  There are two such operations since we can either push a tangle strand from the side corresponding to $\phi_1$ across $\Sigma_g$, or we can push a tangle strand corresponding to $\phi_2$. The motivation for this comes from investigating \autoref{fig:stab1} and imagining applying the operation $\sigma$ from the proof of \autoref{thm:Alg_3_to_Man_3} to the before and after parts of the figure. Let $(\phi_1,\phi_2), (\phi_1', \phi_2') \in \Alg^{(3,1)}$. We note that $\phi_i'$ will be a map from $\pi_1(\Sigma_g - \{p_1,\ldots,p_{2b} \},\ast) \twoheadrightarrow F_{g+b}$ while $\phi_i$ will be a map from $\pi_1(\Sigma_g-\{p_1,\ldots,p_{2b},p_{2b+1},p_{2b+2}\},\ast) \twoheadrightarrow F_{g+b+1}$. Assume without loss of generality that $\phi_1'(p_{2b})=\phi_2'(p_{2b})=t_b$. (We can do this because we first mod out by mapping class group elements; see the proof of \autoref{thm:Alg_3_to_Man_3}, claim 1, for details.)
We write $(\phi_1,\phi_2) \sim_{s_b^1} (\phi_1',\phi_2')$ if
  \begin{align*}
     \phi_1(p_{2b}) &= t_{b+1}\\ 
     \phi_1(p_{2b+1}) &= (t_{b+1})^{-1} \\
     \phi_1(p_{2b+2}) &=  t_b \\
   \phi_2(p_{2b}) &= t_b \\ 
     \phi_2(p_{2b+1}) &= t_{b+1} \\ 
     \phi_2(p_{2b+2}) &= (t_{b+1})^{-1}
 \end{align*}
and $\phi_i$ and $\phi_i'$ agree for all other elements in the generating sets (suitably identifying the groups) for $i = 1,2$.  We similarly define $\sim_{s_b^2}$ by swapping the roles of the indices $1$ and $2$. 

\begin{figure}
    \centering
    \includegraphics[width=.9\linewidth]{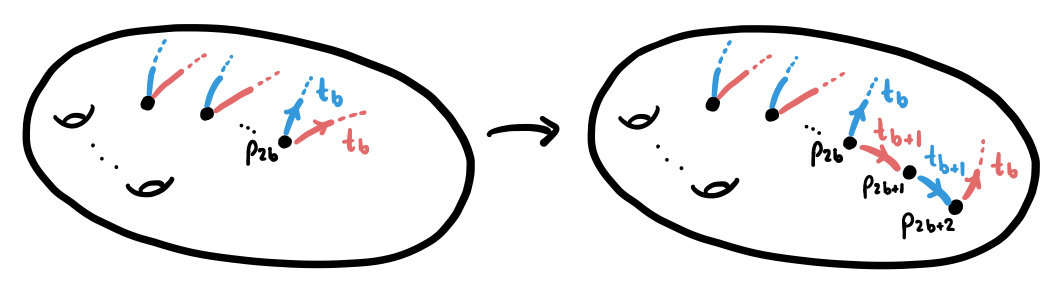}
    \caption{Performing a perturbation, where one of the pink tangle strands is being pulled through the surface. This adds a tangle strand to each side and increases the number of punctures by two. If $\phi_1$ corresponds to pink and $\phi_2$ corresponds to blue, then this is a picture of the $\sim_{s_b^1}$ version of perturbation. Switching the colors would give a picture for the $\sim_{s_b^2}$ version. 
    \label{fig:stab1}}
\end{figure}

Let $\sim$ denote the equivalence relation on $\Alg^{(3,1)}$ generated by $\sim_h, \sim_m, \sim_{s_g}, \sim_{s_b^1}$ and $\sim_{s_b^2}$.  
 
\begin{theorem}
    \label{thm:Alg_3_1_to_Man_3_1}
    The map
    $(M, L) \colon \Alg^{(3,1)} \to \Man^{(3,1)}$
    descends to $\Alg^{(3,1)}/ \sim$ and the resulting map is a bijection.  
\end{theorem}

\begin{proof}
As in the proof of \autoref{thm:Alg_3_to_Man_3}, we consider an intermediate set $\texttt{Diag}^{(3,1)}$ whose elements are \emph{diagrams}, that is, tuples $(\Sigma_g, \alpha, \beta, S_{\alpha}, S_{\beta})$
where $\alpha,\beta$ are cut systems on $\Sigma_g$ and $S_{\alpha}, S_{\beta}$ are shadow diagrams for trivial tangles with endpoints $\{p_1, \ldots, p_{2b}\}$ (which are all only considered up to isotopy). In other words, $(\alpha, S_{\alpha})$ is a curve-and-arc system for one tangle and handlebody, and $(\beta, S_{\beta})$ is a curve-and-arc system for the other.
(Refer to the beginning of \autoref{sec:setup} for the definition of \textit{curve-and-arc system}.) Then the map $(M,L)$ factors as shown below.
    \begin{equation*}
    \begin{tikzcd}
        \Alg^{(3,1)}
        \arrow[rr, "{(M,L)}"] \arrow[dr, "D"] 
        & & \Man^{(3,1)}  \\
        & \texttt{Diag}^{(3,1)} \arrow[ur, "R"] &
    \end{tikzcd}
    \end{equation*}
The map $R \colon \Diag^{(3,1)} \to \Man^{(3,1)}$ is the topological realization of a diagram $(\Sigma_g, \alpha, \beta, S_{\alpha}, S_{\beta})$, where we cross $\Sigma_g$ with an interval, glue disks on the respective sides to $\alpha$ and $\beta$, glue three balls to the resulting sphere boundary components to obtain two handlebodies, and then push the interiors of the shadow arcs in $S_{\alpha}$ and $S_{\beta}$ into their respective handlebody to obtain two tangles. The map $D \colon \texttt{Alg}^{(3,1)} \to \texttt{Diag}^{(3,1)}$ is the construction of $(M,L)$
using \autoref{thm:main}, but where we stop at just a diagram with curves $\alpha$ and arcs $S_{\alpha}$ corresponding to $\phi_1$ and curves $\beta$ and arcs $S_{\beta}$ corresponding to $\phi_2$.

In the proof of \autoref{thm:Alg_3_to_Man_3} (see claim 1) we used the fact that two handlebodies bounding a given surface are equivalent if and only if their diagrams differ by handleslides. In this proof the following fact will take its place: two tangles in a handlebody are isotopic (fixing their boundary points) if and only if their curve-and-arc systems are related by a sequence of isotopies and slides.  This folklore fact appears, for instance, in \cite[Prop.~3.1]{meier2018bridge4manifolds} and \cite[Prop.~5.2]{meier2020filling}, both citing \cite{johannson1995topology} for the proof idea. Our topological input theorem can now be translated into the following diagrammatic statement: two diagrams of the same link in a $3$-manifold are related by a sequence of perturbations, depurtubations, and moves from the above fact. In other words, these are diagrammatic equivalence relations which have the algebraic counterparts $\sim_h, \sim_m, \sim_{s_g}, \sim_{s_b^1}$ and $\sim_{s_b^2}$ as described above.

The rest of the proof then follows in similar fashion as before; mod out the map $D$ by these diagrammatic equivalence relations, and then show each time that the map factors through and a bijection between quotients is achieved. We leave the details to the reader. Then we use both our topological input theorem from this section and the Reidemeister-Singer theorem (as previously) to achieve the result.
\end{proof}

\begin{remark} \label{rem:components}
    \autoref{thm:Alg_3_1_to_Man_3_1} is in fact a generalization of \autoref{thm:Alg_3_to_Man_3} where we take the links to be empty.  Note also that the number of components of the link resulting from a given pair $(\phi_1,\phi_2)$ can easily be read off from the maps $\phi_1$ and $\phi_2$.  We can thereby describe the partition of the set $\Alg^{(3,1)}$ so that the different equivalence classes correspond in the bijection above to manifolds with links of $0,1,2\ldots$ components (in particular, $\sim$ respects this partition).  In particular, \autoref{thm:Alg_3_to_Man_3} is recovered by restricting to the case where $b = 0$.
\end{remark}

\section{Closed 4-manifolds and bridge trisected surfaces}
\label{sec:4-manifolds}

Here we continue the translation of topology into algebra that we started in the previous section, but moving up a dimension. The proofs in this section follow similarly to those in the previous, and thus we omit some of the details.

\subsection{Closed 4-manifolds}
\label{sec:closed_4-manifolds}

In 2016 Gay and Kirby introduced \emph{trisections} of $4$-manifolds, which can be seen as 4-dimensional analogues of Heegaard splittings.

\begin{definition}[Trisection of a $4$-manifold \cite{gay2016trisecting}]
A \emph{$(g;k_1,k_2,k_3)$-trisection} of a smooth, closed, connected, oriented $4$-manifold $X^4$ is a decomposition $X^4=X_1 \cup X_2 \cup X_3$ with the following properties.
\begin{enumerate}
\item Each $X_i$ is a $4$-dimensional $1$-handlebody, that is, diffeomorphic to $\natural^{k_i} (S^1 \times B^3)$. If $k_i=0$, we interpret this boundary connected sum as $B^4$.
\item The $X_i$'s intersect pairwise in genus $g$ handlebodies; that is, their pairwise intersections are diffeomorphic to $H_g \vcentcolon = \natural^g (S^1 \times B^2)$.
\item The triple intersection of the $X_i$'s is a genus $g$ oriented surface (denoted $\Sigma_g$), called the \emph{central surface}.
\end{enumerate}
If $k_1=k_2=k_3 = \vcentcolon k$, we call the trisection \textit{balanced}, and denote this by $(g;k)$.
\end{definition}

A trisection is determined by its \emph{spine}, namely the central surface along with the three handlebodies it bounds, as there is a unique way, up to diffeomorphism, to fill this in with $4$-dimensional $1$-handlebodies \cite{laudenbach1972note}. Every smooth, closed, connected, oriented $4$-manifold admits a trisection, and a trisection of a given $4$-manifold is unique up to stabilization (see \cite{gay2016trisecting} for a proof of this result and the precise definition of stabilization). 

\begin{ttheorem}[\cite{gay2016trisecting}]
    Any pair of trisections of a fixed
    $4$-manifold become isotopic after
    some number of stabilizations.
\end{ttheorem}

Technically, there are two notions of stabilization: a balanced one (which increases the genus of the central surface by three) and an unbalanced one (which increases the genus by one). We will assume all trisections are balanced and work only with balanced stabilizations. We can make this assumption without loss of generality as any unbalanced trisection can be made balanced with some number of unbalanced stabilizations. Note that we could easily expand the algebraic relations included in this section to include those which would correspond to unbalanced stabilizations, but for the sake of simplicity of notation we have not done this.

Let $\Man^4$ denote the set of closed, connected, oriented, smooth 4-manifolds up to orientation-preserving diffeomorphism. Let $\Alg^4$ denote the set of triples $(\phi_1, \phi_2, \phi_3)$ of surjective homomorphisms $\pi_1(\Sigma_g, \ast) \twoheadrightarrow F_g$ for some integer $g$ such that $M(\phi_1, \phi_2), M(\phi_2,\phi_3),$ and $M(\phi_3,\phi_1)$ are all diffeomorphic to $\#^k (S^1 \times S^2)$ for some integer $k$, where $M$ is as in \autoref{sec:closed_3-manifolds}. By the Poincar\'{e} conjecture \cite{perelman2003finite}, this is equivalent to the property that the pushout of $\phi_i$ and $\phi_j$ for $i \neq j$ are (necessarily finitely-generated) free groups. Note that here, an individual $\phi$ is a bounding homomorphism for the special case $b=0$ which satisfies the additional pushout property, and a triple $(\phi_1, \phi_2, \phi_3) \in \Alg^4$ is a \textit{group trisection} of the fundamental group of a $4$-manifold, as described in \cite{abrams2018group} and \autoref{sec:intro}.

We have a map $X \colon \Alg^4 \to \Man^4$ just as in \cite{abrams2018group}. Namely, given $(\phi_1,\phi_2,\phi_3) \in \Alg^4$, identify the handlebodies $H(\phi_1),H(\phi_2),H(\phi_3)$ along their common boundary $\Sigma_g$. Now the three 3-manifolds $M(\phi_i,\phi_j) = H(\phi_i) \cup_{\Sigma_g} H(\phi_j)$ for $i \neq j$ are all diffeomorphic to some $\#^k (S^1 \times S^2)$ by the assumption that the pairwise pushouts of the $\phi_i$ and $\phi_j$ are free groups.  Therefore we may glue in three 4-dimensional 1-handlebodies (uniquely by \cite{laudenbach1972note}) to obtain a smooth, closed 4-manifold $X(\phi_1,\phi_2,\phi_3)$.  We orient $M(\phi_1,\phi_2) \subset X(\phi_1,\phi_2,\phi_3)$ as in \autoref{sec:closed_3-manifolds} and we orient $X$ so that the orientation restricted to the 4-dimensional 1-handlebody that is glued to $M(\phi_1,\phi_2)$ induces this orientation on $M(\phi_1,\phi_2)$.  

We define the relations $\sim_h$ and $\sim_m$ in an analogous fashion as was done for $\Alg^3$ in \autoref{sec:closed_3-manifolds}. The stabilization relation $\sim_s$ on $\Alg^4$ is defined as follows. Here $\phi_i'$ will be a map from $\pi_1(\Sigma_g,\ast) \twoheadrightarrow F_g$ while $\phi_i$ will be a map from $\pi_1(\Sigma_{g+3},\ast) \twoheadrightarrow F_{g+3}$. Let $a_{i}$, $b_{i}$ be the generators of $\pi_1(\Sigma_g,\ast)$ (and, abusing notation, $\pi_1(\Sigma_{g+3},\ast)$), and $h_i$ be the generators of $F_g$ (and, abusing notation, $F_{g+3}$). We say $(\phi_1, \phi_2, \phi_3) \sim_s (\phi_1', \phi_2', \phi_3')$ if $\phi_i(a_j) = \phi_i'(a_j)$ and $\phi_i(b_j) = \phi_i'(b_j)$ for $i = 1,2,3$ and $j= 1,\ldots,g$ (where we are identifying $F_g$ naturally as a subset of $F_{g+3}$), and the rest of the generators are mapped as follows.
 \begin{align*}
     \phi_1(a_{g+1}) &= h_{g+1} & \phi_2(a_{g+1}) &= h_{g+1} & \phi_3(a_{g+1}) &= 1\\ 
     \phi_1(b_{g+1}) &= 1 & \phi_2(b_{g+1}) &= 1 & \phi_3(b_{g+1}) &= h_{g+1} \\
     \phi_1(a_{g+2}) &= h_{g+2} & \phi_2(a_{g+2}) &= 1 & \phi_3(a_{g+2}) &= h_{g+2}\\ 
     \phi_1(b_{g+2}) &= 1 & \phi_2(b_{g+2}) &= h_{g+2} & \phi_3(b_{g+2}) &= 1 \\
     \phi_1(a_{g+3}) &= 1 & \phi_2(a_{g+3}) &= h_{g+3} & \phi_3(a_{g+3}) &= h_{g+3}\\ 
     \phi_1(b_{g+3}) &= h_{g+3} & \phi_2(b_{g+3}) &= 1 & \phi_3(b_{g+3}) &= 1
 \end{align*}
 
This is an algebraic analogue of the topological operation of stabilizing a trisection, just as in \autoref{sec:closed_3-manifolds} where we presented the analogous notion for stabilizations of Heegaard splittings. As before, we let the equivalence relation $\sim_s$ be the symmetrization of the relation $\sim_s$.

As we now have three maps, we must define one more relation which does not have a counterpart in \autoref{sec:3-manifolds}, corresponding to cyclically permuting the ``colors'' of the curves on a trisection diagram (that is, cyclically permuting the roles of the cut system curves $\alpha$, $\beta$, and $\gamma$). We say  $(\phi_1, \phi_2, \phi_3) \sim_c (\phi_1', \phi_2', \phi_3')$ if $\phi_1=\phi'_2$, $\phi_2=\phi'_3$ , and $\phi_3=\phi'_1$.

Let $\sim$ denote the equivalence relation on $\Alg^4$ generated by $\sim_h$, $\sim_m$, $\sim_s$, and $\sim_c$. A result very similar to the following theorem is stated in \cite[Thm.~5]{abrams2018group}.

\begin{theorem}[compare with \cite{abrams2018group}]
    \label{thm:Alg_4_to_Man_4}
    The map
    $X \colon \texttt{\textup{Alg}}^{4} \to \texttt{\textup{Man}}^{4}$
    descends to $\texttt{\textup{Alg}}^{4}/ \sim$
    and the resulting map is a bijection.  
\end{theorem}

\begin{proof}
    The proof proceeds analogously to the discussion
    of the Heegaard splittings of closed 3-manifolds,
    where the role of the Reidemeister-Singer theorem
    is taken on by the uniqueness of trisections up to
    stabilization from
    \cite{gay2016trisecting}.
\end{proof}

\begin{remark} \label{rem:algorithmic}
Note that we can algorithmically determine if the pushout of such a pair $\phi_i$ and $\phi_j$ is in fact a free group.  Namely, we can construct the corresponding 3-manifold and algorithmically check if it is a (possibly empty) connected sum of copies of $S^1 \times S^2$ (see for example \cite{kuperberg2019algorithmic}).  We are unaware if there is a more direct algebraic method to verify this condition. 
\end{remark}

\subsection{Bridge trisected surfaces in 4-manifolds}
\label{sec:links_4-manifolds}

Finally we turn to the setting of knotted surfaces in $4$-manifolds. For us, a \emph{knotted surface} is a closed surface smoothly embedded in a smooth, closed, connected, oriented 4-manifold. 
In particular, our surfaces are not necessarily orientable or connected.

In 2018 Meier and Zupan showed that knotted surfaces in trisected $4$-manifolds can always be isotoped into a compatibly trisected surface. This inherited decomposition is called a \emph{bridge trisection}.

\begin{definition}[Bridge trisection of a knotted surface \cite{meier2017bridgeS4,meier2018bridge4manifolds}] \label{def:bridge}
A \emph{$(g;k_1,k_2,k_3;b;c_1,c_2,c_3)$-bridge trisection} of a knotted surface $S$ in a $4$-manifold $X^4$ is a decomposition $(X^4,S)=(X_1,\mathcal{D}_1) \cup (X_2,\mathcal{D}_2) \cup (X_3,\mathcal{D}_3)$ with the following properties.
\begin{enumerate}
\item The decomposition $X^4=X_1 \cup X_2 \cup X_3$ is a $(g;k_1,k_2,k_3)$-trisection of $X^4$.
\item Each $\mathcal{D}_i$ is a boundary parallel collection of $c_i$ disks in $X_i$.
\item The $\mathcal{D}_i$'s intersect pairwise in trivial $b$-strand tangles (denoted $T_b$) in the handlebodies which are the pairwise intersections of the $X_i$'s.
\end{enumerate}
Just as before, if $k_1=k_2=k_3 = \vcentcolon k$, we replace these parameters with just one $k$ and call the trisection \textit{balanced}. We assume again that all trisections are balanced (with respect to the $k_i$ parameter; the $c_i$ may be different).
\end{definition}

Pairwise, the tangles $T_b$ form unlinks. A bridge trisection is determined by these three tangles, as there is a unique way to smoothly cap off unlinks in  $\#^{k} (S^1 \times S^2)$ with disks \cite{meier2018bridge4manifolds}. A knotted surface in a trisected $4$-manifold is said to be in \emph{bridge position} if its intersection with the trisection of the $4$-manifold results in a bridge trisection. Knotted surfaces can always be isotoped to be in bridge position, and this is unique up to perturbation (for the proof of this result and the precise definition of pertubation, see \cite{meier2017bridgeS4,meier2018bridge4manifolds,hughes2018isotopies}).

\begin{ttheorem}[\cite{meier2017bridgeS4,meier2018bridge4manifolds,hughes2018isotopies}]
    Any pair of bridge trisections for a smoothly
    embedded surface in a fixed underlying trisection of a
    $4$-manifold become isotopic
    after some number of
    perturbations.
\end{ttheorem}

Throughout this section, let $a_{j}$, $b_{j}$, and $p_{k}$ denote the generators of $\pi_1(\Sigma_g - \{p_1,\ldots,p_{2b} \}, \ast)$, where $a_{j}$, $b_{j}$ are the surface generators (for $j=1,\ldots,g$) and $p_{k}$ are the puncture generators (for $k=1,\ldots,2b$), and let $h_j$, $t_{\ell}$ denote the generators of $F_{g+b}$ (for $\ell=1,\ldots,b$).
Given a bounding homomophism 
    \begin{equation*}
        \phi \colon
        \pi_1(\Sigma_{g} - \{p_1,\ldots,p_{2b} \}, \ast)
        \twoheadrightarrow
        \langle t_{1}, \ldots, t_{b},
        h_{1}, \ldots, h_{g} \rangle
    \end{equation*}
we have an associated homomorphism which we call the \emph{associated closed bounding homomorphism} for $\phi$, denoted by
\begin{equation*}
    \overline{\phi} : \pi_1(\Sigma_g, \ast) \twoheadrightarrow \langle h_{1}, \ldots, h_{g} \rangle,
\end{equation*}
which is given by postcomposing $\phi$ by the map quotienting out all of the $t_{\ell}$ (and again calling the images of the $h_j$ by $h_j$) and then sending all $a_j$ and $b_j$ (now thought of as in $\pi_1(\Sigma_g,\ast)$) to the resulting elements of the free group generated by the $h_j$.  Note that since $\phi$ is a bounding homomorphism, $\overline{\phi}$ is also. 

Let $\Man^{(4,2)}$ denote the set of closed, connected, oriented, smooth 4-manifolds $X$ together with a union of closed (potentially non-orientable or disconnected) surfaces $S \subset X$ modulo orientation-preserving diffeomorphisms preserving the surfaces setwise. Let $\Alg^{(4,2)}$ denote the set of triples $(\phi_1,\phi_2,\phi_3)$ of bounding homomorphisms $\phi_1,\phi_2,\phi_3 \colon \pi_1(\Sigma_g - \{p_1,\ldots,p_{2b} \}, \ast) \twoheadrightarrow F_{g+b}$ such that:
\begin{enumerate}
    \item the pushouts of pairs of the associated closed bounding homomorphisms $\overline{\phi_i}, \overline{\phi_j}$ for $i \neq j$ are all free groups, and 
    \item the pushouts of pairs of the bounding homomorphisms $\phi_i,\phi_j$ are all free groups of rank equal to the sum of the rank of the pushout of $\overline{\phi_i}$ and $\overline{\phi_j}$ plus the number of components of $(\phi_1, \phi_2)$ as in \autoref{rem:components}. 
\end{enumerate} 
In other words, an element in $\Alg^{(4,2)}$ is a \textit{group trisection} of a knotted surface group, which could also be described with the following commutative diagram, where every face is a pushout and every homomorphism is surjective (analogous to that in \autoref{sec:intro}).
\begin{center}
\begin{tikzcd}[row sep=scriptsize, column sep=scriptsize]
        & 
        \pi_1(H_g - T_{b},\ast)
        \arrow[rr, twoheadrightarrow] \arrow[dd, twoheadrightarrow] & &
        \pi_{1}(X_{1} - \mathcal{D}_{1},\ast)
        \arrow[dd, twoheadrightarrow] \\
        \pi_1(\Sigma_g - \{2b \text{ pts}\},\ast)
        \arrow[ur, twoheadrightarrow] \arrow[rr, crossing over, twoheadrightarrow] \arrow[dd, twoheadrightarrow]
        & &
        \pi_1(H_g - T_{b},\ast)
        \arrow[ur, twoheadrightarrow]
        \\
        & 
        \pi_{1}(X_{3} - \mathcal{D}_{3},\ast)
        \arrow[rr, twoheadrightarrow] & &
        \pi_{1}(X^{4} - S,\ast) \\
        \pi_1(H_g - T_{b},\ast)
        \arrow[rr, twoheadrightarrow] \arrow[ur, twoheadrightarrow]
        & &
        \pi_{1}(X_{2} - \mathcal{D}_{2},\ast)
        \arrow[from=uu, crossing over, twoheadrightarrow] \arrow[ur, twoheadrightarrow] 
\end{tikzcd}
\end{center}

We will need the following lemma in the construction of the map $(X, S) \colon \Alg^{(4,2)} \to \Man^{(4,2)}$.

\begin{lemma}[The free group characterizes unlinks]
    \label{lem:fundamental_group_detects_unlink}
    The fundamental group detects the unlink in
    the connected sum $\#^k (\sphere{1} \times \sphere{2})$.
    That is, if
    $L = L_{1} \sqcup \ldots \sqcup L_{n}
    \hookrightarrow
    \#^k (\sphere{1} \times \sphere{2})$
    is an $n$-component link with
    $\pi_{1}(\sphere{3} - L,\ast) \cong F_{n+k}$ a free
    group on $n+k$ generators,
    then $L$ is the $n$-component unlink.
\end{lemma}

\begin{proof}
    This proof is a generalization of the
    argument in \cite[Thm.~1]{hillman1981alexander}
    for unlinks in $S^3$.
    Any abelian subgroup of a free group is cyclic,
    so the meridian-longitude generators of the torus around a link component
    span an abelian group in the free
    $\pi_{1}(\sphere{1} \times \sphere{2} - L,\ast)$.
    Since the meridians generate the first homology of a link complement,
    we know that the longitudes have to be nullhomotopic in the link exterior.
    Now apply the loop theorem to obtain disjointly embedded disks
    recognizing the split unlink.
\end{proof}

\begin{remark}
There exist nontrivial links
$L \subset \#^k (\sphere{1} \times \sphere{2})$
whose complement has free fundamental group (but not free of the same rank as the group of the unlink).
For example, the core curve of one of the solid tori in the standard
genus 1 Heegaard splitting of $\sphere{1} \times \sphere{2}$ has complement
homotopy equivalent to the other solid torus; that is, the fundamental group of its complement is free on one generator.
On the other hand, the fundamental group of the complement of the unknot in $\sphere{1} \times \sphere{2}$ is free on two generators (the meridian of the unknot and
the generator of $\pi_1(\sphere{1} \times \sphere{2},\ast)$).
\end{remark}

\begin{remark}
    An analogous statement to
    \autoref{lem:fundamental_group_detects_unlink}
    is false for higher dimensional links.
    Cochran \cite{cochran1983ribbon}
    exhibited links of 2-spheres in 4-spheres whose groups are free
    (but not free on their meridians).
\end{remark}

Now we define the map $(X, S) \colon \Alg^{(4,2)} \to \Man^{(4,2)}$ as follows. Given $(\phi_1,\phi_2,\phi_3) \in \Alg^{(4,2)}$, let $H(\phi_1), H(\phi_2),$ and $H(\phi_3)$ be the handlebodies bounding $\Sigma_g$ that result from the application of \autoref{thm:main}, and further let $T(\phi_i) \subset H(\phi_i)$ for $i = 1,2,3$ be the resulting trivial tangles in these handlebodies.
We glue the handlebodies $H(\phi_1)$, $H(\phi_2)$, and $H(\phi_3)$ together along the common surface $\Sigma_g$ and obtain a closed, oriented, smooth 4-manifold $X = X(\overline{\phi_1}, \overline{\phi_2}, \overline{\phi_3})$ as in \autoref{sec:closed_4-manifolds}. (Here we have used the condition that the pairwise pushouts of the associated closed bounding homomorphisms are free groups to ensure that we obtain $\#^k (S^1 \times S^2)$ as the result of gluing two of the handlebodies together, and thus we can cap off with 4-dimensional 1-handlebodies.) The surface $S$ in $X$ is obtained by considering the unions of the tangles $T(\phi_i) \cup_{\{p_1, \ldots, p_{2b}\}} T(\phi_j) \subset H(\phi_i) \cup_{\Sigma_g} H(\phi_j)$ for $i \neq j$. Since the pushout of $\phi_i$ and $\phi_j$ is a free group of the appropriate rank, by \autoref{lem:fundamental_group_detects_unlink} we know that the unions of these tangles are all unlinks. Therefore, we can take disjoint bridge disks bounding $T(\phi_i) \cup_{\{p_1,\ldots,p_{2b}\}} T(\phi_j)$ in $H(\phi_i) \cup_{\Sigma_g} H(\phi_j)$ and push these into the 4-dimensional 1-handlebody bounding $H(\phi_i) \cup_{\Sigma_g} H(\phi_j)$ in $X$. Then the union of these three sets of disks (with one set in each 4-dimensional 1-handlebody) is the knotted surface $S$.  

We now define the analogues of the various relations from \autoref{sec:links_3-manifolds} and \autoref{sec:closed_4-manifolds} in this setting. The definitions of $\sim_h$ and $\sim_m$ are exactly analogous to the definitions in \autoref{sec:links_3-manifolds}. Just as stabilization appeared in \autoref{sec:links_3-manifolds} as several different relations -- namely, $\sim_{s_g}$ for changing the genus of the surface and $\sim_{s_b^1}, \sim_{s_b^2}$ for perturbing each of the two tangles -- in our current setting, stabilization will also manifest as several different relations.  Here we will have $\sim_{s_g}$ corresponding to changing the genus of the surface, and $\sim_{s_b^1}$, $\sim_{s_b^2}$, and $\sim_{s_b^3}$  corresponding to changing the number of strands in the tangles.

Given $(\phi_1,\phi_2, \phi_3), (\phi_1', \phi_2', \phi_3') \in \Alg^{(4,2)}$, we now define $(\phi_1, \phi_2, \phi_3) \sim_{s_g} (\phi_1', \phi_2', \phi_3')$.  We note that $\phi_i'$ will be a map from $\pi_1(\Sigma_g - \{p_1,\ldots,p_{2b} \},\ast) \twoheadrightarrow F_{g+b}$ while $\phi_i$ will be a map from $\pi_1(\Sigma_{g+1}-\{p_1,\ldots,p_{2b} \},\ast) \twoheadrightarrow F_{g+3+b}$. We say $(\phi_1, \phi_2, \phi_3) \sim_{s_g} (\phi_1', \phi_2', \phi_3')$ if $\phi_i(a_j) = \phi_i'(a_j)$, $\phi_i(b_j) = \phi_i'(b_j)$, and $\phi_i(p_k) = \phi_i'(p_k)$ for $i = 1,2,3$, $j= 1,\ldots,g$, and $k = 1,\ldots,2b$ (where we are identifying $F_{g+b}$ naturally as a subset of $F_{g+3+b}$, identifying the $h_i$ generators in $F_{g+b}$ with $h_i$ in $F_{g+3+b}$ and similarly with $t_i$), and the rest of the generators are mapped just as in the definition of $\sim_s$ in \autoref{sec:closed_4-manifolds}.  

Now we discuss the algebraic version of perturbation in this setting, which is analogous to the definitions of $\sim_{s_b^1}$, $\sim_{s_b^2}$ in \autoref{sec:links_3-manifolds}. Suppose that $b > 0$. This definition is motivated by the pictures of perturbation shown in \autoref{fig:stab2} and \autoref{fig:stab3}, where two arcs of different colors are banded together to create three new arcs, one of each color. There are three such operations to consider depending on how we cyclically permute the colors. For each operation, there are two cases under which the operation can occur: either the two arcs to be banded together share an endpoint, or they do not. 

Let $(\phi_1,\phi_2,\phi_3), (\phi_1', \phi_2',\phi_3') \in \Alg^{(4,2)}$. We note that in both cases $\phi_i'$ will be a map from $\pi_1(\Sigma_g - \{p_1,\ldots,p_{2b} \},\ast) \twoheadrightarrow F_{g+b}$ while $\phi_i$ will be a map from $\pi_1(\Sigma_g-\{p_1,\ldots,p_{2b},p_{2b+1},p_{2b+2}\},\ast) \twoheadrightarrow F_{g+b+1}$. If the arcs to be banded together share an endpoint, then assume without loss of generality that $\phi_1'(p_{2b})=\phi_2'(p_{2b})=t_b$. (We can do this because we first mod out by mapping class group elements; see the proof of \autoref{thm:Alg_3_to_Man_3}, claim 1, for details.) We write $(\phi_1,\phi_2,\phi_3) \sim_{s_b^1} (\phi_1', \phi_2',\phi_3')$ if 
  \begin{align*}
     \phi_1(p_{2b}) &= t_{b+1}\\ 
     \phi_1(p_{2b+1}) &= (t_{b+1})^{-1} \\
     \phi_1(p_{2b+2}) &=  t_b \\
     \phi_2(p_{2b}) &= t_{b+1} \\ 
     \phi_2(p_{2b+1}) &= (t_{b+1})^{-1} \\ 
     \phi_2(p_{2b+2}) &=  t_b \\ 
     \phi_3(p_{2b+1}) &= t_{b+1}\\
     \phi_3(p_{2b+2}) &= (t_{b+1})^{-1}
 \end{align*}
and $\phi_i$ and $\phi_i'$ agree for all other elements in the generating sets (suitably identifying the groups) for $i = 1,2,3$. See \autoref{fig:stab2}. We similarly define $\sim_{s_b^2}$ and $\sim_{s_b^3}$ in this case by swapping the roles of the indices.

If the arcs to be banded together do not share an endpoint, then assume without loss of generality that $\phi_1'(p_{2b-1})=t_b$ and $\phi_2'(p_{2b})=t_b$. We write $(\phi_1,\phi_2,\phi_3) \sim_{s_b^1} (\phi_1', \phi_2',\phi_3')$ if 
  \begin{align*}
     \phi_1(p_{2b-1}) &= t_{b+1}\\ 
     \phi_1(p_{2b+1}) &= (t_{b+1})^{-1} \\
     \phi_1(p_{2b+2}) &=  t_b \\
     \phi_2(p_{2b}) &= t_{b+1} \\ 
     \phi_2(p_{2b+1}) &= (t_{b+1})^{-1} \\ 
     \phi_2(p_{2b+2}) &=  t_b \\ 
     \phi_3(p_{2b+1}) &= t_{b+1}\\
     \phi_3(p_{2b+2}) &= (t_{b+1})^{-1}
 \end{align*}
and $\phi_i$ and $\phi_i'$ agree for all other elements in the generating sets (suitably identifying the groups) for $i = 1,2,3$. See \autoref{fig:stab3}. We similarly define $\sim_{s_b^2}$ and $\sim_{s_b^3}$ in this case by swapping the roles of the indices.

\begin{figure}[h]
    \centering
    \includegraphics[width=.9\linewidth]{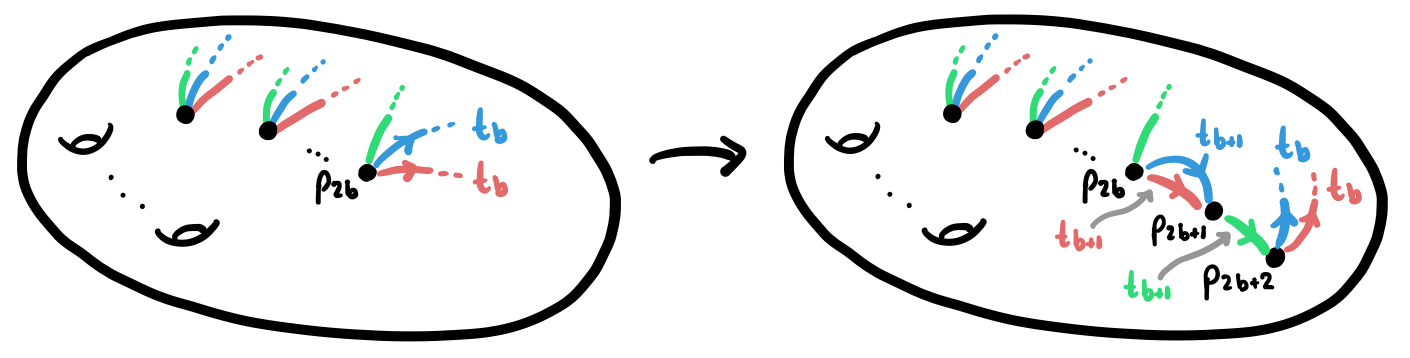}
    \caption{Performing a perturbation, where two arcs of different colors with the same endpoint are banded together. This creates three new tangle strands, one of each color, and increases the number of punctures by two. If $\phi_1$ corresponds to pink, $\phi_2$ corresponds to blue, and $\phi_3$ corresponds to green, then this is a picture of the $\sim_{s_b^1}$ version of perturbation. Cyclically permuting the colors gives a picture for the $\sim_{s_b^2}$ and  $\sim_{s_b^3}$ version. 
    \label{fig:stab2}}
\end{figure}

\begin{figure}[h]
    \centering
    \includegraphics[width=.9\linewidth]{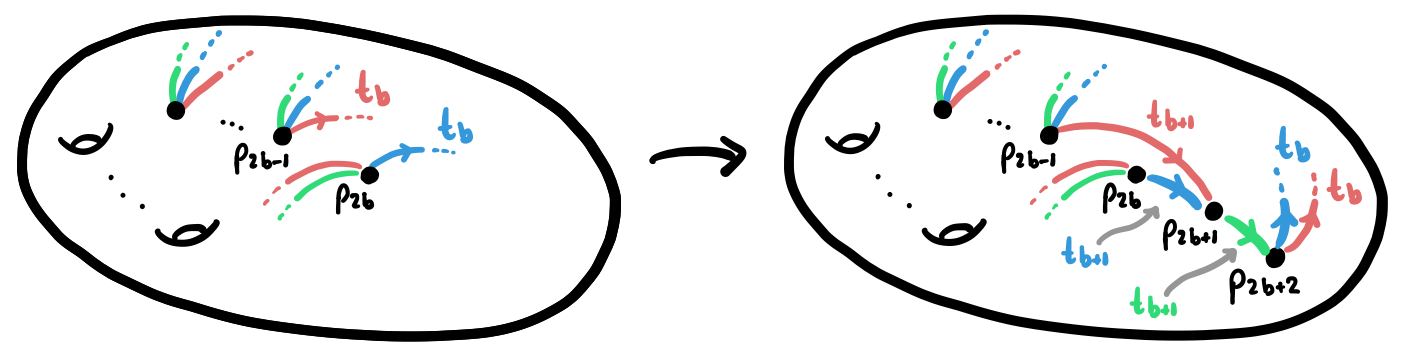}
    \caption{Performing a perturbation, where two arcs of different colors with different endpoints are banded together. This creates three new tangle strands, one of each color, and increases the number of punctures by two. If $\phi_1$ corresponds to pink, $\phi_2$ corresponds to blue, and $\phi_3$ corresponds to green, then this is a picture of the $\sim_{s_b^1}$ version of perturbation. Cyclically permuting the colors gives a picture for the $\sim_{s_b^2}$ and  $\sim_{s_b^3}$ version. 
    \label{fig:stab3}}
\end{figure}

Finally, we must again include a relation corresponding to cyclically permuting the ``colors'' of the curves and arcs on a trisection diagram (that is, cyclically permuting the roles of the three curve-and-arc systems). We say  $(\phi_1, \phi_2, \phi_3) \sim_c (\phi_1', \phi_2', \phi_3')$ if $\phi_1=\phi'_2$, $\phi_2=\phi'_3$ , and $\phi_3=\phi'_1$. Then let $\sim$ denote the equivalence relation on $\Alg^{(4,2)}$ generated by $\sim_h, \sim_m, \sim_{s_g}, \sim_{s_b^1}, \sim_{s_b^2},$ $\sim_{s_b^3}$, and $\sim_c$.  

\begin{theorem}
    \label{thm:Alg_4_2_to_Man_4_2}
    The map
    $(X, S) \colon \Alg^{(4,2)} 
    \to \Man^{(4,2)}$
    descends to $\Alg^{(4,2)}/ \sim$
    and the resulting map is a bijection.  
\end{theorem}

\begin{proof}
As in the proofs of \autoref{thm:Alg_3_to_Man_3}, \autoref{thm:Alg_3_1_to_Man_3_1}, and \autoref{thm:Alg_4_to_Man_4}, we consider an intermediate set $\texttt{Diag}^{(4,2)}$ whose elements are \emph{trisection diagrams}, that is, tuples $(\Sigma_g, \alpha, \beta, \gamma, S_{\alpha}, S_{\beta}, S_{\gamma})$
where $\alpha,\beta,\gamma$ are cut systems on $\Sigma_g$ and $S_{\alpha}, S_{\beta}, S_{\gamma}$ are shadow diagrams for trivial tangles with endpoints $\{p_1, \ldots, p_{2b}\}$ (which are all only considered up to isotopy). In other words, $(\alpha, S_{\alpha})$ is a curve-and-arc system for one tangle and handlebody, and similarly for the other two pairs.
(Refer to the beginning of \autoref{sec:setup} for the definition of \textit{curve-and-arc system}.) Then the map $(X,S)$ factors as shown below.
    \begin{equation*}
    \begin{tikzcd}
        \Alg^{(4,2)}
        \arrow[rr, "{(X,S)}"] \arrow[dr, "D"] 
        & & \Man^{(4,2)}  \\
        & \texttt{Diag}^{(4,2)} \arrow[ur, "R"] &
    \end{tikzcd}
    \end{equation*}
The map $R \colon \Diag^{(4,2)} \to \Man^{(4,2)}$ is the topological realization of a diagram $(\Sigma_g, \alpha, \beta, \gamma, S_{\alpha}, S_{\beta}, S_{\gamma})$, where we cross $\Sigma_g$ with a disk, glue disks on the respective sides to $\alpha$, $\beta$, and $\gamma$, glue 3-balls to the resulting sphere boundary components to obtain three handlebodies, and then push the interiors of the shadow arcs in $S_{\alpha}$, $S_{\beta}$, and $S_{\gamma}$ into their respective handlebody to obtain three tangles which are pairwise unlinks. Since the pairwise unions of the handlebodies are diffeomorphic to $\#^k(S^1 \times S^2)$, we can fill these in uniquely with 4-dimensional 1-handlebodies. Then cap off the unlinks with disks in the pairwise unions of the handlebodies, and then push these disks into the 4-dimensional 1-handlebodies bounded by these unions to create a knotted surface.
The map $D \colon \texttt{Alg}^{(4,2)} \to \texttt{Diag}^{(4,2)}$ is the construction of $(X,S)$
using \autoref{thm:main}, but where we stop at just a diagram with curves $\alpha$ and arcs $S_{\alpha}$ corresponding to $\phi_1$, curves $\beta$ and arcs $S_{\beta}$ corresponding to $\phi_2$, and curves $\gamma$ and arcs $S_{\gamma}$ corresponding to $\phi_3$.

Our topological input theorem can now be translated into the following diagrammatic statement: two diagrams of the same knotted surface in a $4$-manifold are related by a sequence of isotopies, slides, perturbations, and depurtubations, as in the proof of \autoref{thm:Alg_3_1_to_Man_3_1}, and additionally, cyclically permuting the colors. In other words, these are diagrammatic equivalence relations which have the algebraic counterparts $\sim_h, \sim_m, \sim_{s_g}, \sim_{s_b^1}, \sim_{s_b^2},$ $\sim_{s_b^3}$, and $\sim_c$ as described above.

The rest of the proof then follows in similar fashion as before; mod out the map $D$ by these diagrammatic equivalence relations, and show each time that the map factors through and a bijection between quotients is achieved. Then by our topological input theorem from this section, along with that from \autoref{sec:closed_4-manifolds}, we achieve the result.
\end{proof}

One consequence of \autoref{thm:Alg_4_2_to_Man_4_2} is the following corollary. By the Gordon-Luecke theorem, (one-dimensional) knots in $S^3$ are determined by the oriented homeomorphism type of their complements \cite{gordonluecke1989}, but the same is not true for knotted surfaces in $S^4$; see for instance \cite{gordon1976knots,suciu1985ribbon,kaneobu1994peripheral}. However, the extra information contained in a group trisection \emph{is} enough to distinguish knotted surfaces.

\begin{corollary} \label{cor:distinguish}
Although smoothly knotted surfaces in the $4$-sphere cannot always be distinguished by their fundamental groups (or even their complements), they \emph{can} always be distinguished by the group trisections of their fundamental groups.
\end{corollary}

\section{Examples and consequences}
\label{sec:examples}

In this section we discuss some examples and consequences of our results, including examples of non-equivalent group trisections of the same group, an algebraic version of the smooth unknotting conjecture, a group-theoretic characterization of knot groups, and musings on algorithmic decidability. We will frequently make use of the following definition.

\begin{definition}[Fundamental group of pairs or triples of maps] \label{def:fundamental} 
The \textit{fundamental group} of $(\phi_1,\phi_2)$ (in $\Alg^{3}$ or $\Alg^{(3,1)}$) is their pushout. Similarly, the \textit{fundamental group} of $(\phi_1,\phi_2,\phi_3)$ (in $\Alg^{4}$ or $\Alg^{(4,2)}$) is the group that results from pushing out the three homomorphisms. 
\end{definition}

Alternatively, the fundamental group of $(\phi_1,\phi_2)$ or $(\phi_1,\phi_2,\phi_3)$ is the fundamental group of the corresponding topological space constructed from these maps.

\subsection{Non-equivalent group trisections of the same group}

Here we present a few examples of non-equivalent group trisections of the same group. Specifically, we provide non-equivalent triples $(\phi_1,\phi_2,\phi_3)$ in $\Alg^{(4,2)}/\sim$ with the same fundamental group, where the non-equivalence follows from known topological results about the knotted surfaces realizing these groups as their fundamental group. We leave many details of the calculations to the reader, but see \cite{blackwell2022combinatorial} and \cite{ruppik2022casson} for a more thorough treatment of the following examples, as well as \cite[Sec.~4.1]{joseph2022bridge} for a description of how to calculate fundamental groups from tri-plane diagrams. 

Note that in the case of closed $4$-manifolds, there are as many (non-equivalent) group trisections of the trivial group as there are exotic smooth structures on simply connected $4$-manifolds. As a given $4$-manifold with an embedded surface can admit non-equivalent group trisections of the same group, which correspond to different surfaces, the theory of group trisections for ($4$-manifold, knotted surface) pairs is even richer than that for $4$-manifolds alone.

\subsubsection{Unknotted \texorpdfstring{$\mathbb{RP}^2$}{RP2}s in \texorpdfstring{$S^4$}{S4}} \label{sec:unknotted_RP2s}
For our first example of non-equivalent group trisections of knotted surface groups, consider the following.

\begin{corollary}
    There exist two elements of $\Alg^{(4,2)}$ which are not equivalent in the quotient $\Alg^{(4,2)} / \sim$, but both push out to $\ZZ / 2\ZZ$. Under topological realization these elements correspond to unknotted $\mathbb{RP}^2$s in $S^4$.
\end{corollary}

\begin{proof}
As shown in \cite[Fig.~15]{meier2017bridgeS4}, we can represent two unknotted $\mathbb{RP}^2$s in $S^4$, with Euler numbers $\pm 2$, by the tri-plane diagrams in \autoref{fig:triplane1} and \autoref{fig:triplane2}. Both $\mathbb{RP}^2$s produce group trisections of $\ZZ / 2\ZZ$, but by \autoref{thm:Alg_4_2_to_Man_4_2} these group trisections cannot be equivalent as the surfaces are not isotopic. Here we include some details of the construction of these group trisections for the sake of illustration, but see \cite[Sec.~4.2.3]{blackwell2022combinatorial} for a full treatment, including an example of the use of \autoref{thm:main} to recover the $\gamma$ (green) tangle from the group trisection.

\begin{figure}[h]
    \centering
    \includegraphics[width=.6\linewidth]{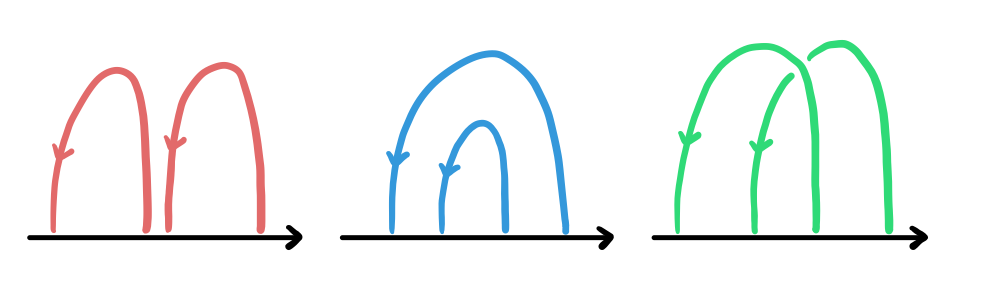}
    \caption{A tri-plane diagram for the unknotted $\mathbb{RP}^2$ in $S^4$ with Euler number~$-2$.
    \label{fig:triplane1}}
\end{figure}

\begin{figure}[h]
    \centering
    \includegraphics[width=.6\linewidth]{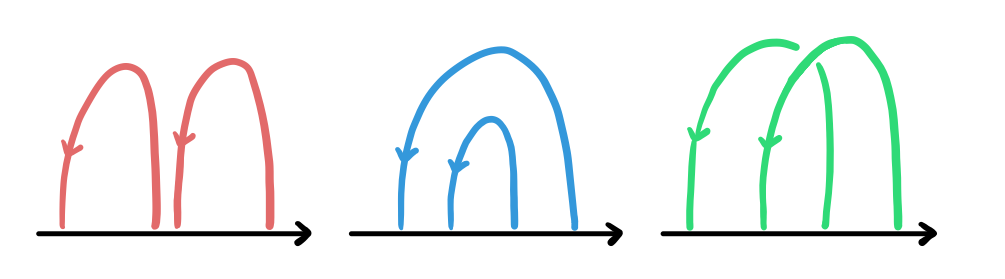}
    \caption{A tri-plane diagram for the unknotted $\mathbb{RP}^2$ in $S^4$ with Euler number~$+2$.
    \label{fig:triplane2}}
\end{figure}

As $\mathbb{RP}^2$ is non-orientable, it is not possible to consistently orient the tri-plane diagrams, but we choose arbitrary orientations for each tangle separately in order to write down the group trisection maps; the choice of orientations here will not matter. For both $\mathbb{RP}^2$s, presentations for the groups making up the initial three epimorphisms of the group trisection are as follows, where (abusing notation) $T_{2}^{\alpha}$, $T_{2}^{\beta}$, and $T_{2}^{\gamma}$ are the trivial tangles as shown in \autoref{fig:triplane1} and \autoref{fig:triplane2}.
\[ 
\begin{array}{rcl}
\pi_1(\Sigma_{0} - \{ p_{1}, \ldots, p_{4} \},\ast) & = & \langle p_1,p_2,p_3,p_4 \mid p_1p_2p_3p_4=1 \rangle \\
\pi_1(H_0 - T_{2}^{\alpha},\ast) & = & \langle x_1, x_2 \rangle \\
\pi_1(H_0 - T_{2}^{\beta},\ast) & = & \langle y_1, y_2 \rangle \\
\pi_1(H_0 - T_{2}^{\gamma},\ast) & = & \langle z_1, z_2 \rangle
\end{array}
\]

Below we write down the initial three maps (in other words, the elements in $\Alg^{(4,2)}$) for the $\mathbb{RP}^2$ with Euler number $-2$, corresponding to the $\alpha$ tangle (left/red), $\beta$ tangle (center/blue), and $\gamma$ tangle (right/green).
\begin{align*}
    p_1 &\mapsto x_1 & p_1 &\mapsto y_1 & p_1 &\mapsto z_1 \\
    p_2 &\mapsto x_1^{-1} & p_2 &\mapsto y_2 & p_2 &\mapsto z_2 \\
    p_3 &\mapsto x_2 & p_3 &\mapsto y_2^{-1} & p_3 &\mapsto z_1^{-1} \\
    p_4 &\mapsto x_2^{-1} & p_4 &\mapsto y_1^{-1} & p_4 &\mapsto z_1z_2^{-1}z_1^{-1}
\end{align*}

Below we write down the initial three maps (in other words, the elements in $\Alg^{(4,2)}$) for the $\mathbb{RP}^2$ with Euler number $+2$, corresponding to the $\alpha$ tangle (left/red), $\beta$ tangle (center/blue), and $\gamma$ tangle (right/green).
\begin{align*}
    p_1 &\mapsto x_1 & p_1 &\mapsto y_1 & p_1 &\mapsto z_1 \\
    p_2 &\mapsto x_1^{-1} & p_2 &\mapsto y_2 & p_2 &\mapsto z_2 \\
    p_3 &\mapsto x_2 & p_3 &\mapsto y_2^{-1} & p_3 &\mapsto z_2z_1^{-1}z_2^{-1} \\
    p_4 &\mapsto x_2^{-1} & p_4 &\mapsto y_1^{-1} & p_4 &\mapsto z_2^{-1}
\end{align*}

The only difference between these maps and the previous is a slight change in the map for the $\gamma$ tangle, corresponding to the crossing change between the two tangles. As the initial maps determine the entire pushout cube, it is sufficient to provide these maps; to recover the other maps in the cube, push out repeatedly until the entire cube is formed.
\end{proof}

\subsubsection{Twist-spun torus knots in \texorpdfstring{$S^4$}{S4}}
\label{sec:twist_spun_torus_knots}

A well-known family of twist-spun torus knots gives the following corollary of a combination of \cite{gordon1973homotopy} and \autoref{thm:Alg_4_2_to_Man_4_2}.

\begin{corollary}
    \label{prop:same_group_different_trisections}
    There exists a collection of three elements of $\Alg^{(4,2)}$ which are not equivalent in the quotient $\Alg^{(4,2)} / \sim$, but push out to the same group. Under topological realization these elements correspond to knotted spheres in $\sphere{4}$ which have isomorphic knotted surface groups. 
\end{corollary}

\begin{proof}
Under topological realization, these elements of $\Alg^{(4,2)}$ correspond to the knotted spheres $\sphere{2} \hookrightarrow \sphere{4}$ listed below.
	\begin{itemize}
		\item $\Twist^2 T_{3,5}$, the 2-twist spin of the $(3, 5)$-torus knot
		\item $\Twist^3 T_{2,5}$, the 3-twist spin of the $(2, 5)$-torus knot
		\item $\Twist^5 T_{2,3}$, the 5-twist spin of the $(2, 3)$-torus knot
	\end{itemize}
Presentations for the corresponding knotted surface groups are given by, for instance,
\begin{equation*}
	\pi_{1}(\sphere{4} -  \Twist^2 T_{3,5}, \ast)
	\cong
	\langle x, y \mid x^3 = y^5, [a^2, x] = 1, [a^2, y] = 1 \rangle,
\end{equation*}
where $a = a(x, y)$ is a meridian of $T_{3,5}$ expressed as a word in the non-meridional generators $x,y$ coming from the genus $1$ Heegaard decomposition of the 3-sphere.
Recall that this Heegaard splitting leads to the presentation $\pi_{1}(\sphere{3} -  T_{3,5}, \ast) \cong \langle x, y \mid x^3 = y^5 \rangle$ of the group of the torus knot complement.
If we write the first homology of the complement multiplicatively generated by $H_{1}(\sphere{4} -  \Twist^2 T_{3,5}) \cong \langle t \rangle$, then under abelianization the generators of the knot group map to $x \mapsto t^5, y \mapsto t^3$, and we pick our orientations so that the meridian maps to $a \mapsto t$.

To write down the presentations for the other knot groups, we will abuse notation and reuse the letters $x, y$ for the generators of the torus knot complement and $a$ for a choice of meridian.
With this,
\begin{equation*}
	\pi_{1}( \sphere{4} - \Twist^3 T_{2,5}, \ast)
	\cong
	\langle x, y \mid x^2 = y^5, [a^3, x] = 1, [a^3, y] = 1 \rangle,
\end{equation*}
\begin{equation*}
	\pi_{1}( \sphere{4} - \Twist^5 T_{2,3}, \ast)
	\cong
	\langle x, y \mid x^2 = y^3, [a^5, x] = 1, [a^5, y] = 1 \rangle.
\end{equation*}
However all three of these groups $\pi_{1}(\sphere{4} - \Twist^{i} T_{j, k}, \ast)$ are abstractly isomorphic to a direct product $\ZZ \times \Dod^{*}$ of the integers with the binary dodecahedral group $\mathrm{Dod}^{*}$; see \cite{zeeman1965twisting}.

See \cite[Sec.~16]{ruppik2022casson} for explicit constructions of group trisections of these groups which come from the tri-plane diagrams described in \cite[Fig.~20]{meier2017bridgeS4}. These group trisections are not equivalent, which follows from \autoref{thm:Alg_4_2_to_Man_4_2} together with Gordon's observation \cite{gordon1973homotopy} that $\Twist^2 T_{3,5}$, $\Twist^3 T_{2,5}$, and $\Twist^5 T_{2,3}$ are all distinct, non-isotopic 2-knots. Gordon shows that the minimal exponent of a meridian's power which is central in the groups $\pi_{1}(\sphere{4} - \Twist^{i} T_{j, k}, \ast)$ has to divide the twisting parameter. Pick a generator $p$ of the punctured sphere group and follow it through the group trisection cube; we call its image in the final knotted sphere group $p$ as well. Observe that the image of $p$ is a meridional generator, and thus the smallest power of $p$ that lands these elements in the center is an invariant distinguishing the group, and consequently the group trisections. 
\end{proof}

\begin{remark}
	\autoref{prop:same_group_different_trisections} generalizes; take coprime integers $p, q, r \ge 2$ and consider the following knotted 2-spheres.
	\begin{itemize}
		\item $\Twist^{p} T_{q, r}$, the $p$-twist spin of the $(q, r)$-torus knot
		\item $\Twist^{q} T_{p, r}$, the $q$-twist spin of the $(p, r)$-torus knot
		\item $\Twist^{r} T_{p, q}$, the $r$-twist spin of the $(p, q)$-torus knot
	\end{itemize}
	Since $p, q, r$ are pairwise coprime, all three of these knotted spheres $\Twist^{i} T_{j, k} \colon \sphere{2} \hookrightarrow \sphere{4}$ for $\{ i, j, k \} = \{ p, q, r \}$ have the same fundamental group for their complement.
	By Zeeman \cite{zeeman1965twisting}, all of them are fibered by the punctured bounded Brieskorn sphere $\Sigma(p, q, r)$, but these fibrations have monodromies with different periods.
	If the fiber is a homology sphere, one can show that the resulting groups of the 2-knots are abstractly isomorphic.
	The example for $(p, q, r) = (2, 3, 5)$ is also discussed by Boyle \cite{boyle1988handles}, where he starts with Zeeman's observation that these twist spun 2-knots are fibered by a punctured Poincar{\'e} homology sphere $\Sigma(2, 3, 5)$.
	Even though they share the same group, these are non-isotopic 2-knots and one way to distinguish them is to look at the smallest power of a meridian which lies in the center of the knot group.
	This family of knots is further discussed in \cite{gordon1973homotopy}.
\end{remark}

\begin{remark}
    In \cite{suciu1985ribbon} Suciu constructs an infinite family of ribbon 2-knots $R_{i} \colon \sphere{2} \hookrightarrow \sphere{4}$ in the 4-sphere, each of which has $\pi_1(\sphere{4} - R_{i},\ast)$ isomorphic to the trefoil group.
    The 2-knots are pairwise non-isotopic, and can be distinguished by the $\ZZ \pi_1$-module structure of $\pi_2$ of their complements.
    By bridge trisecting these examples, one can construct an infinite family of elements in $\Alg^{(4,2)}$ satisfying the same properties as in \autoref{prop:same_group_different_trisections}. See  \cite[Sec.~17]{ruppik2022casson}.
\end{remark}

\subsection{Algebraic version of the smooth unknotting conjecture}

The smooth unknotting conjecture posits that an embedded sphere in $\sphere{4}$ is smoothly unknotted if and only if the group of its complement is infinite cyclic.
Here we use our correspondence in \autoref{thm:Alg_4_2_to_Man_4_2} to give an algebraic statement equivalent to the unknotting conjecture.  This is in the spirit of a result in \cite{abrams2018group} which gives a similar group-theoretic conjecture that is equivalent to the smooth 4-dimensional Poincar\'e conjecture (following the tradition of \cite{stallings1966nottoprove}).  

Note that given $(\phi_1,\phi_2,\phi_3) \in \Alg^{(4,2)}$, we can tell just from the maps $(\phi_1,\phi_2,\phi_3)$ whether the resulting surface is connected. Similarly, we can determine the Euler characteristic of the resulting surface directly from $(\phi_1,\phi_2,\phi_3)$. We say $(\phi_1,\phi_2,\phi_3)$ is \textit{spherical} if its corresponding surface is a sphere. Just as mentioned in \cite{abrams2018group}, it follows from \autoref{thm:Alg_4_to_Man_4} that there is a purely group-theoretic condition on whether or not $X(\overline{\phi_1},\overline{\phi_2},\overline{\phi_3})$, the $4$-manifold built from the associated closed bounding homomorphisms $\overline{\phi_i}$, is orientation-preserving diffeomorphic to the 4-sphere with the standard orientation.  We say in this case that $(\phi_1,\phi_2,\phi_3)$ \emph{represents $S^4$}. Then from \autoref{thm:Alg_4_2_to_Man_4_2}, it follows that the group-theoretic conjecture below is equivalent to the smooth unknotting conjecture.  

\begin{conjecture}
For every spherical $(\phi_1,\phi_2,\phi_3) \in \Alg^{(4,2)}$ which represents $S^4$, we have $(\phi_1,\phi_2,\phi_3)\\ \sim (s_1,s_2,s_3)$,  where $(s_1,s_2,s_3)$ is the group trisection given by maps \[s_1,s_2,s_3 \colon \pi_1(S^2 - \{p_1, p_2\}, \ast) \twoheadrightarrow \mathbb{Z},\] with $s_1 = s_2 = s_3$ and $s_1(p_1) = 1$ (this corresponds to the unknotted sphere in $S^4$). 
\end{conjecture}

\subsection{A group-theoretic characterization of knot groups}
All knots discussed here are smooth. This subsection is motivated by the question: Exactly which groups arise as fundamental groups of codimension-2 knot complements in $S^n$? An algebraic characterization of such groups was given by Kervaire; namely, such groups are exactly the finitely-presentable groups $G$ with $H_1(G) = \mathbb{Z}$, $H_2(G) = 0$, and such that $G$ can be normally generated by a single element \cite{michel2014michel}. When $n = 3$, Artin gave a characterization for fundamental groups of knot complements in terms of group presentations \cite{artin1925theorie}.  When $n=4$, Kamada \cite{kamada1994characterization} and Gonz\'{a}lez-Acu\~{n}a \cite{gonzalez1994} have given similar characterizations also in terms of group presentations.

The methods of this paper give alternative group-theoretic characterizations of fundamental groups of knot complements in $S^n$ for $n=3,4$, in a unified fashion.  We say $(\phi_1,\phi_2) \in \Alg^{(3,1)}$ is \emph{connected} if the resulting link is connected; note that this could instead be phrased entirely in terms of the maps $\phi_1$ and $\phi_2$. We say $(\phi_1,\phi_2)$ \emph{represents $S^3$} if the resulting 3-manifold is diffeomorphic to the 3-sphere; note that this too can be phrased in an entirely algebraic way due to \autoref{thm:Alg_3_to_Man_3}. (See the previous subsection for the corresponding definition when $n=4$.) Then directly from \autoref{thm:Alg_3_1_to_Man_3_1} and \autoref{thm:Alg_4_2_to_Man_4_2} we have the following characterizations.

\begin{corollary}
A group $G$ is the fundamental group of a knot complement in $S^3$ if and only if there exists some connected $(\phi_1,\phi_2) \in \Alg^{(3,1)}$ such that $(\phi_1,\phi_2)$ represents $S^3$, and the fundamental group of $(\phi_1,\phi_2)$ (recall \autoref{def:fundamental}) is $G$.  

Similarly, a group $G$ is the fundamental group of a knot complement in $S^4$ if and only if there exists some spherical $(\phi_1,\phi_2,\phi_3) \in \Alg^{(4,2)}$ such that $(\phi_1,\phi_2, \phi_3)$ represents $S^4$, and the fundamental group of $(\phi_1,\phi_2,\phi_3)$ (recall \autoref{def:fundamental}) is $G$.
\end{corollary}

\subsection{Algorithmic decidability}

Recognizing whether a given finitely presented group is 
the group of the complement of a knotted surface
in the 4-sphere
is not algorithmic;
this is the case
$\mathcal{A} = \mathcal{G}, \mathcal{B} = \mathcal{K}_2$
in \cite[Thm.~1.1]{gordon2010unsolvable}.
It is undecidable
whether a given finite presentation describes
the fundamental group of the complement of a codimension-2 knot in $S^n$ for $n \ge 3$ \cite{gordon1995embedding}. The following then is a corollary of a combination of
\cite[Thm.~1.1]{gordon2010unsolvable} and \autoref{thm:Alg_4_2_to_Man_4_2}.

\begin{corollary}
Given a group $G$, it is an undecidable problem to find a group trisection $(\phi_1,\phi_2,\phi_3)$ in $\Alg^{(4,2)}$ representing $S^4$ with $b>0$ and fundamental group $G$ (recall \autoref{def:fundamental}),
or to show that none exists
(because this would decide whether $G$ is a knotted surface group).
\end{corollary} 

Recall that there \textit{is} an algorithm for deciding whether the pushouts appearing in group trisections are free groups (see \autoref{rem:algorithmic}).

\begin{question}
    There are many undecidability results in group theory (see for example \cite{miller1992decision}).  Can the above techniques be used to show the existence or non-existence of an algorithm to recognize if a closed smooth 4-manifold is diffeomorphic to $S^4$?
    Can the above techniques be used to show the existence or non-existence of an algorithm to recognize if a smooth embedding of a $2$-sphere in $S^4$ is unknotted?  In a different direction, can our framework be useful for determining if there is a polynomial-time algorithm for deciding whether a knot in $S^3$ is the unknot?
\end{question}

\begin{question}
    Are there extensions of the bijections between sets of the form $\Alg^n / \sim$ and sets of the form $\Man^n$ (as in \autoref{sec:3-manifolds} and \autoref{sec:4-manifolds}) for $n >4$?  
    A negative answer says that somehow the free groups and surface groups are responsible for the unique character of low-dimensional topology. 
\end{question}

\printbibliography

\end{document}